\documentclass[12pt]{article}
\usepackage[margin=1in]{geometry}
\usepackage{amsmath,amsfonts,amssymb,amsthm,amscd,mathrsfs}

\usepackage[all]{xy}

\DeclareMathOperator{\re}{Re}

\newcommand{\cp}[1]{^{\circ{#1}}}

\newcommand{\conj}[1]{\overline{#1}}
\newcommand{\der}[2]{\dfrac{\partial{#1}}{\partial{#2}}}

\newcommand{\rest}[1]{{\big\arrowvert_{#1}}}

\newcommand{\mc}{\mathcal}
\newcommand{\mb}{\mathbb}

\newcommand{\mf}{\mathfrak}
\newcommand{\ms}{\mathscr} 		

\theoremstyle{plain}
\newtheorem{theorem}{Theorem}[section]
\newtheorem{corollary}[theorem]{Corollary}
\newtheorem{lemma}[theorem]{Lemma}

\newtheorem{proposition}[theorem]{Proposition}

\newtheorem{convention}[theorem]{Convention}
\newtheorem{assumption}[theorem]{Assumption}
\newtheorem{condition}[theorem]{Condition}

\theoremstyle{definition}
\newtheorem{definition}[theorem]{Definition}

\newtheorem*{notation_nonumber}{Notation}

\newtheorem*{standingAssumption}{Standing Assumption}

\theoremstyle{remark}
\newtheorem{remark}[theorem]{Remark}
\newtheorem{observation}[theorem]{Observation}
\newtheorem{reduction}[theorem]{Reduction}

\numberwithin{equation}{section}

\usepackage{eso-pic}
\usepackage{graphicx}
\usepackage{color}
\usepackage{type1cm}
\definecolor{red}{rgb}{1.,0.,0.}
\definecolor{darkred}{cmyk}{0.,1.,1.,.3}
\definecolor{cyan}{cmyk}{1.,0.,0.,.5}
\definecolor{green}{cmyk}{1.,0.,1.,.5}
\definecolor{lightgreen}{cmyk}{.5,0.,.5,.0}
\definecolor{blue}{cmyk}{1.,1.,0.,0.}
\definecolor{lightblue}{cmyk}{.5,.5,0.,0.}
\definecolor{LIGHTBLUE}{cmyk}{.5,.5,0.,0.}
\definecolor{paleblue}{cmyk}{.2,.2,0.,0.}
\definecolor{paleyellow}{cmyk}{0.,0.,.2,0.}
\definecolor{red}{cmyk}{1.,0.,1.,0.}

\usepackage{fancyhdr}
\renewcommand{\mb}{\mathbb}

\newcommand{\sectionDone}[1]{\section{#1}}
\newcommand{\subsectionDone}[1]{\subsection{#1}}

\newcommand{\name}[1]{
	\label{#1}
}

\newcommand{\personalNote}[1]{
}

\newcommand{\inlinefrac}[2]{{\dfrac{#1}{#2}}}
\newcommand{\inlinefracp}[2]{{\dfrac{#1}{#2}}}
\newcommand{\shortInlineFrac}[2]{#1 / #2}
\newcommand{\inlineDilaFrac}[2]{{\dfrac{\partial #1}{\partial \conj{#2}}\Big/\dfrac{\partial #1}{\partial #2}}}

\newcommand{\disk}{\mathbb{D}}
\newcommand{\tube}{\mc{T}^{\boldsymbol{\circ}}}         
\newcommand{\tubeCore}{\mc{T}^{\hspace{1pt}\displaystyle{\boldsymbol{\cdot}}}} 
\newcommand{\tubeFilled}{\mc{T}^{\boldsymbol{\bullet}}}            

\newcommand{\tubeX}{Q}
\newcommand{\tubeXHole}{\Omega}
\newcommand{\tubeY}[1]{A_{#1}\setminus \overline{B}_{#1}}

\newcommand{\tubeYCore}{B}
\newcommand{\tubeYFilled}{A}

\newcommand{\degq}{{d'}}
\newcommand{\expDSum}[1]{\sigma_{#1}}
\newcommand{\critLoc}{\ms{C}}
\newcommand{\delHO}{\boldsymbol{\delta}}
\newcommand{\rHubOb}{\mathbf{r}}
\newcommand{\pmnHole}{\mf{r}}

\newcommand{\nax}[1]{\underline{#1}} 
\newcommand{\Nax}[1]{\hat{#1}} 

\newcommand{\gpl}[1]{\mc{G}^{\scriptscriptstyle+}_{#1}}
\newcommand{\gmn}[1]{\mc{G}^{\scriptscriptstyle-}_{#1}}

\newcommand{\conjMap}{\mf{c}}

\newcommand{\head}[1]{}
\newcommand{\smallhead}[1]{}

\newcommand{\pr}{\operatorname{pr}}
\newcommand{\crmap}{\underset{\scriptscriptstyle{\times}}{\to}}

\newcommand{\henI}{\cite{hubbard_oberste-vorth:henon1}}  
\newcommand{\henII}{\cite{hubbard_oberste-vorth:henon2}}  

\newcommand{\vt}[1]{\begin{pmatrix} #1 \end{pmatrix}}
\newcommand{\plstyle}[2]{{#1}^{\scriptscriptstyle+}_{#2}}
\newcommand{\mnstyle}[2]{{#1}^{\scriptscriptstyle-}_{#2}}
\newcommand{\plstyleAlt}[2]{{#1}_{#2,\scriptscriptstyle+}}
\newcommand{\mnstyleAlt}[2]{{#1}_{#2,\scriptscriptstyle-}}
\newcommand{\ppl}[1]{\plstyleAlt{\varphi}{#1}}
\newcommand{\pmn}[1]{\mnstyleAlt{\varphi}{#1}}
\newcommand{\pplPlain}{{\varphi}_{\scriptscriptstyle+}}
\newcommand{\pmnPlain}{{\varphi}_{\scriptscriptstyle-}}
\newcommand{\jpl}[1]{\plstyle{J}{#1}}
\newcommand{\jmn}[1]{\mnstyle{J}{#1}}
\newcommand{\kpl}[1]{\plstyle{K}{#1}}
\newcommand{\kmn}[1]{\mnstyle{K}{#1}}

\newcommand{\spl}[1]{\plstyle{s}{#1}}
\newcommand{\smn}[1]{\mnstyle{s}{#1}}
\newcommand{\hpl}{\plstyle{h}{}}
\newcommand{\hmn}{\mnstyle{h}{}}
\newcommand{\Fpl}[1]{\plstyle{\mc{F}}{#1}}
\newcommand{\Fmn}[1]{\mnstyle{\mc{F}}{#1}}
\newcommand{\lfpl}[1]{\plstyle{\mc{L}}{#1}}
\newcommand{\lfmn}[1]{\mnstyle{\mc{L}}{#1}}
\newcommand{\leaf}{\mc{L}}
\newcommand{\upl}[1]{\plstyle{U}{#1}}
\newcommand{\umn}[1]{\mnstyle{U}{#1}}

\newcommand{\bigUPl}{\tilde{U}^+}
\newcommand{\bigUMn}{\tilde{U}^-}

\newcommand{\bigDiskWithHoles}[3]{f_{#2}^{-1}\bigl(\Delta_{#1,#2}(r)\bigr)\setminus\bigcup_{#3\in p^{-1}(#1)}\Delta_{#3,#2}(r)}

\newcommand{\De}{{\Delta}}
\newcommand{\de}{{\delta}}
\newcommand{\la}{{\lambda}}
\newcommand{\si}{{\sigma}}

\newcommand{\D}{{\Bbb D}}
\newcommand{\R}{{\Bbb R}}
\newcommand{\T}{{\Bbb T}}
\newcommand{\Z}{{\Bbb Z}}

\newcommand{\ra}{\rightarrow}
\newcommand{\tl}{\tilde}

\def\ssk{\smallskip}
\def\msk{\medskip}

\def\nin{\noindent}

\newcommand{\comm}[1]{}

\newcommand{\QED}{\rlap{$\sqcup$}$\sqcap$\smallskip}

\newcommand{\ed}{\mathbf{d}}

\begin{document}

\title{The Critical Locus and Rigidity of Foliations of Complex H\'enon Maps}
\author{Misha Lyubich and  John W. Robertson}
\maketitle

\begin{abstract}
We study H\'enon maps which are perturbations of a hyperbolic polynomial
$p$ with connected Julia set. We give a complete description of the critical
locus of these maps. In particular, we show that for each critical
point $c$ of $p$, there is a primary component of the critical locus asymptotic
to the line $y=c$.  Moreover, primary components are conformally equivalent to  the punctured disk, 
and their orbits cover the whole critical set.
We also describe the holonomy maps from such a component 
to itself along the leaves of two natural foliations.
Finally, we show that a quadratic H\'enon map 
taken along with the natural pair of foliations, is a rigid object, 
in the sense that a conjugacy between two such maps respecting the foliations 
is a holomorphic or antiholomorphic affine map.
\end{abstract}

\tableofcontents

\newpage
  
  \section*{Preamble}

  This paper was written in 2005--2006,
but has never appeared even as a preprint.
Meanwhile, the results have been  developed further and have found 
some applications, see  \cite{F,Tanase,FL}. We are grateful to Tanya Firsova
for insisting that this paper should be made available and for helping
with the proof reading.

\section*{Introduction}
\addcontentsline{toc}{part}{Introduction}

The family of H\'enon maps are a basic example of nonlinear dynamics. Both the real
and the holomorphic versions of these maps
have been studied extensively, and yet there is still a great deal that 
is not well understood about them. Some of the sources of fundamental results about
H\'enon maps are \cite{friedland-milnor}, \cite{forn:henon}, \cite{hubbard_oberste-vorth:henon1},
\cite{hubbard_oberste-vorth:henon2}, \cite{bedfordsmillie1}, \cite{bedfordsmillie2},
\cite{bedfordsmillie3}, \cite{bedfordlyubichsmillie4}, \cite{bedfordsmillie5},
\cite{bedfordsmillie6}, and
\cite{bedfordsmillie7}.

In this article we study holomorphic H\'enon maps of $\mb{C}^2$. These are maps of the form \name{defOff}
\[f_a\vt{x \\ y} = \vt{ p(x) - a y \\ x }\] 
where $p$ \name{defOfp} is a monic polynomial of degree $d > 1$.\name{defOfd} 
H\'enon maps have constant Jacobian, and the parameter
$a$ \name{defOfa} is the value of the Jacobian. In the degenerate case where $a=0$ the map reduces to
$f_0\vt{x \\ y} = \vt{ p(x) \\ x}$ and we see that the H\'enon map degenerates to
the polynomial map $p(x)$, acting on the copy of the complex plane given by
the curve $x=p(y)$.

The H\'enon maps we study here
are perturbations of hyperbolic polynomial maps with connected Julia set.
The Julia sets and {\em natural foliations}
of these maps was described in great detail by Hubbard and Oberste-Vorth
in {\henI} and {\henII}.
In this paper we will describe the {\em tangency locus} between the natural foliations
and will derive from it that the H\'enon map endowed with the pair of
foliations is a {\em rigid}  object.  

\medskip
Let us now outline the content of the paper in more detail.

Throughout Section~\ref{sectionFolitaionsNearDegeneracy}
we will recap results of Hubbard \& Oberste-Vorth {\henI} keeping careful
track of what happens as the Jacobian of the H\'enon map goes to zero.
When the Jacobian is equal to zero, the map degenerates, 
but the foliations and plurisubharmonic functions
associated to the map persist, and become easy to analyze. 

In Section~\ref{sectionTheCriticalLocusNearInfinity}
we present basic facts about the critical locus and, by
direct calculation, obtain a description of
the tangent spaces to its  primary components 
at infinity.

In Section~\ref{sectionStableAndUnstableManifolds}
we recall relevant material from {\henII}
concerning the stable and unstable foliations
and describe the critical locus when the Jacobian is zero.

In Section~\ref{sectionComponentsOfTheCriticalLocus}
we construct tubes that trap the components of 
the critical locus as the Jacobian varies away from zero.
This allows us to  prove that the primary horizontal components
of the critical locus are punctured disks. We then show
that every component of the critical locus is an iterate
of a primary  component.

In Section~\ref{sectionHolonomyInTheCriticalLocus} we 
describe the holonomy maps
on a primary horizontal component of the critical locus 
along the natural foliations. 

The pair of natural  of foliations of a H\'enon map
can be thought of as giving natural coordinates near infinity.
In Section~\ref{section-conjugacies} we prove
 that if a conjugacy between two  H\'enon  maps
 in question respects these foliations
then it is forced
to be holomorphic or antiholomorphic  near infinity.
For degree two maps, this implies that it is actually  
affine, which is our main rigidity result.
%

A list of notations is provided as a reference at the end of the paper.
These notations are used with the following convention.

\begin{convention}
When we extend a certain object from $\mb{C}^2$ to $\mb{P}^1\times\mb{P}^1$
we add a hat to the symbol to distinguish it, unless there is
no chance for confusion. When referring to a set with the subset
$\{a=0\}$ removed we append the symbol $*$ as a superscript to the symbol denoting that set.
\end{convention}

\newpage
\part{Foliations.}

\section{The Foliations near Degeneracy}

\name{sectionFolitaionsNearDegeneracy}

\subsectionDone{The Foliations.}
\name{subsectionTheFoliations}

In this section we define the foliations associated to a H\'enon map. These foliations
are not new, they were introduced and studied in {\henI}. We give a careful development
of them from scratch, following the same methods as {\henI}, in order to
study what happens in the degenerate case, and to have a good handle on these
foliations as the Jacobian $a$ is allowed to vary.

In studying H\'enon maps it is common to define domains $V_+$ and
$V_-$ \name{defOfV} such that $f_a(V_+)\subset V_+$ and $f_a^{-1}(V_-) \subset V_-$
and such that every point that has unbounded forward orbit eventually
enters $V_+$ and every point with unbounded backward orbit eventually
enters $V_-$. We will give precise definitions of these domains shortly.

We first recap the construction of the functions $\ppl{a}\colon V_+\to\mb{C}$ and
$\pmn{a}\colon V_-\to\mb{C}$, both of which are holomorphic for $a$ in some disk
such that $\ppl{a}\circ f_a = \ppl{a}^d$
and $\ppl{a}(x,y) \sim x$ for $\vert x \vert > \vert y\vert $ as $\vert x\vert\to \infty$
and $\pmn{a}\circ f_a^{-1}=
\pmn{a}^d/a$ holds\footnote{
The definition of $\pmnPlain$ given in {\henI} has an inconsistency that
is trivial to correct, but is essential to our calculations (specifically the 
conditions $\pmn{a} \sim  y$ and
$\pmn{a}\circ f^{-1} = \pmn{a}^d$ are incompatible).}
for $a\not=0$
and $\pmn{a}(x,y)\sim y$ for $\vert y\vert > \vert x \vert$ as $\vert y\vert\to \infty$.
\name{defOfPplAndPmn}

Throughout this paper it will be convenient
to consider the highest term of $p(x)$ separately, thus
we write $p(x)=x^d + q(x)$\name{defOfq} where $d\geq 2$ and $\deg q(x) < d$.
We also let $\degq = \deg q(x)$.\name{defOfDegq}

We want to construct domains $V_+$ and $V_-$ where functions $\ppl{a}$ and
$\pmn{a}$ are defined for all $a$ in the disk $\disk_R$ of radius $R$. We will need
to control convergence of an infinite product to construct
$\ppl{a}$ and $\pmn{a}$  and will choose a value $r$ which will control the rate
of convergence of this series.

Fix values $0 < r < 1$
and $R > 0$,\name{defOfR} and choose $\alpha > 0$ such that 
\begin{itemize}
\item $\Big\arrowvert \dfrac{q(y)}{y^d}\Big\arrowvert + \dfrac{R+1}{\vert y^{d-1}\vert}< r$,
\item $\vert p(y)\vert > (2R+1)\vert y\vert$
\end{itemize}
whenever $\vert y\vert \geq \alpha$.

We then define the domains $V_+$ and $V_-$ 
to be given by

\[V_+\equiv\{(x,y)\vert \vert x\vert > \vert y\vert\ \text{and}\ \vert x\vert > \alpha\},\]
\[V_-\equiv\{(x,y)\vert \vert y\vert > \vert x\vert\ \text{and}\ \vert y\vert > \alpha\},\]

We let $\vt{x_n \\ y_n} = f_a\cp{n}\vt{x \\ y}$ for $n\in\mb{Z}$\name{defOfXAndY} 
so $x_n$ and $y_n$ are polynomial functions in $x$, $y$ and $\inlinefrac{1}{a}$ for $n <0$
and $x_n$ and $y_n$ are polynomial functions in $x$, $y$ and $a$ for $n >0$.

\begin{lemma}
\name{preimageOfVmnInsideVmn}
Let $a\in \disk_R$.
If $(x,y)\in V_-$ then
$\vert x_{-1}\vert = \vert y\vert > \vert x \vert$ and $\vert y_{-1}\vert >2\vert y\vert =2\vert x_{-1}\vert$.
Thus $f_a^{-1}(V_-)\subset V_-$.
\end{lemma}
\begin{proof}
The statement about $x_{-1}$ is obvious. For $y_{-1}$ we have
\[\vert y_{-1}\vert = 
\big\arrowvert \frac{p(y)-x}{a} \big\arrowvert \geq 
 \frac{\vert p(y)\vert -\vert y\vert}{\vert a\vert} \geq
\frac{2R\vert y\vert}{\vert a\vert}  > 2\vert y\vert.\]
\end{proof}

\begin{lemma}
Let $a\in \disk_R$. If $(x,y)\in V_+$ then
$\vert y_1\vert = \vert x\vert > \vert y \vert$ and 
$\vert x_1\vert >(R+1)\vert y_1\vert =(R+1)\vert x\vert$.
Thus $f_a(V_+)\subset V_+$.
\end{lemma}
\begin{proof}
The statement about $y_1$ is obvious. For $x_1$ we have
\[\vert x_1\vert \geq \vert p(x)\vert -\vert ay\vert
\geq \vert p(x)\vert -R\vert x\vert\geq (R+1)\vert x\vert = (R+1)\vert y_1\vert.\]
\end{proof}

\subsectionDone{Degeneration of $\pmn{a}$.}

\head{Leading terms.}

We are chiefly interested in the degenerate case ($a=0$) and perturbations of this case
($a$ small). 
We start by working out
the degree of $x_{-k}$ and $y_{-k}$ as polynomials in $y$ and in $\inlinefrac{1}{a}$.
In doing so it will be convenient to make the definition
$\expDSum{k}= 1+d+d^2+\dotsb+d^{k-1}$ for $k\geq 1$ and $\expDSum{k}=0$ for $k\leq 0$.

By an easy induction we obtain
\begin{lemma}
\name{degreesMn}
Given that $k\geq 1$
then the leading term of $y_{-k}(x,y,a)$ is 
$$
\inlinefrac{1}{a^{\expDSum{k}}}{\bigl(p(y)-x\bigr)^{d^{k-1}}}
$$
if $y_{-k}$ is considered as a polynomial in $\shortInlineFrac{1}{a}$. The leading
term of $y_{-k}(x,y,a)$ considered as a polynomial in $y$ is just the term
$\shortInlineFrac{y^{d^k}}{a^{\expDSum{k}}}$ of it's leading term 
in $\frac{1}{a}$. Since $x_{-k}(x,y,a)=y_{-(k-1)}(x,y,a)$ this also gives us
the leading terms of $x_{-k}(x,y,a)$ in $\frac{1}{a}$ and in $y$ except that
$x_{-1}(x,y,a)=y$. 
\end{lemma}

\head{Definition of $\pmn{a}$.}

The function $\pmn{a}$ is constructed as a limit
\[\pmn{a}=\lim_{n\to\infty} (y_{-n}\cdot a^{\expDSum{n}})^\frac{1}{d^n},\]
with an appropriate choice of the branch of the root.
We are interested in this as a function of $V_-\times \disk_R$ where $(x,y)\in V_-$ and
$a\in \disk_R$.  Sense is made of the above limit using the telescoping formula
\begin{equation}
\name{telescopingMn}
\pmn{a}(x,y)=\lim_{n\to\infty} y \cdot \exp\Bigl(\dfrac{1}{d}\log\dfrac{ay_{-1}}{y^d} + 
\dfrac{1}{d^2}\log\dfrac{ay_{-2}}{y_{-1}^d} +\dotsb
\Bigr)
\end{equation}

We are most interested in this about the point $a=0$ (where the map $f_a$ degenerates
and $\pmn{a}$ can no longer be defined using its relationship with $f_a$).

\head{Introduction of $\smn{k}$.}

We note that
\[\dfrac{ay_{-k}}{y_{-(k-1)}^d}=\dfrac{p(y_{-(k-1)})-x_{-(k-1)}}{y_{-(k-1)}^d} =
1+\dfrac{q(y_{-(k-1)})-x_{-(k-1)}}{y_{-(k-1)}^d}.\]
Let 
\begin{equation}
\name{definingSmn}
\smn{k}(x,y,a)=\dfrac{q(y_{-(k-1)})-x_{-(k-1)}}{y_{-(k-1)}^d}.
\end{equation} 
\name{defOfSmn}

\head{Controlling the Size of $\smn{k}$.}

\begin{lemma}
\name{smnBound}
$\vert \smn{k}(x,y,a)\vert <r$ for $x,y\in V_-, a\in \disk_R, k\geq 1$.
\end{lemma}
\begin{proof}
We have: \[\vert \smn{k}\vert
\leq \dfrac{\vert q(y_{-(k-1)})\vert  + \vert x_{-(k-1)}\vert}{\vert y_{-(k-1)}^d\vert}\leq
\Big\arrowvert \dfrac{q(y_{-(k-1)})}{y_{-(k-1)}^d}\Big\arrowvert + 
\dfrac{1}{\vert y_{-(k-1)}^{d-1}\vert} < r\]
\end{proof}

\head{Estimates and convergence for the telescoping sum.}

We evaluate $\log\inlinefrac{ay_{-k}}{y_{-(k-1)}^d}=\log(1+\smn{k})$ using the principal branch
of $\log$.
By Lemma~\ref{smnBound}, $\log(1+\smn{k})\leq -\log(1-r)$.
Hence the series

\begin{equation}
\name{infiniteSumMn}
\frac{1}{d}\log\frac{ay_{-1}}{y^d} + \frac{1}{d^2}\log\frac{ay_{-2}}{y_{-1}^d} + \dotsb
\end{equation}
converges uniformly and absolutely to a holomorphic function
bounded by \newline
$\inlinefrac{-\log(1-r)}{d-1}$.
Letting $B=(r-1)^{-\frac{1}{d-1}}$, we conclude:

\head{Existence, bounds and holomorphy of $\pmn{a}$.}

\begin{corollary}
\name{corPmnBound}
$\pmn{a}$ as defined by equation (\ref{telescopingMn})
is holomorphic as a function on $V_-\times \disk_R$ with
$(x,y)\in V_-$ and $a\in \disk_R$. Moreover
$B^{-1} < \Big\arrowvert\dfrac{\pmn{a}}{y}\Big\arrowvert < B$.
\end{corollary}

\head{Extension of $\inlinefrac{\pmn{a}}{y}$ to $\mb{P}^1\times\mb{P}^1$.}

It will be convenient at times to understand the behavior of $\pmn{a}$
in a suitable compactification of $\mb{C}^2$.
By the Riemann Extension Theorem (see e.g. \cite{cas} page 132),

\begin{corollary}
\name{corBoundednessMn}
Let $\hat{V}_-$ denote the union of
$V_-$ and the line $y=\infty$ in $\mb{P}^1\times\mb{P}^1$,
with the point $(\infty,\infty)$ excluded. Then
$\inlinefrac{\pmn{a}}{y}$ extends holomorphically to $\hat{V}_-\times \disk_R$
and the norm of the extension is bounded above by $B$ and below by $B^{-1}$
\end{corollary}
\name{defOfHatVMn}

\head{Vanishing of $\smn{k}$ in $v$ at infinity.}
Later, when we study the extension of $\inlinefrac{\pmn{a}}{y}$ described above
it will be useful to understand the behavior of $\smn{k}$ near infinity.
Lemma~\ref{degreesMn} implies:

\begin{lemma}
\name{vanishingInVmn}
Letting $v=\inlinefrac{1}{y}$ then
$\smn{k}(x,y,a)$ vanishes to order at least $d^{k-1}$
in~$v$.
\end{lemma}

\head{Asymptotics of $\pmn{a}(x,y)$ as $\vert y\vert \to \infty$.}

\begin{corollary}
$\pmn{a}(x,y)\sim y$ as $\vert y\vert \to \infty$.
\end{corollary}
\begin{proof}
We know that $\inlinefrac{\pmn{a}(x,y)}{y}=
\exp\Bigl(\inlinefrac{1}{d}\log(1+\smn{1}) + \inlinefrac{1}{d^2}\log(1+\smn{2}) + \dotsb\Bigr)$.
Since $\smn{k}$ vanishes in $v$ for all $k$ we see that this infinite sum
is a holomorphic function in $V_-\times \disk_R$ 
that vanishes in~$v$. 
\end{proof}

\head{Vanishing of $\smn{k}$ in $a$.}

Let us now study the behavior of $a\mapsto \pmn{a}$ near $a=0$.

\begin{lemma}
\name{vanishingOrderMn}
$\smn{k}(x,y,a)$ vanishes in $a$ precisely to order
\begin{itemize}
\item $(d-\degq)\cdot \expDSum{k-1}$ for $q$ nonconstant;
\item $d\cdot\expDSum{k-1}-\expDSum{k-2}$ for $q$ constant.
\end{itemize}
Moreover, $\smn{k}(x,\shortInlineFrac{1}{v},a)$ vanishes in $v$ precisely to order 
\begin{itemize}
\item $(d-\degq) d^{k-1}$ for $q$ nonconstant.
\item $d^k-d^{k-2}$ for $q$ constant and $k\geq 2$,
\item $d$ for $q$ constant and $k=1$.
\end{itemize}
\end{lemma}
\begin{proof}
It follows from Lemma~\ref{degreesMn} that the
denominator of $\smn{k}$ is a polynomial
in $\inlinefrac{1}{a}$ of degree $d\cdot\expDSum{k-1}$.

The numerator of $\smn{k}$ is a polynomial
in $\inlinefrac{1}{a}$ of degree
\begin{itemize}
\item $\degq\cdot\expDSum{k-1}$ for $q$ nonconstant;
\item $\expDSum{k-2}$ for $q$ constant.
\end{itemize}

To justify that the highest degree terms
in the numerator never cancel, observe that it is impossible for the degree
of $q(y_{-(k-1)})$ to match the degree of $x_{-(k-1)}$ as polynomials
in $\shortInlineFrac{1}{a}$ except when $k=1$ or when
$q$ is constant and $k=2$. It is easy to check that the Lemma still holds in these cases.

The last assertion also easily follows from
Lemma~\ref{degreesMn}.
\end{proof}

\head{Precise behavior of $\pmn{a}$ near $a=0$.}

\begin{lemma}
\name{pmnNearZeroA}
$\pmn{a}(x,y)=\bigl(p(y)-x\bigr)^\frac{1}{d}+ah(x,y,a)$ for some holomorphic
function $h$ on $V_-\times \disk_R$.
\end{lemma}
\begin{proof}
According to Lemma~\ref{vanishingOrderMn} all $\smn{k}$, $k\geq 2$,
vanish at $a=0$ and hence the series
(\ref{infiniteSumMn})
takes the form $\inlinefrac{1}{d}\log\inlinefrac{ay_{-1}}{y^d} + a g(x,y,a)$ for
some holomorphic function $g(x,y,a)$ on $V_-\times \disk_R$.
Hence by (\ref{telescopingMn})
 \[\pmn{a}(x,y)=y \cdot\exp\Bigl(\inlinefrac{1}{d}\log\inlinefrac{ay_1}{y^d}\Bigr) 
\cdot \exp\bigl(ag(x,y,a)\bigr)=\bigl(p(y)-x\bigr)^\frac{1}{d}\exp\bigl( ag(x,y,a)\bigr),\]
and the conclusion follows.
\end{proof}

\head{Swelling nature of $f_a(V_-)$ as $a\to 0$.}

The domain $f_a(V_-)$ swells as $a\to 0$ to include all of $\mb{C}^2$
except the curve $C(p)$. We make this precise:

\begin{lemma}
\name{imageOfVmnIsFilling}
Given $(x,y)\not\in \jmn{0}=C(p)$ then $(x,y)\in f_a(V_-)$
for all sufficiently small values $a$. More generally if $K \Subset \mb{C}^2\setminus C(p)$
then $K\subset f_a(V_-)$ for all sufficiently small $a$.
\end{lemma}
\begin{proof}
This follows because
$(x,y)\in f_a(V_-)$ iff $\vert p(y)-x\vert\geq \alpha\vert a\vert$ and
$\vert p(y)-x\vert \geq \vert y\vert\vert a\vert$.
\end{proof}

\head{Extending powers of $\pmn{a}$ to $\umn{}$.}

We let
$\mf{V}_{k,-}^\circ=\{(x,y,a)\vert (x,y)\in f_a\cp{k}(V_-), a\in \disk_R^*\}$\name{defOfMfVMn}
and we let $\mf{V}_{k,-}$
be the union of $\mf{V}_{k,-}^\circ$ and the set $\{(x,y,0)\vert (x,y)\not\in C(p)\}$.

\head{Showing $\pmn{0}^d=p(y)-x$ on $\mb{C}^2\setminus C(p)$.}

\begin{lemma}
\name{pmnForZeroA}
Given $k\geq 1$ then
$\pmn{a}^{d^k}$ extends from
a holomorphic function on $V_-$, to a holomorphic function on $\mf{V}_{k,-}$ by defining
$\pmn{a}^{d^k}(x,y)\equiv a^{\expDSum{k}} \pmn{a}(f_a^{-k}(x,y))$
for $a\not=0$ and $\pmn{0}^{d^k}(x,y)=\bigl(p(y)-x\bigr)^{d^{k-1}}$.
\end{lemma}
\begin{proof}
If $a\not=0$ then
we can extend the function $\pmn{a}^{d^k}$
to be holomorphic on $f_a\cp{k}(V_-)$ by defining
$\pmn{a}^{d^k}(x,y)\equiv a^{\expDSum{k}} \pmn{a}(f_a^{-k}(x,y))$.
This agrees with $\pmn{a}^{d^k}$ on $V_-$.

According to our definition and Lemma~\ref{pmnNearZeroA}
\begin{equation}
\name{extensionMn}
\begin{split}
\pmn{a}^{d^k}(x,y)& =  a^{\expDSum{k}}\pmn{a}(x_{-k},y_{-k}) \\
& =a^{\expDSum{k}}y_{-k}\cdot \Bigl( 1 + \inlinefrac{q(y_{-k})-x_{-k}}{y_{-k}^d}\Bigr)^\frac{1}{d} +
a^{1+\expDSum{k}}h(x_{-k},y_{-k}) \\
& =a^{\expDSum{k}} y_{-k}(1+\smn{k}(x,y))^\frac{1}{d} + a^{1+\expDSum{k}}h(x_{-k},y_{-k})
\end{split}
\end{equation}
for $(x,y,a)\in \mf{V}^\circ_{k,-}$.
Now $(x_{-k},y_{-k})\in V_-$ when $(x,y,a)\in \mf{V}^\circ_{k,-}$ 
and so $y_{-k}\not=0$.
Also $\smn{k}(x_{-k},y_{-k})$ is defined 
and holomorphic since $y_{-k}\not=0$. From Lemma~\ref{degreesMn}
we see that $a^{\expDSum{k}} y_k$ is a polynomial in $x,y,a$ whose
only term not divisible by $a$ is $\bigl(p(y)-x\bigr)^{d^{k-1}}$. 
By Lemma~\ref{vanishingOrderMn} $\smn{k}$ vanishes in $a$ and so
$\pmn{0}^{d^k}(x,y)\equiv\bigl(p(y)-x\bigr)^{d^{k-1}}$
gives a continuous, and therefore, a holomorphic, extension of $\pmn{a}^{d^k}$
to $\mf{V}_{k,-}$.
\end{proof}

\head{We define $\jmn{0}$ and $\kmn{0}$.}

We let both $\jmn{0}$ and $\kmn{0}$
denote the curve $p(y)-x=0$. It follows from the previous result
that this is consistent with the convention that
$\jmn{a}$ and $\kmn{a}$ will denote the sets $J_-$
and $K_-$ for the parameter value $a$.

\subsectionDone{Degeneration of $\ppl{a}$}

\head{Leading terms.}

Here we include the corresponding constructions for forward iterates.

\begin{lemma}
\name{degreesPl}
The leading term of $x_k(x,y,a)$ is $x^{d^k}$
if $x_k$ is considered as a polynomial in $x$.
The leading term of $y_k(x,y,a)$ is $x^{d^{k-1}}$
as a polynomial in $x$.
\end{lemma}
\begin{proof}
This follows from an easy induction using
$\vt{x_1 \\ y_1}=\vt{p(x)-ay \\ x}$
and $\vt{x_k \\ y_k} = \vt{ p(x_{k-1})-ay_{k-1} \\ x_{k-1} }$.
\end{proof}

\head{Definition of $\ppl{a}$.}

The function $\ppl{a}$ is  constructed as a limit
$\ppl{a}=\lim_{n\to\infty} {x_n}^\frac{1}{d^n}$ with an appropriate choice of root
on the domain $V_+\times \disk_R$.
Sense is made of this using the telescoping formula,

\begin{equation}
\name{telescopingPl}
\ppl{a}(x,y)=\lim_{n\to\infty} x \cdot \exp\Bigl(\inlinefrac{1}{d}\log\inlinefrac{x_1}{x^d} +
\inlinefrac{1}{d^2}\log\inlinefrac{x_2}{x_1^d} +
\dotsb 
\Bigr)
\end{equation}

\head{Controlling the size of $\ppl{a}$.}

Letting $\spl{k}=\dfrac{q(x_{k-1})-ay_{k-1}}{x_{k-1}^d}$ \name{defOfSpl} 
so $\inlinefrac{x_k}{x_{k-1}^d}=1+\spl{k}$ we see that
\begin{lemma}
\name{splBound}
$\vert \spl{k}(x,y,a)\vert < r$ for $(x,y)\in V_+, a\in \disk_R$.
\end{lemma}
\begin{proof}
\[\vert \spl{k}\vert  <
\dfrac{\vert q(x_{k-1})\vert  + R\vert y_{k-1}\vert}{x_{k-1}^d}\leq
\Big\arrowvert \dfrac{q(x_{k-1})}{x_{k-1}^d}\Big\arrowvert + \inlinefrac{R}{\vert x_{k-1}^{d-1}\vert} < r\]
\end{proof}

We evaluate
$\log\inlinefrac{x_k}{x_{k-1}^d}=\log(1+\spl{k})$ using the principal branch
of $\log$. Since $\vert\spl{k}\vert < r$ then
$\vert\log(1+\spl{k})\vert < -\log(1-r)$ exactly as before for $\smn{k}$.
It follows that the series
\begin{equation}
\name{infiniteSumPl}
\inlinefrac{1}{d}\log\inlinefrac{x_1}{x^d} + \inlinefrac{1}{d^2}\log\inlinefrac{x_2}{x_1^d} + \inlinefrac{1}{d^3}\log\inlinefrac{x_3}{x_2^d} +\dotsb
\end{equation}
converges absolutely and uniformly.

\head{Existence, holomorphy and boundedness of $\ppl{a}$.}

Since $\vert \inlinefrac{1}{d^k}\log\inlinefrac{x_k}{x_{k-1}^d}\vert < \inlinefrac{-\log(1-r)}{d^k}$
then the infinite sum (\ref{infiniteSumPl}) is no larger than $\log B=\inlinefrac{-\log(1-r)}{d-1}$.
We conclude that

\begin{corollary}
\name{corPplBound}
The function $\ppl{a}$ defined by
equation (\ref{telescopingPl}) is
well defined and holomorphic for all $(x,y)\in V_+$ and all $a\in \disk_R$.
Additionally $B^{-1} < \Bigl\arrowvert\inlinefrac{\ppl{a}}{x}\Bigr\arrowvert < B$.
\end{corollary}
\begin{proof}
The final claim follows immediately from the expression (\ref{telescopingPl})
and the bounds just derived on the series (\ref{infiniteSumPl}).
\end{proof}

\head{Extension of $\inlinefrac{\ppl{a}}{x}$ to $\mb{P}^1\times\mb{P}^1$.}

For the next lemma we consider $\mb{C}^2$ as lying in $\mb{P}^1\times \mb{P}^1$.

\begin{corollary}
\name{corBoundednessPl}
Let $\hat{V}_+$ denote the union of $V_-$ and the line
$x=\infty$ in $\mb{P}^1\times \mb{P}^1$, with the point $(\infty,\infty)$
excluded. Then
$\inlinefrac{\ppl{a}}{x}$ extends holomorphically to $\hat{V}_+\times \disk_R$
and $B^{-1}\leq \Big\arrowvert \inlinefrac{\ppl{a}}{x}\Big\arrowvert \leq B$.
\end{corollary}
\begin{proof}
This follows from the
Riemann extension theorem (see e.g. \cite{cas} page 132).
\end{proof}
\name{defOfHatVPl}

\head{Vanishing of $\spl{k}$ in $u$ at infinity.}
In order to better understand this extension
we will need to understand $\spl{k}$ on
this extension. We extract the relevant information in the following lemma.

\begin{lemma}
\name{vanishingInUpl}
Letting $u=\inlinefrac{1}{x}$ then
$\spl{k}(x,y,a)$ vanishes to order at least $d^{k-1}$
in~$u$.
\end{lemma}
\begin{proof}
Writing $\spl{k}=
\dfrac{q\bigl(x_{k-1}(\shortInlineFrac{1}{u},y,a)\bigr)-ay_{k-1}(\shortInlineFrac{1}{u},y,a)}{x_{k-1}^d(\shortInlineFrac{1}{u},y,a)}$
and multiplying the numerator and denominator by $u^{d^k}$ this
follows from Lemma~\ref{degreesPl}.
\end{proof}

\head{Showing $\ppl{a}(x,y,a)\sim x$ as $\vert x\vert\to\infty$.}

\begin{corollary}
$\ppl{a}(x,y,a)\sim x$ as $\vert x\vert \to\infty$.
\end{corollary}
\begin{proof}
We know that $\inlinefrac{\ppl{a}(x,y)}{x}=\exp\Bigl(\inlinefrac{1}{d}\log(1+\spl{1}) + \inlinefrac{1}{d^2}\log(1+\spl{2}) + \dotsb\Bigr).$
Since $\spl{k}$ vanishes in $u$ for all $k$ we see that this infinite sum
is a holomorphic function on $V_+\times\disk_R$
that vanishes in $u$. 
\end{proof}

\head{Behavior of $\ppl{a}$ when $a=0$.}

We now have an expression for $\ppl{a}$ as $x$ multiplied by
the exponent of a uniformly convergent sum of functions which are holomorphic
for $(x,y)\in V_+$ and $a\in \disk_R$. 

\head{Extension of powers of $\ppl{a}$ to $\upl{}$.}

We let
$\mf{V}_{k,+}=\{(x,y,a)\vert f_a\cp{k}(x,y)\in V_+, a\in \disk_R\}$.

\begin{lemma}
\name{pplForZeroA}
For each $k\geq 0$ the function $\ppl{a}^{d^k}$ extends
to a holomorphic function on $\mf{V}_{k,+}$ given by
by $\ppl{a}^{d^k}\equiv \ppl{a}\circ f_a^{k}$.
The function $\ppl{0}$ is defined and holomorphic 
on all of $\mb{C}^2\setminus \kpl{0}$ and is the 
B\"ottcher coordinate $b_p(x)$ of $p(x)$.
\end{lemma}
\begin{proof}
Unlike the case of $\pmn{a}$, there is no
difficulty in case $a=0$ here.  It is immediate that $\ppl{a}$ is holomorphic
as a function on $V_+\times \disk_R$. It is clear from the definition
that $\ppl{0}$ is the B\"ottcher coordinate
$b_p(x)$ of $p(x)$. The rest of the Lemma is obvious.
\end{proof}

\subsectionDone{The Functions $\gpl{a}$ and $\gmn{a}$.}

\head{Definition of $\upl{a}$, $\umn{a}$, $\bigUPl$ and $\bigUMn$.}

\begin{definition}
We follow the standard convention
that $\upl{a}\equiv \mb{C}^2\setminus \kpl{a}$
and $\umn{a}\equiv \mb{C}^2\setminus \kmn{a}$.
We define $\bigUPl \equiv\{(x,y,a)\vert (x,y)\in \upl{a}, a \in \disk_R\}$
and $\bigUMn\equiv \{(x,y,a)\vert (x,y)\in \umn{a}, a \in \disk_R\}$ with
the convention that $\umn{0}\equiv \mb{C}^2\setminus C(p)$.
\end{definition}
\name{defOfU}

\head{$\bigUPl$ and $\bigUMn$ are open.}

\begin{lemma}
\name{bigUPlAndbigUMnOpen}
The sets $\bigUPl$ and $\bigUMn$ are open in $\mb{C}^3$.
\end{lemma}
\begin{proof}
$\bigUPl = \bigcup_{n=0}^\infty f^{-n}(V_+)$
(even when $a=0$) and therefore $\bigUPl$ is open.
The result follows for $\bigUMn$ by a similar construction
for $a\not=0$ and by Lemma~\ref{imageOfVmnIsFilling} for $a=0$.
\end{proof}

Another fact we will want is
\begin{lemma}
\name{inflatingKplZero}
If $p$ is hyperbolic then
no point interior to $\kpl{0}$ lies in the closure 
of $\bigUPl$.
\end{lemma}
\begin{proof}
This is a trivial consequence of the fact that any point interior to $\kpl{0}$
is attracted to an attracting cycle.
\end{proof}

\head{We define $\gpl{a}$ and $\gmn{a}$.}

We recall\footnote{Again with the correction in $\gmn{a}$, similar to the one made
  for $\pmn{a}$. Notice that it makes the second Green function to be a {\em non-zero} constant
on $K_a^-$.}
the Green's functions
\[\gpl{a}(x,y)\equiv
\left\{
\begin{array}{ll}
\log\vert \ppl{a}(x,y)\vert & (x,y)\in \upl{a} \\
0 & (x,y)\in \kpl{a}
\end{array} \right.
\]
and
\[\gmn{a}(x,y)\equiv
\left\{
\begin{array}{ll}
\log\vert \pmn{a}(x,y)\vert & (x,y)\in \umn{a} \\
\inlinefrac{1}{d-1}\log\vert a\vert & (x,y)\in \kmn{a}.
\end{array}\right.
\]
\name{defOfG}
We take the value of $\gmn{0}(x,y)$ to be $-\infty$ if $(x,y)\in C(p)$, i.e. if $x=p(y)$.
These satisfy 
\[ \gpl{a}\circ f_a = d\cdot \gpl{a} \]
and 
\[ \gmn{a} \circ f_a^{-1}= d\cdot \gmn{a} - \log\vert a\vert\quad\text{for $a\not=0$}.\]
This second relation is sometimes more conveniently written
\[\bigl(\gmn{a}-\inlinefrac{1}{d-1}\log\vert a\vert\bigr)\circ f_a^{-1} =
d\cdot\bigl (\gmn{a} - \inlinefrac{1}{d-1}\log\vert a\vert\bigr).\]

%
%

\begin{convention}
We will sometimes write $\gpl{}(x,y,a)$ for $\gpl{a}(x,y)$ and 
$\gmn{}(x,y,a)$ for $\gmn{a}(x,y)$. This will be convenient, for instance,
when postcomposing $\gpl{}$ with function whose output lies
in $\mb{C}^2\times \disk_R$.
\end{convention}

Hubbard \& Oberste-Vorth proved that the Green's functions are continuous when $f_a$
is nondegenerate and the same argument gives continuity in $x$, $y$ and $a$ for $\gpl{a}$
when $a=0$. We extend this to $\gmn{a}$ when $a=0$.

\begin{theorem}
\name{thm-greensContinuous}
The functions $\gpl{a}(x,y)$ and $\gmn{a}(x,y)$ are continuous
in $x$, $y$ and $a$ for $a\in \disk_R$.
\end{theorem}
\begin{proof}
This follows by the same argument
as is used in {\henI} 
except in the case of $\gmn{}$ when $a=0$.
For $(x',y')\not\in C(p)$ the continuity of
$\gmn{a}$ at $(x',y')$ and $a=0$ follows from Lemma~\ref{pmnForZeroA}.
For $(x',y')\in C(p)$ more work is required. 
If we restrict $\gmn{a}$ to
the slice $a=0$ then we already have shown continuity, so we will assume for
most of the rest of this proof that $a\not=0$ (so $f_a^{-1}$ is defined).

If $f^{-n}(x,y)\in V_-$ then
$B^{-1} < \big\arrowvert \inlinefrac{\pmn{a}(x_{-n},y_{-n})}{y_{-n}}\big\arrowvert < B$
by Corollary~\ref{corPmnBound}.
and so
$B^{-1} < \big\arrowvert \inlinefrac{ \pmn{a}^{d^n}(x,y) }{ a^{\expDSum{n}} y_{-n} }\big\arrowvert < B$.
Applying $\log$ to the right hand inequality yields
$\gmn{a}(x,y)  < \inlinefrac{1}{d^n}\log B +
\inlinefrac{1}{d^n}\log \big\arrowvert a^{\expDSum{n}} y_{-n} \big\arrowvert$.
Now $(x_{-n},y_{-n})=f_a^{-n}(x,y)=f_a^{-1}(x_{-(n-1)},y_{-(n-1)})=
(y_{-(n-1)},\inlinefrac{1}{a}\bigl( p(y_{-(n-1)}) - x_{-(n-1)}\bigr))$. Therefore
$y_{-n}=\inlinefrac{1}{a}\bigl( p(y_{-(n-1)}) - y_{-(n-2)}\bigr)$
for $n\geq 2$.

We let $z_{-n}\equiv a^{\expDSum{n}} y_{-n}$ from which it follows that
\[z_{-n} = p(z_{-(n-1)},a^{\expDSum{n-1}}) - a^{ \expDSum{n} - \expDSum{n-2} -1 }z_{n-2}\quad\text{for $n\geq 2$.}\]
where $z_0= y$ and $z_{-1} = p(y) - x$. Writing our bound for $\gmn{a}(x,y)$ in terms of $z_{-n}$
we get:
\begin{equation}
\name{controllingGmn}
\gmn{a}(x,y)  < \inlinefrac{1}{d^n}\log B +
\inlinefrac{1}{d^n}\log \arrowvert z_{-n} \arrowvert
\end{equation}
if $f_a^{-n}(x,y)\in V_-$.

\head{We estimate the recursion formula.}

We use the convention that $p(x,y)\equiv y^d p(x/y)$.\name{defOfHomogeneousP}
Now fix a constant $C> \frac{1}{2}$ to be greater than $(d+1)$ times the absolute value of the
largest coefficient of $p(x,y)$.
It is then evident that
$\vert p(x,y)\vert \leq C\max\{\vert x\vert^d,\vert y\vert^d\}$.
Thus from the recursion relation for $z_{-n}$ we have
$$\vert z_{-n}\vert \leq \max
\{2\cdot \big\arrowvert p(z_{-(n-1)},a^{\expDSum{n-1}})\big\arrowvert,
2\cdot \big\arrowvert   a^{ \expDSum{n} - \expDSum{n-2} -1 }z_{-(n-2)}\big\arrowvert \}
$$
$$
\leq \max\{ 2C\vert z_{-(n-1)}\vert^d ,
2C\vert a\vert^{d\cdot\expDSum{n-1}},
2 \vert a\vert^{d^{n-1} + d^{n-2} - 1}\vert  z_{-(n-2)}\vert\}.
$$

\head{We  choose a small value $\epsilon_v$, and construct a neighborhood of $(x',y'), a=0$.}

Let $\epsilon_v$ be an arbitrary positive number satisfying
$\epsilon_v < \inlinefrac{1}{2C}$.
Given a point $(x',y')\in C(p)$ consider the neighborhood 
\[\tilde{\mc{U}}(\epsilon_v)\equiv \Bigl\{(x,y,a)\Big\arrowvert\,\vert p(y)-x\vert < \epsilon_v^d,\ 
\vert a\vert < \min\{2C\epsilon_v^d,\inlinefrac{1}{2}\}\ \text{and}\ 
\vert a\vert^{d-1} \vert y\vert <  C\epsilon_v^d\Bigr\}\]
of $(x',y',0)\in\mb{C}^2\times\disk_R$.
Let  $\tilde{\mc{U}}^*(\epsilon_v)\equiv \tilde{\mc{U}}(\epsilon_v)\cap (\mb{C}^2\times\disk_R)$ 
It can be shown by induction that 
\begin{equation}
\name{induction}
\vert z_{-n} \vert < (2C)^{\expDSum{n}} \epsilon_v^{d^n}\quad\text{for $(x,y,a)\in\tilde{\mc{U}}^*(\epsilon_v)$ and $n\geq 1$.}
\end{equation}

Combining equation (\ref{controllingGmn}) and equation (\ref{induction}) gives
$\gmn{a}(x,y) < \inlinefrac{1}{d^n}\log B
+ \inlinefrac{1}{d^n}\log\Big\arrowvert (2C)^{\expDSum{n}}\epsilon_v^{d^n}\Big\arrowvert$
if $(x,y,a)\in\tilde{\mc{U}}^*(\epsilon_v)$ and 
$f_a^{-n}(x,y)\in V_-$. If $(x,y)\not\in \jmn{a}$ then $f_a^{-n}(x,y)\in V_-$
for all sufficiently large $n$ so we can take the limit as $n\to \infty$ and conclude
that $\gmn{a}(x,y) < \log\bigl (\vert 2C\vert^\frac{1}{d-1}\epsilon_v\bigr)$.
On the other hand, if $(x,y)\in \jmn{a}$ then $\gmn{a}(x,y)=\inlinefrac{1}{d-1}\log\vert a\vert$
by definition.

\head{Collect our bounds on different sets.}

We conclude that given any $M > 0$ and any $(x',y')\in C(p)$ then
there is some $\epsilon_v$ such that
$\gmn{a}(x,y) < -M$ when $(x,y,a) \in \tilde{\mc{U}}^*(\epsilon_v)\setminus \jmn{a}$, and
some $\epsilon_a$ such that if $\vert a\vert < \epsilon_a$ and
$(x,y)\in \jmn{a}$ then $\gmn{a}(x,y) < -M$. Also there is an 
open subset of $a=0$ about
$(x',y')$ on which the values of $\gmn{0}$ are always smaller
than $-M$. Combining these sets gives
a neighborhood of $(x',y',0)$
on which $\gmn{a}(x,y) < -M$.
Hence $\gmn{a}$ is continuous in $x$, $y$ and $a$.
\end{proof}

\newpage
\part{The Critical Locus.}

\section{The Critical Locus near Infinity.}

\name{sectionTheCriticalLocusNearInfinity}

\subsection{The Foliations and the Critical Locus.}

\head{We introduce $\Fpl{a}$ and $\Fmn{a}$.}

The fibers of $\ppl{a}$ and $\pmn{a}$
form holomorphic foliations of $V_+$ and
$V_-$ respectively. These foliations naturally
extend to much larger sets using dynamics.

\head{We introduce $\Fpl{a}$ and $\Fmn{a}$ on $\bigUPl$ and $\bigUMn$
respectively.}

\begin{lemma}
The holomorphic foliations defined by $\ppl{a}$ on $V_+$ and 
$\pmn{a}$ on $V_-$ can be extended to
all of $\bigUPl$ and $\bigUMn$ respectively.
The resulting foliations, which we denote
$\Fpl{a}$ and $\Fmn{a}$ respectively, are respected by the dynamics.
\end{lemma}
\begin{proof}
This is a consequence of Lemma~\ref{pmnForZeroA} and Lemma~\ref{pplForZeroA}.
\end{proof}

\begin{notation_nonumber}
We will use $\lfpl{a}(z)$ to mean the entire leaf of $\Fpl{a}$ passing through
$z$, and similarly for $\lfmn{a}(z)$.
Given a set $B\subset\mb{C}^2$ and a point $z\in B$,
we will use $\lfpl{a}(z,B)$ (resp. $\lfmn{a}(z,B)$)
to denote the connected component of $\lfpl{a}(z)\cap B$
(resp. $\lfmn{a}(z)$\ ) containing $z$. 
\end{notation_nonumber}

\begin{definition}
Given any holomorphic foliation $\mc{F}$ on a two dimensional 
complex manifold $M$ and a point $z\in M$ we will say that
a holomorphic function $g$ defined in a neighborhood $N$ of $z$
is a {\it local defining function} for $\mc{F}$ if for each leaf
$\leaf$ of $\mc{F}$, $g$ is constant on each connnect component of 
$\leaf\cap N$ and if $\ed g$ is never zero.
\end{definition}

Given a pair of holomorphic foliations $\mc{F}_1$ and
$\mc{F}_2$ on a two-dimensional complex manifold,
the critical locus $\mc{C}$ of $\mc{F}_1$ and $\mc{F}_2$ is the
complex variety given locally as the zero set (counting multiplicity)
of the holomorphic function $h$ satisfying 
$h \ed x\wedge \ed y=\ed g_1 \wedge \ed g_2$ for a pair of 
local definining functions $g_1$ and $g_2$ of $\mc{F}_1$
and $\mc{F}_2$ respectively. Equivalently
$h=\der{g_1}{x}\der{g_2}{y}-\der{g_2}{x}\der{g_1}{y}$ and hence
$\mc{C}$ is equal to the critical set (counting multiplicity of components)
of the map $(x,y)\mapsto \bigl(g_1(x,y),g_2(x,y)\bigr)$.

It is straightforward to check that altering the choice of local defining
functions for the foliations only results in multiplying $h$ by
a nonvanishing holomorphic function. It is clear that the critical locus
of a pair of foliations is exactly the set of points on which
the foliations are tangent. However, some care must be taken, as it is possible
that some components of the critical locus have multipliciy higher than one.
An example is given by the foliations with local defining functions $x$ and $x+x^2y$.
The critical locus is defined by $x^2$, which is the $y$ axis, but with multiplicity two.
Under the circumstances we are interested in we will
be able to verify that every component of the critical locus has multiplicity one.
Until that is done we will need to take into account multiplicity of components
when dealing with the critical locus.


\begin{definition}
Let $\critLoc_a$ be the critical locus
of the foliations $\Fpl{a}$ and $\Fmn{a}$. 
\name{defOfCritLoc} It is easy to confirm that
$\critLoc_a$ is a closed analytic subvariety of $\upl{a}\cap \umn{a}$ invariant under $f_a$.
\end{definition}

Observe that since we can
define $\ppl{a}$ on all of $\upl{a}$ up to local choice of a root then
we can define $\log\ppl{a}$ on all of $\upl{a}$ up to local addition
of a constant, and therefore $\ed \log\ppl{a}$ is a global holomorphic
form on $\upl{a}$. Similarly for $\ed \log\pmn{a}$ on $\umn{a}$.
If we fix $a$ and consider $\ed \log\ppl{a}\wedge\ed \log\pmn{a}=w_a(x,y) \ed x \wedge \ed y$
we get the same result as if we compute $\ed$ with $a$ nonconstant
obtaining
$$
\ed \log\pplPlain\wedge\ed \log\pmnPlain\wedge \ed a=\tilde{w}(x,y,a)\ed x\wedge \ed y\wedge \ed a
$$
in the sense that $\tilde{w}(x,y,a)=w_a(x,y)$.
Since $\ppl{a}$ and $\pmn{a}$ do not vanish in $\mb{C}^2$ then
$\ed \log\ppl{a}\wedge\ed \log\pmn{a}\wedge \ed a$ and 
any local branch\footnote{We use $\pplPlain^{-1}$ and $\pmnPlain^{-d}$ because
these are single valued and extend holomorphically about 
$(\infty,c)\in \mb{P}^1\times \mb{P}^1$. 
See Section~\ref{subsection-criticalLocusNearInfinity}.}
 of
$\ed \pplPlain^{-1}\wedge \ed \pmnPlain^{-d}\wedge \ed a$
are just multiples
of each other by a nonvanishing holomorphic function.
Thus on the common domain of definition $\tilde{w}$ and $w_a$ are
multiples of each other by a nonvanishing holomorphic function.

\begin{definition}
The variety defined by $\tilde{w}$ in $\bigUPl\cap \bigUMn$ will also be called the critical
locus and will be denoted by $\critLoc$. 
For each $a$ it is exactly the locus along which
the foliations defined by $\ppl{a}$ and $\pmn{a}$ are tangent, and is 
equal to $\critLoc_a$.
\end{definition}

\begin{lemma}
\name{intersectionLemma}
Assume $X$ is a one dimensional complex variety in $\disk^2$
and $Y$ is a smooth holomorphic curve in $\disk^2$. Assume $Y$ is not 
contained in $X$ and that $z$ is an intersection point of $X$ and $Y$.
Let $r\colon\disk\to Y$ be a local parameterization of $Y$ about $z$
with $r(0)=z$ and let $h$ be a local defining function for $X$
about $z$. Then the order of contact, or intersection multiplicity,
of $X$ and $Y$ at $z$ (counting multiplicity
of the components of $Y$ at $z$) is the order of the zero of $h\circ r$ at $0\in \disk$.
\end{lemma}
\begin{proof}
Since everything is local then without loss of generality we can 
assume that $Y$ is the $y$ axis and $z$ is the origin. Then the 
intersection multiplicity is 
$$  (X,Y)=\dim_\mb{C} \mc{O}_{\mb{C}^2,(0,0)}/\bigl(x,h(x,y)\bigr)=
\dim_\mb{C}\mc{O}_{\mb{C},0}/\bigl(h(x,y)\bigr)=
$$
$$
\dim_\mb{C}\mc{O}_{\mb{C},0}/\bigl(h(0,y)\bigr)=
\dim_\mb{C}\mb{C}[y]/(y^k)=k
$$ where 
$h(0,y)=a_ky^k + a_{k+1}y^{k+1}+\dotsb$, $a_k\not=0$.
This is easily seen to be the order of vanishing of $h\circ r$ at zero.
\end{proof}

\begin{lemma}
\name{smoothnessCritereon}
If the critical locus of a pair of holomorphic foliations in $\disk^2$
has a singularity at a point $z$ then the leaves of the foliations
must have order of contact greater than two at $z$.
\end{lemma}
\begin{proof}
We can assume without loss of generality that $z$ is the origin, 
that one foliation is the vertical complex lines, 
and that the other foliation has local defining 
function $h$ in a neighborhood of the origin.

One calculates that critical locus is defined by $\der{h}{y}=0$. If either 
$\frac{\partial^2 h}{\partial y^2}\rest{0}\not=0$ or 
$\frac{\partial^2 h}{\partial x \partial y}\rest{0}\not=0$
then $\der{h}{y}=0$ defines a smooth variety at zero. Hence if the critical locus
of the pair of foliations has a singularity at zero then $\frac{\partial^2 h}{\partial y^2}\rest{0}=0$,
and so by lemma \ref{intersectionLemma} the leaves of the two foliations have contact of
order at least three at zero.
\end{proof}

\begin{lemma}
\name{intersectionComparison}
Given two holomorphic foliations $\mc{F}_1$ and $\mc{F}_2$ 
of a two dimensional complex manifold and their 
critical locus $\mc{C}$ then:
\begin{enumerate}
\item Given any point $z\in\mc{C}$, the order of contact of the leaves of $\mc{F}_1$
and $\mc{F}_2$ at $z$ is one more than the order of contact of the leaf of either
foliation with $\mc{C}$ at $z$ (counting the multiplicity of components of $\mc{C}$).
\item The subset $\mc{K}_k\subset\mc{C}$ of points of $\mc{C}$ where the foliations
have contact of order at least $k$ is an analytic subset of $\mc{C}$.
\end{enumerate}
\end{lemma}
\begin{proof}
We will confirm both properties locally.
Without loss of generality we can assume that
$\mc{F}_1$ is the foliation of $\disk^2$ by vertical complex
lines and $\mc{F}_2$ is defined by some local defining function 
$h\colon\disk^2\to\mb{C}$. Assume $z=(z_1,z_2)\in\disk^2$ and
parameterize the leaf of $\mc{F}_1$ through $z$ by $y\mapsto (z_1,y)$.
Then by lemma \ref{intersectionLemma} the order of contact between the leaves
through $z$ is the order of vanishing of $h(z_1,y)-h(z_1,z_2)$ in $y$ at $y=z_2$. This is
precisely the smallest $k\in\mb{N}$ such that 
$\frac{\partial^j h}{\partial y^j}\rest{(z_1,z_2)}=0$ for $j=1,\dotsc,k-1$.
However, since $\mc{C}$ is defined by $\der{h}{y}$ we also conclude by
lemma \ref{intersectionLemma} that the order of contact of 
the vertical leaf throuth $z$ with $\mc{C}$ at $z$ is
the vanishing order of $\der{h}{y}(z_1,y)$ in $y$ at $y=z_2$ which
is exactly one less than the vanshing order of $h(z_1,y)-h(z_1,z_2)$
in $y$ at $y=z_2$. This completes the proof of part 1.

We also conclude that for each integer $\ell$, the set 
$\mc{K}_\ell \subset\mc{C}$ is the common zero set of 
$\frac{\partial^j h}{\partial y^j}$ for $j=1,\dotsc,\ell-1$. This completes the
proof of part 2.
\end{proof}

\begin{corollary}
\name{cor-leafContactToolbox}
Assume we are given two holomorphic foliations $\mc{F}_1$ and $\mc{F}_2$ defined on some
complex two dimensional manifold. Let $\mc{C}$ be the critical
locus. If the leaves of $\mc{F}_1$ and $\mc{F}_2$
have order of contact two at every point of some component $X$
of $\mc{C}$ then $X$ is smooth, $X$ meets no other component
of $\mc{C}$, $X$ is a component of $\mc{C}$ with multiplicity one,
and $X$ is everywhere transverse to both foliations.
\end{corollary}
\begin{proof}
Since $\mc{C}$ is smooth everywhere the leaves have order of contact two then 
$X$ is smooth and meets no other component of $\mc{C}$ by
lemma \ref{smoothnessCritereon}.
It follows by part 1 of Lemma~\ref{intersectionComparison} that $X$ must meet each leaf with  
order of contact one, i.e. $X$ must be transverse to each leaf, and $X$ must have
multiplicity one as a component of $\mc{C}$.
\end{proof}

\subsectionDone{The Critical Locus near Infinity.}
\name{subsection-criticalLocusNearInfinity}

\head{The maps $f_a$ and $f_a^{-1}$ on $\mb{P}^1\times\mb{P}^1$.}

We consider the critical
locus for the extension of $f_a$ to $\mb{P}^1\times \mb{P}^1$. 
The map $f_a$ is well defined except at $(\infty,\infty)$. 
The map $f_a^{-1}$ is well defined except at
$(\infty,\infty)$ or when $a=0$ and 
$(x,y)\in C(p)$.

The map $f_a$ sends the line $y=\infty$ to the line $x=\infty$ and sends the line
$x=\infty$ in turn to the point $(\infty,\infty)$ at which $f_a$ is undefined.
Similarly, the map $f_a^{-1}$ sends the line $x=\infty$ to the line $y=\infty$ and
sends the line $y=\infty$ in turn to the point $(\infty,\infty)$
at which $f_a^{-1}$ is undefined.

In this section we will show that for each critical point $c$ of the polynomial
$p$, and for each $a\in \disk_R$ the critical locus has a branch asymptotic
to the curve $y=c$ as $\vert x\vert \to \infty$. We will do this by
showing that the critical locus (extended to $\mb{P}^1\times\mb{P}^1$)
contains the point $(\infty,c)$ and by computing
the tangent line to the critical locus at $(\infty,c)$.

\head{Reinterpreting our extensions to $\mb{P}^1\times\mb{P}^1$}

Because we will be working in $\mb{P}^1\times\mb{P}^1\times \disk_R$
it will be convenient at times to use the coordinate systems $(u,y,a)$ where $u=1/x$
or $(x,v,a)$ where $v=1/y$ instead of using $(x,y,a)$.
Corollary~\ref{corBoundednessPl}, written using 
$(u,y,a)$ coordinates, states that  
\begin{equation}
\name{hplDef}
\hpl(u,y,a)\equiv\frac{1}{u\ppl{a}(u,y,a)}
\end{equation}
is holomorphic on $\hat{V}_+\times \disk_R$ and $B^{-1} \leq \vert \hpl\vert \leq B$.
Similarly Corollary~\ref{corBoundednessMn}, 
written using $(x,v,a)$ coordinates, states that
\begin{equation}
\name{hmnDef}
\hmn(x,v,a)\equiv\frac{1}{v\pmn{a}(x,v,a)}
\end{equation}  
is holomorphic on $\hat{V}_-\times \disk_R$
and $B^{-1} \leq \vert \hmn\vert \leq B$.
We find it useful here to write $f_a^{-1}$ with input
written in the $(u,y,a)$ coordinates, but the output is written
in the coordinates $(x,v,a)$, i.e. 
$f_a^{-1}(u,y,a)=(y,\inlinefrac{au}{up(y)-1},a)$.

Let 
$\hat{\mf{V}}^*_-=\{(x,y,a)\vert (x,y)\in f_a(\hat{V}_-), a\in \disk_R^*\}$
and let 
\[\hat{\mf{V}}_-=\hat{\mf{V}}^*_- \cup 
\Bigl(\bigl(\umn{0}\cup (\{\infty\}\times \mb{C})\bigr) \times \{0\}\Bigr)
\subset \mb{P}^1\times \mb{P}^1\times\disk_R.\]
The set $\hat{\mf{V}}_-$ is just the interior of the closure of $\hat{\mf{V}}^*_-$
in $\mb{P}^1\times\mb{P}^1\times \disk_R$.
Then  from the recursion relation it follows that $\pmn{a}^{-d}(u,y,a)=
\inlinefrac{u}{up(y)-1}\hmn\bigl(y,\inlinefrac{au}{up(y)-1},a\bigr)$
on $\hat{\mf{V}}^*_{-}$. Since both sides of the equality are holomorphic
on $\hat{\mf{V}}_-$ they are equal on $\hat{\mf{V}}_-$.

Therefore
\begin{equation}
\name{dPplInverse}
\ed \pplPlain^{-1}(u,y,a)=\bigl(\hpl
+ u \der{\hpl}{u}\bigr)\ed u + u \der{\hpl}{y}\ed y + u\der{\hpl}{a}\ed a
\quad\text{on $\hat{V}_+\times\disk_R$.}
\end{equation}

Also, letting $\lambda=up(y)-1$ we obtain,

\begin{multline}
\name{dPmnInverse}
\ed\pmnPlain^{-d}(u,y,a)= - \lambda^{-2}\cdot\hmn\circ f_a^{-1}\ed u \\
- \lambda^{-2}u^2p'(y)\cdot\hmn\circ f_a^{-1}\ed y     
+ \lambda^{-1}u\cdot\der{\hmn}{x}\circ f_a^{-1} \ed y \\
+ \lambda^{-1}u\cdot\der{\hmn}{v}\circ f_a^{-1}\cdot    
\Bigl(\lambda^{-1}u \ed a - \lambda^{-2}a\ed u -
\lambda^{-2}au^2p'(y) \ed y\Bigr)   \\
+ \lambda^{-1}u\cdot\der{\hmn}{a}\circ f_a^{-1} \ed a \quad\text{on $\hat{\mf{V}}_-$.}
\end{multline}

\head{Explicit formula describing the critical locus $\critLoc$.}

One then calculates

\begin{multline}
\name{formExpansion}
\ed \pplPlain^{-1} \wedge \ed \pmnPlain^{-d} \wedge\ed a= \\
\Biggl(
\bigl(\hpl + u \der{\hpl}{u}\bigr)
\Bigl(
-  \lambda^{-2}u^2p'(y)\cdot \hmn\circ f_a^{-1}
+u\lambda^{-1}\der{\hmn}{x}\circ f_a^{-1}
-\lambda^{-3}au^3p'(y) \cdot\der{\hmn}{v}\circ f_a^{-1}
\Bigr)\\
+
u \der{\hpl}{y}
\Bigl(
\lambda^{-2}\cdot \hmn\circ f_a^{-1}
+ au\lambda^{-3}\cdot\der{\hmn}{v}\circ f_a^{-1}
\Bigr)
\Biggr)
\ed u\wedge\ed y\wedge\ed a \\
\text{on $\hat{\mf{V}}_-\cap (\hat{V}_+\times \disk_R)$.}
\end{multline}

It is easy to show that the domain $\hat{\mf{V}}_-\cap(\hat{V}_+\times\disk_R)$
contains the plane $u=0$ in $\mb{P}^1\times\mb{P}^1\times \disk_R$.
\begin{corollary}
\name{cor-foliationsExtendToInfinity}
The Foliations $\Fpl{a}$ and $\Fmn{a}$ extend to holomorphic
foliations on $\hat{V}_+$ and $\hat{\mf{V}}_-$ respectively
and $\ppl{a}^{-1}$ and $\pmn{a}^{-d}$ are local defining
functions for these foliations. In both extended foliations
 the leaf through
$u=0,y=c$ is the line $u=0$ regardless of the value of $a$.
\end{corollary}
\begin{proof}
Clearly, for each $a\in \disk_R$, $\ppl{a}^{-1}$ and $\pmn{a}^{-d}$ define
the same foliations as $\ppl{a}$ and $\pmn{a}$ on $V_+$ and
on $\hat{\mf{V}}_-\cap (\umn{a}\times\{a\})$. Since $\pplPlain^{-1}$ and $\pmnPlain^{-d}$
are holomorphic functions on $\hat{V}_+\times\disk_R$ and $\hat{\mf{V}}_-$
which both vanish on $u=0$ by equations
\eqref{hplDef} and \eqref{hmnDef}
then we need only confirm that $\ed \ppl{a}^{-1}$ and $\ed\pmn{a}^{-d}$
do not vanish on $u=0$. But this follows from equations \eqref{dPplInverse}
and \eqref{dPmnInverse}.
\end{proof}

\head{Necessary computations to remove excess factors of $u$ from our formula defining $\critLoc$.}

\begin{lemma}
\name{computationToRemoveExcessFactors}
$\hpl\rest{u=0}=1$, and
$\der{\hpl}{y}$ vanishes in $u$ to order at least $d$. Similarly,
$\hmn\rest{v=0}=1$, and
$\der{\hmn}{x}$ vanishes in $v$ to order at least $d$.
\end{lemma}
\begin{proof}
From equation (\ref{hplDef}), equation (\ref{telescopingPl}) 
and the definition of $\spl{k}$ we know 
\begin{equation}
\name{grep}
\hpl=(1+\spl{1})^{-\frac{1}{d}}\exp\bigl(-\frac{1}{d^2}\log(1+\spl{2}) +\dotsb\bigr)
\end{equation}
We know from Lemma~\ref{vanishingInUpl} that if
$k\geq 2$ then $\spl{k}$ vanishes to order at least $d$ in $u$.

Since $\spl{1}=\inlinefrac{q(x)-ay}{x^d}=u^d q(\inlinefrac{1}{u})-ayu^d$ 
so $\der{\spl{1}}{y}=-au^d$,
it follows from equation (\ref{grep})
that $\der{\hpl}{y}$ vanishes to order at least $d$ in $u$.
The claim about $\hpl\rest{u=0}$ follows from equation \ref{grep}.

The results for $\hmn$ are proven similarly.
\end{proof}

\head{Removal of excess factors of $u$ from our formula defining $\critLoc$.}

\begin{corollary}
\name{cor-spuriousDoubleComponent}
The critical locus of the extensions of $\Fpl{a}$ and $\Fmn{a}$
in $(\hat{V}_+\times \disk_R)\cap \hat{\mf{V}}_-$ contains the plane $u=0$ with multiplicity two.
\end{corollary}
\begin{proof}
Since the $v$ coordinate of $f_a^{-1}(u,y,a)$ is $\inlinefrac{au}{up(y)-1}$
one concludes from Lemma~\ref{computationToRemoveExcessFactors}
that equation (\ref{formExpansion})
is divisible by $u^2$.
\end{proof}

Corollary~\ref{cor-spuriousDoubleComponent} is one of the
reasons these calculations are done so carefully. If we had
simply approximated the critical locus at infinity
using a Taylor series we could not make this conclusion.

While we need to extend $\critLoc$ to a variety on $(\hat{V}_+\times\disk_R)\cap \hat{\mf{V}}_-$,
the double component $u=0$ is spurious for our purposes. 
We let $\tilde{w}$ be the holomorphic function satisfying
$\ed \pplPlain^{-1} \wedge \ed \pmnPlain^{-d} \wedge\ed a=u^2
\tilde{w}(u,y,a) \ed u\wedge\ed y\wedge\ed a$.
Then the zero set of $\tilde{w}$ gives an extension of $\critLoc$
to $(\hat{V}_+\times \disk_R)\cap \hat{\mf{V}}_-$. We will abuse notation and
refer to the extended variety as $\critLoc$
as well. It will be obvious from the context whether we are using
the extension or not. Extending $\critLoc$ automatically extends $\critLoc_a$
to $(\hat{V}_+\times\disk_R)\cap \hat{\mf{V}}_-$ for each $a\in\disk_R$ since
$\critLoc_a$ is the zero set of $w_a(u,y)\equiv\tilde{w}(u,y,a)$.

\begin{lemma}
\name{tangentSpaceAtInfinity}
Given $a\in \disk_R$ and a point $(0,y_0,a_0)$ (written in $(u,y,a)$ coordinates)
the defining function for $\critLoc$ takes the form
$\tilde{w}(u,y,a)=-p'(y)-uH$ for some holomorphic function $H$ defined
in a neighborhood of $(0,y_0,a_0)$. Thus $(0,c,a)$
is in $\critLoc$ iff $c$ is a critical point
of $p$. Moreover if $c$ is an order one
critical point of $p$ then
$\critLoc$ is smooth at $(0,c,a)$ and
the tangent plane to $(0,c,a)$ is given by
\[p''(c)\ed y + C \ed u=0\]
for some $C$ depending upon $c$ and $a$.
\end{lemma}
\begin{proof}
From Lemma~\ref{computationToRemoveExcessFactors}
we can write: $\hmn\circ f_a^{-1}(u,y,a)=\hmn(y,\inlinefrac{au}{up(y)-1},a)=1+uH_2$,
$\der{\hmn}{x}\circ f_a^{-1}(u,y,a)=\der{\hmn}{x}(y,\inlinefrac{au}{up(y)-1},a)=u^d H_4$
and $\der{\hpl}{y}=u^dH_6$
for holomorphic functions $H_2, H_4$ and $H_6$.
From equation (\ref{formExpansion}) 
one obtains
$u^2\tilde{w}(u,y,a)=(1+uH_1)\Bigl(\{-u^2p'(y)(1+uH_2)\} +
uH_3\cdot u^dH_4 -  u^3H_5\Bigr)-u\cdot u^dH_6\cdot H_7=-u^2p'(y)-u^3H_8$
where each $H_i$ is a function which is holomorphic in a neighborhood
of the given $y_0, a\in \disk_R$ and for $u$ sufficiently close to zero.
Simplifying gives 
$\tilde{w}(u,y,a)=-p'(y)-uH_8$
for some holomorphic $H_8$. 
\end{proof}

\begin{theorem}
\name{thm-tangentSpaceAtInfinity}
If $c$ is an order one critical point of $p$
then the critical locus $\critLoc_a$ passes through
the point $(0,c,a)$ (written in $(u,y,a)$ coordinates)
and is smooth at this point. Moreover $\ppl{a}^{-1}\rest{\critLoc_a}\to\mb{C}$ 
is a local biholomorphism about this point. 
\end{theorem}
\begin{proof}
Smoothness follows directly from Lemma~\ref{tangentSpaceAtInfinity}.
Combining equation \eqref{dPplInverse} with Lemma~\ref{computationToRemoveExcessFactors}
gives $\ed \ppl{a}^{-1}\rest{\critLoc_a}(0,c,a)=\ed u$. By Lemma~\ref{tangentSpaceAtInfinity}
this is nondegenerate on the tangent line
of $\critLoc_a$ at $(0,c,a)$.
\end{proof}


%

  \newpage 
  
\section{Stable and Unstable Manifolds}

\name{sectionStableAndUnstableManifolds}

\subsection{Crossed Mappings}

\head{Definition of a crossed mapping.}

In this section we recall the definition of and basic results about crossed mappings from
\cite{hubbard_oberste-vorth:henon2}. This is a holomorphic version of a standard construction
of stable and unstable manifolds such as Theorem 6.4.9 criteria (4) in \cite{katok}.
We will only define degree one crossed mappings, and we will not consider
any crossed mappings of degree higher than one in this paper.

Let $B_1 = U_1 \times V_1$ and $B_2= U_2\times V_2$ be bidisks.

\begin{definition}
A (degree one) {\it crossed mapping} from $B_1$ to $B_2$ is a triple $(W_1, W_2, f)$, where
\begin{enumerate}
\item $W_1$ is an open subset of $U'_1 \times V_1$ for some open $U'_1\Subset U_1$.
\item $W_2$ is an open subset of $U_2 \times V'_2$ for some open $V'_2\Subset V_2$.
\item $f\colon W_1\to W_2$ is a holomorphic isomorphism, such that for all
$y\in V_1$ the mapping
\[\pr_1\circ f\colon W_1\cap (U_1\times \{y\})\to U_2\]
is a biholomorphism, 
and for all $x\in U_2$
the mapping
\[\pr_2\circ f^{-1}\colon W_2\cap (\{x\} \times V_2 )\to V_1\]
is 
a biholomorphism.
\end{enumerate}
\end{definition}

To make the notation less cumbersome, $f\colon B_1 \crmap B_2$ is often
written leaving the precise $W_1$ and $W_2$ to be determined by the context.

Given a hyperbolic polynomial
map $p$, Hubbard \& Oberste-Vorth construct a family of bidisks in $\mb{C}^2$ such
that if $\vert a\vert$ is sufficiently small then $f_a$ induces crossed mappings
between the bidisks of the family. They use this to get good hold on the stable
and unstable manifolds of $J_a$.

%
%
%
%

\begin{proposition}
\name{prop-GraphsUnderCrossedMappings}
If $f\colon B_1\crmap B_2$ is a crossed mapping of degree one and
$X\subset B_1$ is the graph of an analytic map from $U_1$ to $V_1$ then
the image of $X$ in $B_2$ is the graph of an analytic map from $U_2$ to $V_2$.
\end{proposition}
\begin{proof}
This is Proposition 3.4 of {\henII} for degree one crossed mappings.
\end{proof}

\head{Inverses and composition of crossed mappings are crossed mappings.}

\begin{proposition}
\name{propInvertingAndComposingCrossedMappings}
(a) Let $f\colon B_1 \crmap B_2$ be a crossed mapping 
of degree $1$. Then $f^{-1}\colon B_2\crmap B_1$ is also
a crossed mapping if all the coordinates are flipped.

\ssk
(b) If $B_1, B_2$ and $B_3$ are bidisks, $W_1\subset B_1$,
$W_2\subset B_2$, $\tilde{W}_2\subset B_2$ and $\tilde{W}_3\subset B_3$
and $f_1\colon W_1\to\tilde{W}_2$ and $f_2\colon W_2\to\tilde{W}_3$
are degree one crossed mappings, then
\[ f_2\circ f_1\colon f_1^{-1}(W_2)\to f_2(\tilde{W}_2)\]
is a degree one crossed mapping from $B_1$ to $B_3$.
%
%
\end{proposition}
\begin{proof} 
This is Proposition 3.7 of {\henII}.
\end{proof}

Suppose that $B_1,\dotsc,B_{n+1}$ are bidisks such that $B_i= U_i\times V_i$. Suppose also
that $W_i\subset V_i$, ($i=1,\dotsc,n$) and $\tilde{W}_i\subset B_i$ ($i=2,\dotsc,n+1$) are open
subsets so $f_i\colon W_i\to \tilde{W}_{i+1}$ are crossed mappings of degree 1. Let
\[ S^n_1 = W_1\cap f_1^{-1}(W_2)\cap\dotsb\cap (f_1^{-1}\circ\dotsb\circ f_{n-1}^{-1})(W_{n-1})\]
and
\[S^n_2 = \tilde{W}_{n+1} \cap f_n(\tilde{W}_n)\cap \dotsc\cap
(f_n\circ\dotsc\circ f_2)(\tilde{W}_2)\]
so that  that the restriction $g$
of $f_n\circ \dotsc\circ f_1$ to $S^n_1$ makes $g\colon S^n_1\to S^n_2$ a crossed
mapping of degree 1 from $B_1$ to $B_{n+1}$.

\head{Definition of size for use with crossed mappings.}

Let $U$ be a disk,
and $U'$ a relatively compact open subset. Define the {\it size} of $U'$ in $U$
to be $\inlinefrac{1}{M}$ where $M$ is the largest modulus of an annulus 
$U\setminus V$ for $V$ a compact contractible set containing $U'$.

\head{We introduce horizontal and vertical size of a crossed mapping.}

\begin{definition}
Given a crossed mapping $f\colon U_1\times V_1 \crmap U_2\times V_2$
we will let the {\it horizontal size} of $f$ be the size of $\pr_1(W_1)$ in $U_1$
and we will let the {\it vertical size} of $f$ be the size of $\pr_2(W_2)$ in $V_2$.
\end{definition}

%

\head{Stable manifold theorem for crossed mappings.}

\begin{proposition}
Let $B_0 = U_0\times V_0$, $B_1= U_1\times V_1$, $\dotsc$ be an infinite sequence
of bidisks, and $f_i\colon B_i\crmap B_{i+1}$ be crossed mappings of degree 1, with $U'_i$
of uniformly bounded horizontal size in $U_i$. Then the set
\[\Bigl\{\begin{bmatrix} x \\ y \end{bmatrix}\in B_0 \Big\arrowvert
    f_n\circ \dotsb \circ f_0\Bigl(\begin{bmatrix} x \\ y \end{bmatrix}\Bigr) \in B_n\ \text{for all $n$}\ \Bigr\}\]
is an 
analytic disk in $B_0$, which maps by $\pr_2$ isomorphically to
$V_0$,
which we will call the stable disk of the sequence of crossed mappings.
\end{proposition}
\begin{proof}
This is Corollary 3.12 of {\henII}.
\end{proof}

Similarly, when we have a backward sequence of crossed mappings
\[\dotsb \crmap B_{-1} \crmap B_0\]
with uniformly bounded vertical sizes, it will have an {\it unstable disk}, which
maps by $\pr_1$ isomorphically to $U_0$.

\subsection{Recalling Context and Constructions.}

\head{Definition of $U$.}

In Section 2 of {\henII} Hubbard \& Oberste-Vorth an open neighborhood $U$\name{defOfHobObU} of $J(p)\subset\mb{C}$
is constructed such that $U'\equiv p^{-1}(U)\Subset U$\name{defOfHobObUPrime}
and the map  $p\colon U'\to U$ is a covering map. This is a standard construction.
We can
assume without loss of generality that $U$ is chosen sufficiently
small that it is a finite distance from any critical points of $p$.
Open neighborhoods $U_z\subset U$ in $\mb{C}$ \name{defOfUz} are attached to each
point $z\in J(p)$. Later in section 4 of the same paper these neighborhoods are used
to associate an open subset $B_z$  of $\mb{C}^2$ to each
point $z\in J(p)$ so that
$f_a\colon  B_{z} \to B_{p(z)}$ is a crossed mapping for each $z\in J(p)$
and each $a$ with $\vert a\vert$  sufficiently small.  A radius $\rHubOb$ \name{defOfHubObr} is chosen
and $U_z$ is defined to be the ball of radius $\rHubOb$ about the point $z$,
where the distance is measured in the Kobayashi metric of $U$. The radius
$\rHubOb$ is only special in that it is chosen so be so small that 
the neighborhoods $U_z$ create telescopes for $p$ of uniformly bounded modulus.

We will strengthen this requirement on $\rHubOb$ a small amount here by requiring that $\rHubOb$ is
small enough that for each $z\in J(p)$ the map $p\cp{2}$ is a biholomorphism
of $U_z$ onto its image and
the map $p$ is a biholomorphism
from the ball of radius $3\rHubOb$ about $z$ onto its image. 

%

The neighborhoods $B_z$ are constructed in {\henII} by first selecting a small value $\delHO$ \name{defOfDelHO} (which must
satisfy various requirements). One then defines $v\colon\mb{C}^2\to\mb{C}$
by 
\begin{equation}
\name{equationForV}
v(x,y)\equiv p(y)-x.
\end{equation}
\name{defOfv}
Taking \name{defOfVPrime} 
\begin{equation}
V'\equiv \pr_1^{-1}(U)\cap v^{-1}(\disk_{\delHO})
\end{equation} 
then a well-defined function
$u\colon V'\to U'$ is constructed, which is given by $u(x,y)=p^{-1}(x)$, the inverse image
always being chosen to be the one ``close to y''.
We make this precise in Lemma~\ref{uDefinedPrecisely}.

\head{We define $\beta$ and prepare to make the definition of $u\colon V'\to U'$ precise.}


\head{We define $m_p$.}

We define $m_p$ to be the minimum of $\vert p'\vert$ on $\overline{U}$.
Since $\overline{U}$ is a finite distance from the critical set of $p$
then $m_p > 0$. 

\begin{lemma}
\name{radiusBetaBallMapsBihol}
There exists some $\beta>0$ \name{defOfBeta} such that
\begin{enumerate}
  \item $p$ maps the (Euclidean) ball of radius $\beta$ about an arbitrary point $y\in U$ biholomorphically
onto its image.
  \item $\beta$ is smaller than the Euclidean distance from $U'$ to $\partial U$.
  \item The Euclidean ball of radius $\beta/2$ about any point $u\in U'$
is mapped biholomorphically onto its image, which contains 
the Euclidean disk of radius $m_p\beta/8$.
\end{enumerate}
\end{lemma}
\begin{proof}
Part 1 follows from a straightforward proof by contradiction.
Part 2 is obvious. Part 3 follows from the Koebe $1/4$ theorem.
\end{proof}


\head{We shrink $\delHO$.}

We shrink $\delHO$ as necessary so that
$\delHO < \inlinefrac{1}{8}m_p \beta$.  We accordingly shrink the sets $V'$
and the boxes $B_z$.

\head{Precise definition of $u\colon V'\to U'$.}

\begin{lemma}
\name{uDefinedPrecisely}
There is a well defined holomorphic function $u\colon V'\to U'$
which satisfies $p(u(x,y))=x$ and $u(x,y)$ is the unique
preimage of $x$ within a distance $\inlinefrac{1}{2}\beta$
of $y$.
\end{lemma}
\name{defOfu}
\begin{proof}
That $u$ is well defined follows using Lemma~\ref{radiusBetaBallMapsBihol} and the definition of $V'$.
Holomorphy of $u$ follows from the fact that $p\colon U'\to U$
is a local biholomorphism.
\end{proof}

%

Then for each $z\in J(p)$ the set
\begin{equation}
B_z\equiv \biggl\{ \vt{x \\ y}\in V' \big\arrowvert u(x,y)\in U_z\biggr\}
\end{equation}
is an open neighborhood of the point $\bigl(p(z),z\bigr)$.

Hubbard \& Oberste-Vorth then prove that the mapping
\[\vt{x \\ y} \mapsto \vt{u(x,y) \\ v(x,y)} \]
is a biholomorphic isomorphism of $B_z$ onto the bidisk $U_z\times \disk_{\delHO}$,
where $u$ and $v$ are defined in Lemma~\ref{uDefinedPrecisely} and equation
\eqref{equationForV}.

This can be extended to all of $V'$.

\head{$(u,v)$ give coordinates on $V'$.}

\begin{lemma}
\name{uVCoordinates}
\[\vt{x \\ y}\mapsto \vt{ u(x,y) \\ v(x,y)} \]
is a biholomorphic isomorphism of $V'$ onto $U'\times \disk_{\delHO}$.
\end{lemma}
\begin{proof}
Using Lemma~\ref{radiusBetaBallMapsBihol} and Lemma~\ref{uDefinedPrecisely}
one can construct a holomorphic inverse.
\end{proof}

\subsection{The Stable and Unstable Manifolds as $a$ Varies}

Given that $p$ is a hyperbolic
polynomial, there is some $A>0$ \name{defOfA} such that $f_a\colon B_z\to B_{p(z)}$
is a crossed mapping for each $z\in J(p)$ whenever $0<\vert a\vert < A$.
It can be verified that $A$ can be chosen so that
these crossed mappings have uniformly bounded horizontal and vertical sizes. 
Given any $\nax{z}=(\dotsb,z_{-2},z_{-1},z_{0})\in \hat{J}(p)$ the sequence of neighborhoods
$\dotsb\to B_{z_{-2}}\to 
B_{z_{-1}}\to B_{z_0}\to B_{p(z_0)}\to B_{p\cp{2}(z_0)}\dotsb$ (all mapped by $f_a$)
form a sequence of crossed neighborhoods. Then for
each $\vert a\vert < A$ there is a natural map $\pi_a$ \name{defOfPiA} from
the natural extension $\hat{J}(p)$ \name{defOfNaxJ} of
the Julia set to $J_a$ which is a homeomorphism for each $a\not=0$
and is the standard projection of $\hat{J}(p)$ to $J(p)$ for $a=0$.


By definition, the point $\pi_a(\nax{z})$ is the unique point in $B_{z_0}$ which
lies in both the stable and the unstable manifold of the sequence of crossed mappings.
We note that given a sequence of crossed mappings $U_0\times V_0 \to U_1\times V_1 \to \dotsb$
then the construction
of  the stable manifold of a crossed mapping in {\henII} is found by first taking a point $u_n \in U_n$
and taking the preimage of $\{u_n\}\times V_n$ by $f_1\circ \dotsb \circ f_n$, which is, by the hypothesis
on crossed mappings, necessarily the graph of a function $g\colon V_1\to U_1$.
One then takes a limit of these graphs as $n\to \infty$.

\head{Holomorphic variation of stable manifolds about the degenerate map.}

We will use the notation $\tilde{f}(x,y,a)\equiv \bigl(f_a(x,y),a\bigr)$\name{defOfFTilde} when it is convenient 
to think of $f_a(x,y)$ as a self map in $x$, $y$ and $a$.

\begin{lemma}
\name{stableManMovesHol}
Given $\nax{z}\in \Nax{J}(p)$ there is a unique holomorphic map
$g_+\colon V_0\times \disk_A\to U_0$ such that the local stable
manifold through the point $\pi_{a'}(\nax{z})\in J_{a'}$, $a'\not=0$
is the graph of $g_+(\cdot,a')\colon V_0\to U_0$. In the case where $a=0$
this graph gives a vertical line through $\bigl(p(z),z\bigr)$ (the natural analogue of the stable
manifold in the degenerate case).
\end{lemma}
\name{defOfg}
\begin{proof}
As the maps $f_a$ depend holomorphically on $a$ then for $a\not=0$ these graphs
fit together holomorphically, since they are just the preimage of $\{u_n\}\times V_n\times \disk_A$
by $\tilde{f}\cp{n}$, which is holomorphic in both $(x,y)\in \mb{C}^2$ and in $a\in \disk_A$.  Thus,
the stable manifold of $f_a$ through $z\in J$ is the graph of a holomorphic
map $g_{z+}\colon V_0\times \disk_A^{*}\to U_0$. At $a=0$ the map is not a crossed mapping,
however $g_{z+}$ is holomorphic and clearly bounded on $V_0\times \disk_A^*$ and
so $g_{z+}$ has a unique holomorphic extension to a map (which we will still call $g_{z+}$)
from $V_0\times \disk_A$ to $U_0$. 

Now the function $u-g_{z+}(v,a)$ vanishes on the stable manifold and by continuity there 
is a neighborhood $N$ of $\pi_0\bigl(\hat{f}^{-1}(z)\bigr)$ 
such that $f_a(N)\subset U_0\times V_0$
for sufficiently small $a$. Therefore the pullback $\tilde{f}^*\bigl(u-g_{z+}(v,a)\bigr)$
by $\tilde{f}$ is defined in $N$ for all sufficiently small $a$. Since
the stable manifold of $\pi_a\bigl(\hat{f}^{-1}(\nax{z})\bigr)$ is mapped into the stable
manifold of $\pi_a(\nax{z})$ by $f_a$ then $\tilde{f}^*\bigl(u-g_{z+}(v,a)\bigr)$
vanishes on the graph 
$\check{g}_{\hat{f}^{-1}(\nax{z})+}\colon V_0\times \disk_A\to U_0\times V_0\times \disk_A$
of $g_{\hat{f}^{-1}(\nax{z})+}$. This is true even if $a=0$ because of continuity. 

Now the image of $f_0\circ \check{g}_{\hat{f}^{-1}(\nax{z})+}$ lies 
on $C(p)$ since $f_0$ maps all of $\mb{C}^2$ to $C(p)$.
If the image of 
$\check{g}_{\hat{f}^{-1}(\nax{z})+}$ does not lie in a fiber of $f_0$ then
it follows that $f_0\circ \check{g}_{\hat{f}^{-1}(\nax{z})+}$ contains an open
subset of $C(p)$ in its image. Therefore $u-g_+(v,0)$ would
have to vanish on $C(p)$. But this is impossible since $C(p)$
is given by $v=0$ in the $(u,v)$ coordinates and restricting
$u-g_+(v,0)$ to $v=0$ one obtains the false statement $u-g_+(0,0)\equiv 0$. 
This contradiction shows that $\check{g}_{\hat{f}^{-1}(\nax{z})+}$
lies in a fiber of $f_0$, that is, in a vertical line.
\end{proof}


\begin{observation}
Recalling that for each $\nax{z}=(\dotsb,z_{-2},z_{-1},z_0)\in\Nax{J}(p)$ the neighborhood
$B_{z_0}$ was equal to $U_{z_0}\times \disk_{\delHO}$ using $(u,v)$ coordinates,
where $U_{z_0}$ was an open neighborhood of $z_0\in J(p)\subset\mb{C}$
then we see that the local stable manifold of $\pi_a(\nax{z})$ given in
Lemma~\ref{stableManMovesHol} by a holomorphic function 
$g_{\nax{z}+}\colon V_0\to U_0$ is just a holomorphic map from 
$\disk_{\delHO}$ to $U_{z_0}$. What is more, this stable manifold was
precisely the stable manifold of the sequence $B_{z_0}\to B_{p(z_0)}\to B_{p\cp{2}(z_0)}\to\dotsb$
and was therefore dependent only on $z_0$ and not on any other point in the history $\nax{z}$.
Thus the map $g_{\nax{z}+}$ depends only on $z_0$. 
\end{observation}

\begin{convention}
In accordance with the above observation we will reduce our notation of
$g_{\nax{z}+}$ to $g_{z+}$ as $g_{\nax{z}+}$ only depends on the final
term $z$ of $\nax{z}$.
\end{convention}

\begin{convention}
\name{gConvention}
We will continue to use $\check{g}_{z\pm}$ to 
denote the graph in $V'\subset\mb{C}^2$
of $g_{z\pm}$. When we are thinking
of $g_{z\pm}$ as a function of $a$ as well,
we will similarly use $\check{g}_{z\pm}$ to
denote the graph in $V'\times \disk_A\subset \mb{C}^3$.
\end{convention}
\name{defOfgCheck}

\head{Holomorphic variation of unstable manifolds about the degenerate map.}

\begin{lemma}
\name{unstableManMovesHol}
Given $z\in \hat{J}(p)$ then there is a unique holomorphic map
$g_-\colon U_0\times \disk_A\to V_0$ such that the local unstable
manifold through the point $\pi_a(z)\in J_a$
is the graph of $g_-(\cdot,a')\colon U_0\to V_0$. For
$a=0$ it is simply the portion of the graph $x=p(y)$ about
$\pi_0(z)$.
\end{lemma}
\begin{proof}
Since for all $a\not=0$ these are just the images of
of some (arbitrary) $U_{-n}\times \{v_0\}$ under
$f_{-n}\circ \dotsb\circ f_1$ in $U_0\times V_0$
then by the hypothesis on a sequence of crossed mappings
the result is a graph $g_-\colon U_0\to V_0$. These graphs again all
fit together to form a single holomorphic ``sheet'' as when put together they
are simply the image of $U_{-n}\times \{v_0\}\times \disk_A$
iterated $n$ times by $f_a$ (which depends holomorphically on $a$).
For each $a$ the limit of these graphs will be a graph. Thus the limit is a graph, which is, 
by {\henII},
the local unstable manifold of $\pi_a(z)\in J_a$ for each $\vert a\vert < A$.
In the case $a=0$ it is simply a portion of the graph $x=p(y)$ about the
point $\pi_0(z)$ (and is therefore the appropriate version of the unstable manifold
for this case).
\end{proof}

We recall Theorem 5.9 of \cite{bedfordsmillie1}.
\begin{theorem}
Given that $a\not=0$, if $f_a$ is hyperbolic and $\vert a\vert\leq 1$, then
$W^s(J_a)=\jpl{a}$. If $s_1,\dotsc,s_k$ are the sinks
of $f_a$ then $W^u(J_a)=\jmn{a}\setminus\{s_1,\dotsc,s_k\}$.
\end{theorem}

\begin{definition}
\name{def-localStableMan}
Given $z\in J(p)$ let $\Delta_{z,a}$ be the image $\hat{g}_{z+}(\disk_{\delHO},a)\subset \mb{C}^2$
which is precisely the local stable manifold in $B_z$ corresponding to the sequence of crossed 
mappings $B_z\to B_{p(z)}\to B_{p\cp{2}(z)}\to \dotsb$.
Given $r < 1$ we let $\Delta_{z,a}(r)$ be the image of $\disk_{r\delHO}$ under
$\hat{g}_{z+}$, and we let $\Sigma_{z,a}(r)$ be the image of the circle $S_{r\delHO}$
under $\hat{g}_{z+}$. Hence both $\Delta_{z,a}(r)$ and $\Sigma_{z,a}(r)$ lie
in $\Delta_{z,a}$.
\end{definition}
\name{defOfLocalStableMan}

\head{The Stable Manifolds}

We will show that given $z_1,z_2\in J(p)$, 
if $\Delta_{z_1,a}$ and $\Delta_{z_2,a}$ operlap 
then $z_1=z_2$. First we recall the standard telescope lemma.


\begin{lemma}
\name{usingTelescopeDisksToDistinguishPoints}
If $z_1$ and $z_2$ are points in $J(p)$ and if
$U_{p\cp{n}(z_1)}\cap U_{p\cp{n}(z_2)}\not =\emptyset$ for every $n\geq 0$
then $z_1=z_2$.
\end{lemma}
\begin{proof}
This the standard telescope result.
\end{proof}
%
%

We now show our desired result about disjointness of the sets $\Delta_{z,a}$.

\head{The local stable manifolds $\Delta_{z,a}$ are disjoint.}

\begin{lemma}
\name{stableManifoldsDisjoint}
Assume that for $z_1, z_2\in J(p)$ one has
$\Delta_{z_1,a}\cap \Delta_{z_2,a}\not=\emptyset$.
Then $z_1=z_2$.
\end{lemma}
\begin{proof}
Assume that $\Delta_{z_1,a}\cap \Delta_{z_2,a}\not=\emptyset$.
Let $w$ be a point in the intersection. Then 
$f_a\cp{n}(w)\subset B_{p\cp{n}(z_1)}$ and
$f_a\cp{n}(w)\subset B_{p\cp{n}(z_2)}$ 
for all $n\geq 0$. Recalling that $B_z=U_z\times \disk_{\delHO}$
(in the $(u,v)$ coordinates defined on $V'$) then one has
$u\bigl(f_a\cp{n}(w)\bigr)\in U_{p\cp{n}(z_1)} \cap U_{p\cp{n}(z_2)}$
for all $n\geq 0$ and so $z_1=z_2$ by Lemma~\ref{usingTelescopeDisksToDistinguishPoints}.
\end{proof}

\begin{lemma}
\name{stableMansVaryCont}
The maps $g_{z_0+}\colon \disk_{\delHO}\times \disk_A \to U_{z_0}$
vary continuously (in the sense of locally uniform convergence of maps)
with $z_0\in J(p)$.
\end{lemma}
\begin{proof}
From Lemma~\ref{stableManifoldsDisjoint} 
the functions $g_{z_0+}\colon\disk_{\delHO}\times \disk_A\to U'$ 
have disjoint graphs from which the result follows.
\end{proof}



\begin{lemma}
\name{jMovesCont}
The maps $\pi_a\colon \Nax{J}(p)\to\mb{C}^2$ vary continuously
with $a$ in the sense of uniform
convergence of maps for $a\in \disk_A$.
\end{lemma}
\begin{proof}
This is clear for nonzero $a$. Because
$\pi_a(z)$ varies holomorphically with $a$ it follows easy
that one also has continuity at $a=0$.
\end{proof}

%
%
%

\head{The closure of $\critLoc$ in the zero slice lies in $\critLoc_0$,
$\jpl{0}$ or $\jmn{0}$.}

\begin{lemma}
\name{closureOfCriticalLocusInZeroSlice}
If $(x,y,0)\in\overline{\critLoc}$ then either $(x,y)\in\overline{\critLoc_0}$
or $x\in J(p)$ or $(x,y)\in C(p)$.
\end{lemma}
\begin{proof}
If $z\in\overline{\critLoc}$ but $z\not\in\overline{\critLoc}_0$ then
since $\critLoc$ is a closed analytic variety on its domain of definition
then $z\in (K(p)\times\mb{C})\cup C(p)$. Now if $z$ lies in the interior
of $K(p)\times \mb{C}$ then $z$ is attracted to the cycle of an attracting periodic point
$\alpha$ of $p$. 
It follows that $z$ lies in the interior of $\kpl{a}$ for all
sufficiently small $a$, which is a contradiction.
Therefore $z\in J(p)\times \mb{C}$ or $z\in C(p)$.
\end{proof}

\head{Description of the critical locus in the degenerate case.}

It is easy to see what the critical locus $\critLoc_0$ looks like if $J(p)$
is connected. 
Since $\ppl{0}(x,y)=b_p(x)$,
the B\"ottcher coordinate of $x$, then the leaves of $\Fpl{0}$
are simply vertical. Since $\pmn{0}(x,y)=\bigl(p(y)-x\bigr)^{\frac{1}{d}}$,
then the leaves of $\Fmn{0}$ are translates in the $x$ direction of the graph
$C(p)$. Thus $\Fpl{0}$ and $\Fmn{0}$ will be tangent exactly along the horizontal lines $y=c$ where $c$
is a critical point of $p$. Thus $\critLoc_0$ is a union of horizontal lines
(restricted to $U_{0+}\cap U_{0-}=\mb{C}^2\setminus \bigl( C(p) \cup (K(p)\times\mb{C})\bigr)$),
one at the level of each critical point of $p$.  
This fact will be quite important in what follows.

\begin{lemma}
\name{degenerateCriticalLocus}
If the Julia set of $p$ is connected then
$\critLoc_0$ is the union of the sets $\{(x,c)\vert x\not\in K(p)\}$
over the critical points $c$ of $p$.
\end{lemma}

\head{The need to control the critical locus at the boundary when we perturb from the degenerate case.}

When we perturb $a$ away from zero, we will need to be able to control the motion
of the critcal locus. The main difficulty is that we must control what happens
at the boundary of $U_{0+}$ and $U_{0-}$. We will do this
by choosing a tube about each of horizontal lines in $\critLoc_0$ which 
contains the perturbed component of $\critLoc_a$ as $a$
moves away from zero. Of course, that the different components of $\critLoc_0$
remain distinct components when $a$ is perturbed will have to be shown e.g. consider the
equations $y(y-1)x + a=0$. When $a=0$ this is a curve which has three components, but
as $a$ is perturbed to a nonzero value this becomes
a smooth curve with only one component. Since $\critLoc_0$ is not
defined outside of $U_{0+}\cup U_{0-}$ it is conceivable that such a thing could happen
to $\critLoc_a$ as $a$ is varied from zero, i.e. we could have a large portion of $\critLoc_a$
which is ``hidden'' in the boundary of definition when $a=0$. 
We have to demonstrate that such oddities do not occur.

We will make use of
the following version of the ``Inclination Lemma'' about
the degenerate map $f_0$. Compare to \cite{katok} and \cite{ruelle}.\footnote{This
lemma is also sometimes called the $\lambda$-lemma, but we avoid this term
because the term $\lambda$-lemma is typically used in complex
dynamics to refer to a result about holomorphic motions. See \cite{lambda}, \cite{Lyu83}.}

\begin{lemma}[An inclination lemma near the degenerate case.]
\name{continuityOfPlLaminations} Given a sequence $a_k\to 0$
and a sequence of points  $w_k\in \critLoc_{a_k}$ converging to some
point $w_\infty\in \jpl{0}$ then the the leaves $\lfpl{a_k}(w_k)$
converge locally and without ramification to a vertical line
through $w_\infty$ (i.e. to the leaf of $\jpl{0}$ through
$w_\infty$). 
\end{lemma}
\begin{proof}[Sketch of Proof]
We construct a neighborhood $\mc{N}$ of $J_0$ in $\mb{C}^3$
such that each point $w$ of $\mc{N}$ lies in a box
$B_{\mf{z}(w)}$ and such that if $w$ and $f_a(w)$ both
lie in $\mc{N}$ then $f_a\colon B_{\mf{z}(w)}\to B_{\mf{z}\bigl(f_a(w)\bigr)}$
is a crossed mapping.

We then apply iterates of $f_a$ to each member of our sequence to 
move all the points of the sequence a definate distance
away from $\kpl{}$. Then the leaves of $\Fpl{}$ through the new
sequence must converge to the leaves of $\Fpl{0}$, and
must therefore become vertical lines in the limit.

We then pull back the leaves by iteration so that they
pass through the members of the original sequence. By using 
the boxed mapping construction, we can guarantee that these leaves
are still graphs in their respective boxes $B_{z(w_k)}$.
It is then easy to show that the leaves become vertical in the limit.

Once this is established, 
one can apply this argument to $f_{a_k}(w_k)$
and obtain a sequence of plaques which become vertical in the limit.
Taking the preimages of these plaque under $f$ to obtain 
a sequence of plaques through the points $w_k$, 
it is easy to show that
the resulting plaques become vertical over arbitrarily large sets.
\end{proof}

  \newpage

\section{Components of the Critical Locus.}

\name{sectionComponentsOfTheCriticalLocus}

\subsection{Trapping and Mapping Components.}
\name{subsection-trappingAndMappingComponents}

\begin{standingAssumption}
Throughout the rest of this paper we will make the additional assumption
that the orbits of each of the critical points remain bounded, that is, that
$J(p)$ is connected. We also assume that all of the critical points of $p$ are
simple.
\end{standingAssumption}

\head{We define $H_{ca}$.}

\begin{definition}
\name{def-primaryHorizontalComponent}
Given a critical point $c$ of $p(x)$ we 
let $H_{ca}$ be the component of $\critLoc_a$ which is asymptotic to $y=c$ as $\vert x\vert \to\infty$.
We call $H_{ca}$ the {\it primary horizontal component} of the 
critical locus corresponding to the critical point $c$ of $p(z)$.
\end{definition}
\name{defOfH}

We will show that if $a$ is sufficiently small then given
distinct critical points $c_1,c_2$ then $H_{c_1,a}$ and $H_{c_2,a}$ are
distinct and disjoint components of $\critLoc_a$.
Note that we know from Lemma~\ref{degenerateCriticalLocus} that 
$H_{c0}=\{(x,c)\,\vert\ x\not\in K(p)\}$.

\head{We construct a tube in which to contain $H_{ca}$.}

For each critical point $c$ of $p(z)$ we select an open 
disk $\tubeYFilled_c\subset \mb{C}$ about $c$. We assume
these disks are chosen small enough to have disjoint closures.
We let $\tubeXHole_c$ be an open disk about $p(c)$, sufficiently small that
$\overline{\tubeXHole}_c$ lies in the basin of an attracting periodic point.
We let $\tubeX_c\equiv \mb{C}\setminus{\overline{\tubeXHole}_c}$.

If necessary, we shrink $\tubeYFilled_c$ so that $\tubeX_c\times \tubeYFilled_c$ 
is a positive distance from $C(p)$. 
For each critical point $c$ we let $\tubeYCore_c$ be a second disk about $c$
of half the radius of $\tubeYFilled_c$. We let 
$\tube_c=\tubeX_c\times (\tubeY{c})$.

We think of $\tube_c$ as a thin hollow tube about $\{y=c\}$
as pictured in Figure~1. 
We will show that for small values of
$a$ the component $H_{ca}$ of $\critLoc_a$  
remains inside this tube. We let $\tubeCore_c=\tubeX_c\times \tubeYCore_c$ be the 
core of the tube $\tube_c$ and we let $\tubeFilled_c=\tubeX_c\times \tubeYFilled_c$
be the filled tube.


\begin{figure}
\name{tubePicture}
\centerline{\includegraphics{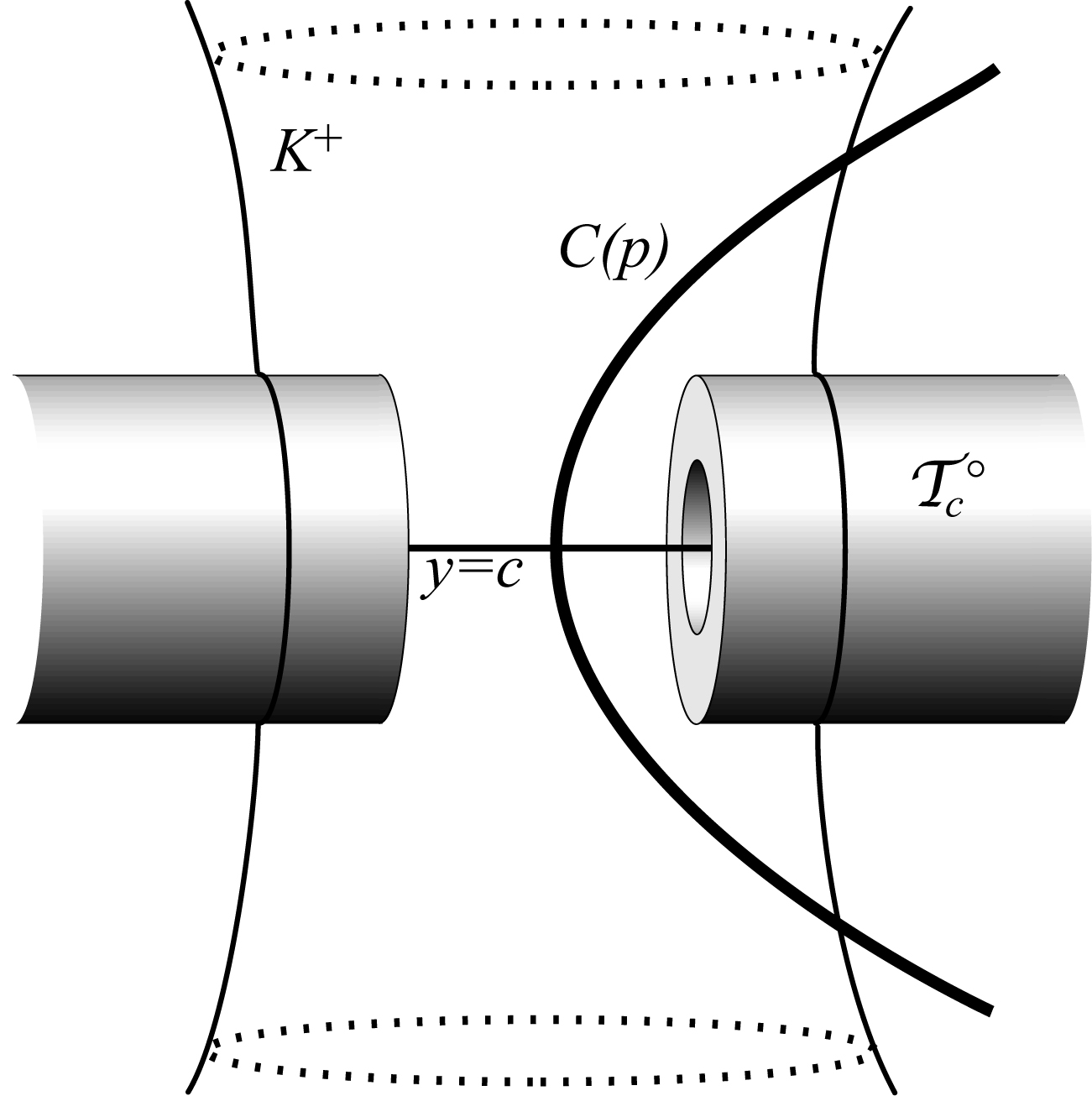}}
\caption{The tube $\tube_c$ and its environment.}
\end{figure}

\begin{lemma}
\name{tubeToolbox}
There exists some $\delta_T > 0$ and some $\epsilon_T >0$ such that if
$\vert a\vert < \epsilon_T$ then:
\begin{enumerate}
\item \name{part-filledTubeAvoidsKmn}
The filled tube $\tubeFilled_c$ is distance at least $\delta$ from
the set $\kmn{a}$ for each critical point $c$ of $p(z)$.
\item \name{part-trapHoleInKpl}
One has $\overline{\tubeXHole_c \times \tubeYFilled_c}\subset \kpl{a}$
for each critical point $c$.
\item \name{part-trappedInTube}
For each critical point $c$, $H_{ca}\subset\tubeCore_c$
and this is the only component of $\critLoc_a$ which intersects
$\overline{\tubeFilled_c}$.
\item \name{part-controllingTangenciesInTheTube}
The foliations $\Fpl{a}$ and $\Fmn{a}$ have contact of order two at any point of $H_{ca}$
for any critical point $c$. In particular, $H_{ca}$ is smooth, has
multiplicity one, and is everywhere
transverse to both $\Fpl{a}$ and $\Fmn{a}$.
\end{enumerate}
\end{lemma}
\begin{proof}
\ref{part-filledTubeAvoidsKmn}.
This is a trivial consequence of the construction of 
$\tubeFilled_c$ and Lemma~\ref{bigUPlAndbigUMnOpen}.

\smallskip
\ref{part-trapHoleInKpl}.
This follows from the fact that 
$\overline{\tubeXHole_c \times \tubeYFilled_c}$ is relatively compact
in a basin of attraction of $f_0$.

\smallskip
\ref{part-trappedInTube}.
We prove first that if $\vert a\vert$
is sufficiently small then for each critical point $c$, 
any component $X$ of $\critLoc_a$ which intersects
$\overline{\tubeFilled_c}$ lies in $\tubeCore_c$.
Hence, for the sake of contradiction, assume 
that there exists a sequence $a_i\to 0$ and a
sequence $z_i$ such that $z_i\in \critLoc_{a_i}\cap \overline{\tube_c}$.
It follows from Lemma~\ref{tangentSpaceAtInfinity}
that the sequence $z_i$ is bounded. Letting $z_\infty$ be any accumulation
point of $z_i$ then either $z_\infty\in \upl{0}$ or
$z_\infty\in \jpl{0}$. One concludes from Lemma~\ref{pmnForZeroA}
combined with either Lemma~\ref{pplForZeroA} in the former case or with
Lemma~\ref{continuityOfPlLaminations} in the latter case that 
$z_\infty\in\bigcup \overline{H_{c0}}$. However this is a contradiction
since $\overline{\tube_c}$ is disjoint from each $\overline{H_{c0}}$.
Thus $\critLoc_a\cap\overline{\tube_c}=\emptyset$
for $\vert a\vert$ sufficiently small.

Now assume $X$ is a component of $\critLoc_a$ intersecting $\overline{\tubeFilled_c}$.
Since $X$ is disjoint from $\overline{\tubeXHole_c\times \tubeYFilled_c}$ by
part~\ref{part-trapHoleInKpl} and 
from $\overline{\tube_c}$ then it follows easily from the definitions of these sets that 
$X\subset (\mb{C}\setminus \overline{\tubeXHole_c})\times \tubeYCore_c\subset \tubeCore_c$.
We will show that $X$ must be $H_{ca}$. 

Choose some sequence of points $z_k\in X$ 
such that
$$
  \lim_{k\to\infty} \gpl{a}(z_k)=\sup_{z\in X} \gpl{a}(z).
$$
A straightforward argument by contradiction shows that $z_k\to\infty$.
By Lemma~\ref{tangentSpaceAtInfinity}, $z_k$ converges to the point 
of the extension of $\critLoc_a$ to $\mb{P}^1\times\mb{P}^1\times\disk_R$
with $(u,y,a)$ coordinates $(0,c,a)$. 
From Theorem~\ref{thm-tangentSpaceAtInfinity}, $X=H_{ca}$.

\smallskip
\ref{part-controllingTangenciesInTheTube}.
If there is no $\epsilon_T$ such as the lemma claims to
exist then one can choose a critical point $c$ of $p(x)$ and sequences
$a_k\to 0$ and $z_k\in H_{ca_k}$ such that $\lfpl{a_k}(z_k)$ and 
$\lfmn{a_k}(z_k)$ have order of contact at least three 
at $z_k$ for each $i$. We can assume without loss of generality
that $\vert z_k\vert\to\infty$ or that $z_k$ converges
to some point $z_\infty\in\overline{\tubeCore_c}$.

If $\vert z_k\vert\to\infty$ then
we will need to use the extension of $\critLoc$ to $x=\infty$,
hence we change coordinates letting $u=1/x$.
Since $z_k\in \critLoc_{a_k}$ by assumption and the leaves
of $\Fpl{a_k}$ and $\Fmn{a_k}$ have contact of 
order at least three at $z_k$ then  
Lemma~\ref{intersectionComparison} shows that 
$\Fpl{a_k}$ and $\critLoc_{a_k}$ have intersection
multiplicity at least two at $z_k$. Since $\critLoc_{a_k}$
is defined by $w(u,y,a_k)$ for $a_k$ and $u$ sufficiently
small and $y$ bounded then the directional derivative
of $w(u,y,a)$ along the leaf of $\Fpl{a_k}$ through $z_k$
is zero by Lemma~\ref{intersectionLemma}. But 
$\frac{\partial w}{\partial y}(0,c,0)=0$ since this is the
directional derivative along the leaf of $\Fpl{0}$ by 
Corollary~\ref{cor-foliationsExtendToInfinity}.
But this contradicts the expression for $w$ about
$(0,c,0)$ given in Lemma~\ref{tangentSpaceAtInfinity}.

Having obtained a contradiction if $\vert z_k\vert\to\infty$,
assume $z_k\to z_\infty=(x_\infty,y_\infty)$. By part~\ref{part-filledTubeAvoidsKmn} we know
$z_\infty\not\in\kmn{0}$, so $(x_{\infty},y_{\infty},0)\in\mf{V}_{0,-}$. 
By Lemma~\ref{inflatingKplZero} 
we know $z_\infty\in \overline{\mb{C}^2\setminus \kpl{0}}$.
Hence $\pmnPlain^{d^j}$ is well defined at $(x_\infty,y_\infty,0)$
for some integer $j$. Then by  
either Lemma~\ref{continuityOfPlLaminations} or
Lemma~\ref{pplForZeroA}, depending on
whether $z_\infty\in \jpl{0}$ or not, one concludes that
one can parameterize plaques of the leaves $\lfpl{a_k}(z_k)$
such that they converge to a parameterization of a plaque of the leaf
of $\jpl{0}$ or $\Fpl{0}$ through $z_\infty$ (which is a vertical
line in either case). Since $\Fpl{a_k}$
and $\Fmn{a_k}$ have contact of order at least three at $z_k$ for each $k$ then 
by Lemma~\ref{intersectionLemma} the directional derivative of 
$\pmn{a_k}^{d^j}$ along the leaf $\lfpl{a_k+}(z_k)$ is zero. But by continuity
of $\pmnPlain^{d^j}$ it follows that the directional derivative of $\pmn{0}^{d^j}$
in the $y$ direction is zero. This contradicts the expression
for $\pmn{0}^{d^j}$ given in Lemma~\ref{pmnForZeroA}. Hence
we conclude that there is no such sequence $a_k$ and $z_k$
and so $\Fmn{a}$ and $\Fpl{a}$ have contact of order two 
at every point of $H_{ca}$ whenever $\vert a\vert$ is sufficiently small.

The rest of the Lemma is an immediate consequence of
Corollary~\ref{cor-leafContactToolbox}.
\end{proof}

\begin{lemma}
\name{gplProperOnH}
If $\vert a\vert < \epsilon_T$ then
the map $\gpl{a}\rest{H_{ca}}\colon H_{ca}\to (\frac{1}{d-1}\log\vert a\vert,\infty)$ is proper.
\end{lemma}
\begin{proof}
This is easy to prove using parts~\ref{part-filledTubeAvoidsKmn}
and \ref{part-trappedInTube} of Lemma \ref{tubeToolbox}.
\end{proof}

\head{Relating the index of the gradient to the order of a critical point.}

The following is a standard fact:

\begin{lemma}
\name{gradientsAndCriticalPoints}
If $M$ is a Riemann surface and if $g$ is harmonic on $M$ then
\begin{itemize}
\item The zeros of $\nabla g$ are discrete
\item The zeros of $\nabla g$ correspond to critical points of a holomorphic $h$ 
such that $g=\re h$ 
and the index of the zero is equal the negative of the order of  the critical point.
\end{itemize}
\end{lemma}

\begin{theorem}
\name{thm-extensionToH}
If $\vert a\vert < \epsilon_T$ then $H_{ca}$ is a punctured disk and
the map $\ppl{a}$ extends holomorphically from from $H_{ca}\cap V_+$
to a biholomorphism $\ppl{a}\colon H_{ca}\to\mb{C}\setminus \overline{\disk}$.
\end{theorem}
\begin{proof}
Using Lemma~\ref{gradientsAndCriticalPoints} and 
part~\ref{part-controllingTangenciesInTheTube} of Lemma~\ref{tubeToolbox}
we can conclude that $\gpl{a}$ has no critical points 
on $H_{ca}$. Moreover, by Lemma~\ref{gplProperOnH},
$\gpl{a}$ is proper.
By Theorem~\ref{thm-tangentSpaceAtInfinity}, Corollary~\ref{corBoundednessPl}
and the definition of $\gpl{a}$ it follows that that the fibers of $\gpl{a}\rest{H_{ca}}$ 
about the point $(\infty,c)\in H_{ca}$ are topological circles and, by Morse theory, $H_{ca}$
is a topological annulus. Since $H_{ca}$ contains a punctured disk about $(\infty,c)$ then
$H_{ca}$ is either $\disk^*$ or $\mb{C}^*$.

The function $\ppl{a}\circ f_a\cp{k}\colon H_{ca}\to \mb{C}\setminus\overline{\disk}$ 
induces the map 
$j\mapsto d^k\cdot j$ on $\pi_1(H_{ca})\to \pi_1(\mb{C}\setminus\overline{\disk})$
by Corollary~\ref{corBoundednessPl}, Theorem~\ref{thm-tangentSpaceAtInfinity} and
the recursion relation for $\ppl{a}$. It follows that 
$\ppl{a}\circ f_a\cp{k}$ has a holomorphic $d^{k\text{th}}$ root which
is equal to $\ppl{a}$ on $H_{ca}\cap V_+$. 
It then follows that $\ppl{a}$ has a holomorphic extension to all of $H_{ca}$.
What is more, $\ppl{a}$ is proper since $\gpl{a}$ is. By considering $\ppl{a}$ about
$(\infty,c)$ it follows easily that $\ppl{a}\rest{H_{ca}}\colon H_{ca}\to \mb{C}\setminus\overline{\disk}$
has degree one and is therefore a biholomorphism.
\end{proof}

Given $a,b$ with $\vert a\vert < \epsilon_T $
and $\vert b\vert < \epsilon_T$ we define a biholomorphism 
$\tau_{ab}\colon H_{ca}\to H_{cb}$ by 
$\tau_{ab}=\biggl(\ppl{b}\rest{H_{cb}}\biggr)^{-1}\circ \ppl{a}$.
Then $\ppl{a}\circ \tau_{ab}=\ppl{b}$.

\begin{proposition}
\name{holomorphicVariationOnCritical}
The maps $\tau_{ab}\colon H_{ca}\to H_{cb}$ vary holomorphically
in $a$ and $b$.
\end{proposition}
\begin{proof}
The precise meaning of this is that if one defines $\ms{H}_c\subset\critLoc$
to be $\{(x,y,a)\vert (x,y)\in H_{ca}, a\in\disk_{\epsilon_T}\}$ then from 
Lemma~\ref{tangentSpaceAtInfinity} and 
part~\ref{part-trappedInTube} of Lemma~\ref{tubeToolbox}
it is clear that $\ms{H}_c$ is a component
of $\critLoc\cap (\mb{C}^2\times \disk_{\epsilon_T})$ and 
this proposition states that the map 
$\tilde{\tau}\colon \ms{H}_c\times \disk_{\epsilon_T}\to \ms{H}_c$
given by $\tilde{\tau}(x,y,a,b)\equiv \tau_{a,b}(x,y)$ 
is holomorphic.

The proof is elementary since from our original construction the function
$\ppl{a}(x,y)\colon V_+\times \disk_R\to \mb{C}$ is holomorphic
in $x$, $y$ and $a$. It was shown in the proof of Theorem~\ref{thm-extensionToH}
that $\ppl{a}\circ f\cp{k}\rest{H_{ca}}$ 
has a $d^{k\text{th}}$ root which agrees with $\ppl{a}$ when $\vert a\vert < \epsilon_T$
and this root gives the extension of $\ppl{a}$ to $f^{-k}(V_+)\cap H_{ca}$.
It follows that $\ppl{a}\circ f_a^k(x,y)\colon 
\ms{H}_c\cap \{(x,y,a)\vert f_a\cp{k}(x,y)\in V_+,a\in\disk_{\epsilon_T}\}\to \mb{C}$
has a $d^{k\text{th}}$ root which agrees with $\ppl{a}(x,y)$ on $\ms{H}_c\cap (V_+\times \disk_{\epsilon_T})$.
Consequently the extension $\ppl{a}(x,y)\colon \ms{H}_c\to \mb{C}\setminus\overline{\disk}$
is holomorphic in $x$, $y$ and $a$. It is easy to see that 
the map $\widetilde{\ppl{a}}\equiv \bigl(\ppl{a}(x,y),a\bigr)\colon
\ms{H}_c\to  (\mb{C}\setminus\overline{\disk})\times \disk_{\epsilon_T}$ is a biholomorphism.
It follows that $\tilde{\tau}_{ab}$ is holomorphic in $a$ and $b$ from the easily
verified relationship
$\widetilde{\ppl{b}}\circ \tilde{\tau}_{ab}=\widetilde{\ppl{a}}$.
\end{proof}

\subsection{Classification of the Critical Components.}

Since our strategy has been to consider the degenerate map $f_0$ and
then to consider $\critLoc_a$ as a deformation of $\critLoc_0$, we need to ensure
that we have accounted for every component of $\critLoc_a$. It is plausible
that $\critLoc_a$ has a component that ``escapes to infinity'' as $a\to 0$, and thus
this component would be invisible to us in $\critLoc_0$. 
We will start by showing that any component of $\critLoc_a$
meets either $\jmn{a}$ or $\jpl{a}$.
We will be able to use this to show that
any component of $\critLoc_a$ is an iterate of a component of the form $H_{ca}$, and
thus we have accounted for every component of $\critLoc_a$ by accounting for the components
$H_{ca}$ and their iterates.

\head{Ensuring that components of $\critLoc_a$ abut against $\jpl{a}$ or $\jmn{a}$.}

\begin{lemma}
\name{abutmentForComponents}
If $W_a$ is some component of $\critLoc_a$ then $\partial W_a$
contains a point in either $\jpl{a}$ or $\jmn{a}$.
\end{lemma}
\begin{proof}
Consider the positive plurisubharmonic function
$\mf{g}(x,y,a)\equiv \gpl{a}(x,y) + \gmn{a}(x,y) - \dfrac{1}{d-1}\log\vert a\vert$ on
$\bigUPl\cap \bigUMn$. It is easy to show that if $z_n$ is a sequence
of points of $W_a$ such that 
$\lim_{n\to\infty} \mf{g}(z_n)=\inf_{z\in W_a}\mf{g}(z)$
then $z_n$ has an accumulation point in $\jpl{a}\cup\jmn{a}$.
\end{proof}

We will need specific local stable manifolds about the points of $J_a$. 
We know that if $\vert a\vert$ is sufficiently small then given  
$\nax{z}=(\dotsb,z_{-2},z_{-1},z_0)\in \Nax{J}(p)$
there is an associated neighborhood $B_{z_0}=V_0\times U_{z_0}$ of $\pi_a(\nax{z})$
and the local stable manifold in $B_{z_0}$ is the graph of a holomorphic function
from $V_0\to U_{z_0}$.

We recall that in Lemma~\ref{uVCoordinates} it was shown that
$(u,v)$ coordinates are defined on an open set $V'$ which contains
each of the sets $B_z$ and that $(u,v)$ coordinates provide a biholomorphic
isomorphism of $V'$ onto $U'\times \disk_{\delHO}$. Also, since
$v(x,y)=p(y)-x$ then $v$ is defined on all of $\mb{C}^2$, not just on $V'$.

\head{Small tube about $C(p)$ a finite distance from filled trapping tubes.}


\begin{definition}
\name{def-r}
We now fix some positive $r < 1$ such that
each of the filled tubes $\tubeFilled_c$ is a finite distance from the
set $\{(x,y)\in\mb{C}^2\big\arrowvert\ \vert v(x,y)\vert \leq r\delHO\}$.
\end{definition}
\begin{proof}[Proof that such an $r$ exists.]
Each of the filled tubes lies a finite distance from $C(p)$ by construction.
We let $s$ be half the minimal distance between $C(p)$
and the nearest tube.
Since $\bigl( x-v(x,y),y\bigr)\in C(p)$ for all $(x,y)\in\mb{C}^2$
then the set $\{(x,y)\big\arrowvert\ \vert v(x,y)\vert \leq s\}$
is comprised of points no further than $s$ from $C(p)$. Thus $r=s/\delHO$ will do.
\end{proof}

\head{Bounding $\gmn{a}$ on local stable manifolds.}

\begin{lemma}
\name{controlGmnOnStableDisks}
Given $\epsilon > 0$ there exists $\delta_r(\epsilon) > 0$ such that if $\vert a\vert < \delta_r(\epsilon)$
and if $(x,y)\in V'\subset \mb{C}^2$ then
\begin{itemize}
\item $\vert v(x,y)\vert \leq r\delHO$ implies that 
$\gmn{a}(x,y) < \dfrac{1}{d}\log\vert r\delHO\vert + \epsilon$
\item $\vert v(x,y)\vert\geq r\delHO$ implies that 
$\dfrac{1}{d}\log\vert r\delHO\vert - \epsilon < \gmn{a}(x,y)$
\end{itemize}
\end{lemma}
\begin{proof}
This is an easy consequence of Theorem~\ref{thm-greensContinuous}.
\end{proof}

\head{Gradient of $\gmn{a}$ on local unstable manifold not tangent to boundary of crossed mapping block.}

\begin{lemma}
\name{nontangencyCondition}
There exists $\epsilon_S > 0$ such that if $\vert a\vert < \epsilon_S$ and $z\in J(p)$ 
then the gradient of the
restriction of $\gmn{a}$ to the local stable manifold $\check{g}_{z+}(\disk_{\delHO},a)$ is defined
and nonzero on the curve $\check{g}_{z+}(S_{r\delHO},a)$. 
\end{lemma}
\begin{proof}
First we recall that $\gmn{a}-\inlinefrac{1}{d-1}\log\vert a\vert$ is pluriharmonic,
and is therefore smooth, away from its zero set. From Lemma~\ref{controlGmnOnStableDisks}
we conclude that as long as $\vert a\vert < \delta_r(\epsilon)$
and $\vert a\vert <\bigl((r\delHO)^{1/d}e^{-\epsilon}\bigr)^{d-1}$
then if $(x,y) \in \check{g}_{z+}(S_{r\delHO},a)$ then $(x,y)\in B_z\in V'$
and $\vert v(x,y)\vert =r\delHO$ so 
\[\gmn{a}(x,y)-\inlinefrac{1}{d-1}\log\vert a\vert
> \inlinefrac{1}{d}\log \vert r\delHO\vert - \epsilon -\inlinefrac{1}{d-1}\log\vert a\vert > 0.\]
Thus $\gmn{a}$ is smooth at such points.

Assume that no such $\epsilon_S$ existed. Then there exists a sequence $a_i\to 0$ and
a sequence of points $z_i\in J(p)$ and a sequence of points 
$w_i \in U_{z_i} \times S_{r\delHO}\subset B_{z_i}$ such that
for each $i$ the restriction of $\gmn{a_i}$ to the stable manifold 
$\check{g}_{z_i+}(\disk_{\delHO},a)\subset B_{z_i}$
has gradient zero at the point $\check{g}_{z_i+}(w_i,a_i)$.
Then by compactness we can replace our sequence with a subsequence if necessary
such that both $z_i$ converges to some point $z_\infty\in J(p)$
and $w_i$ converges to some point $w_\infty\in V'$.

Then by Lemma~\ref{stableManMovesHol} and Lemma~\ref{stableMansVaryCont}
we see that 
$\check{g}_{z_i+}(w_i,a_i)\to \check{g}_{z_\infty+}(w_\infty,0)$ and
$w_\infty\in U_{z_\infty}\times S_{r\delHO}\subset V'$
and the gradient of the restriction of $\gmn{0}$ to the stable manifold
in $B_{z_\infty}$ is zero at the point
$\check{g}_{z_\infty+}(w_\infty,0)$ on the curve 
$\check{g}_{z_\infty+}(S_{r\delHO},0)$. 
Since the sequence $\gmn{a_k}\bigl(\check{g}_{z_k+}(\cdot,a_k)\bigr)$
converges locally uniformly to
$\gmn{0}\bigl(\check{g}_{z_\infty+}(\cdot,0)\bigr)$
by Lemma~\ref{stableMansVaryCont}, so derivatives
of $\check{g}_{z_k+}(\cdot,a_k)$ converge locally uniformly
to the derivatives of $\check{g}_{z_\infty+}(\cdot,0)$, 
then because $\gmn{a}$ is smooth on a neighborhood of the image
of $\check{g}_{z_\infty+}(S_{r\delHO},0)$ then
the gradient of $\gmn{a_k}$ projected to the tangent space
of $\Delta_{z_k,a_k}$ at $\check{g}_{z_k+}(w_k,a_k)$ converges
to the gradient of $\gmn{0}$ projected to the tangent space
of $\Delta_{z_k,a_k}$ at $\check{g}_{z_k+}(w_\infty,0)$.
This is a contradiction
since the gradient of the restriction of $\gmn{0}=\inlinefrac{1}{d}\log\vert v\vert$ 
to a vertical line 
does not vanish on the curve $\vert v\vert=r\delHO$.
\end{proof}

\head{The index of the gradient of $\gmn{a}$ about the local stable manifold $\Delta_{z,a}(r)$ is one.}

\begin{lemma}
\name{indexOneAroundLoop}
The index of the gradient of the restriction of $\gmn{a}$ 
to $\check{g}_{z+}(\disk_{\delHO})$ around the curve
$\check{g}_{z+}(S_{r\delHO})$ is one for
all $\vert a\vert < \epsilon_S$.
\end{lemma}
\begin{proof}
The lemma is easily seen to be true for $a=0$ since then
$\check{g}_{z+}(S_{r\delHO})$ is a loop around a single
singularity of $\frac{1}{d}\log\vert p(y)-x\vert$.
By Lemma~\ref{nontangencyCondition},
the index can not change for $\vert a\vert <\epsilon_S$.
\end{proof}

\begin{lemma}
\name{disksLandInDisks}
If $a$ is sufficiently small then for any $z\in J(p)$ and for
any $w\in p^{-1}(z)$ one has
$f_a\bigl(\Delta_{w,a}(r)\bigr) \Subset \Delta_{z,a}(r)$.
Also $f_a\bigl(\Delta_{w_1,a}(r)\bigr)\cap f_a\bigl(\Delta_{w_1,a}(r)\bigr)=\emptyset$
for $w_1$ and $w_2$ distinct points of $p^{-1}(z)$.
\end{lemma}
\begin{proof}
Using Lemma~\ref{controlGmnOnStableDisks} it follows that
if $(x,y)\in\overline{\Delta_{w,a}(r)}$ and
$f_a(x,y)\in \overline{\Delta_{w,a}\setminus \Delta_{w,a}(r)}$
then
$\vert a\vert  \geq \inlinefrac{(r\delHO)^{1-1/d}}{\epsilon^{d+1}}$.
The second statement is an immediate consequence of Lemma~\ref{stableManifoldsDisjoint}.
\end{proof}


\head{The forward image of each filled tube lies near $C(p)$.}

\begin{lemma}
\name{tubesLandNearCurve}
There exists $M>0$ 
such that given $\epsilon > 0$ then for $\vert a\vert < \epsilon/M$ one has
$\vert v\bigl(f_a(x,y)\bigr)\vert <\epsilon$ for each $(x,y)\in \tubeFilled_c$
and for each critical point $c$ of $p(x)$.
\end{lemma}
\begin{proof}
One has $\vert v\bigl(f_a(x,y)\bigr)\vert=\vert v\bigl(p(x)-ay,x\bigr)\vert =\vert ay\vert$.
Choosing $M > 0$ such for each critical point $c$ of $p(x)$ one has $\tubeFilled_c\subset (\disk_M\times \mb{C})$ 
one concludes that taking $\vert a\vert < \inlinefrac{\epsilon}{M}$ is sufficient.
\end{proof}

We will let $V'_r\equiv\{(x,y)\in V' \big\arrowvert \ \vert v(x,y)\vert < r\}$. \name{defOfVPrimeR}

\begin{corollary}
\name{cor-noPointsInJ}
If $a$ is sufficiently small then for any $z\in J(p)$ the set
$\overline{\Delta_{z,a}(r)}\setminus \bigcup_{w\in p^{-1}(z)}f_a\bigl(\Delta_{w,a}(r)\bigr)$
contains no points of $J_a$.
\end{corollary}
\begin{proof}
If $a$ is sufficiently small then $J_a\subset V'_r$ and $\pi_a$ is defined.  
Since $J_a\subset V'_r$ then given any $\nax{z}=(\dotsb,z_{-2},z_{-1},z_0)\in \Nax{J}(p)$
then $\pi_a(\nax{z})\in \Delta_{z,a}(r)$. From Lemma~\ref{stableManifoldsDisjoint}
we know that a point in $J_a$ can lie in $\Delta_{z,a}$ iff the corresponding
history in $\Nax{J}(p)$ ends with the point $z$. Then the result 
is easy since the set removed from $\overline{\Delta_{z,a}(r)}$ contains
all points in $J_a$ corresponding to histories in $\Nax{J}(p)$ which
could end with $z_0$.
\end{proof}

\begin{corollary}
\name{cor-preciseIndex}
There exists $\epsilon > 0$ such that if $\vert a \vert < \epsilon$ then
the index of the vector field $\nabla(\gmn{a}\rest{\Delta_{z,a}})$
around the boundary of  
$\overline{\Delta_{z,a}(r)}\setminus \bigcup_{w\in p^{-1}(z)}f_a\bigl(\Delta_{w,a}(r)\bigr)$
is $1-d$.
\end{corollary}
\begin{proof}
We note from Corollary~\ref{cor-noPointsInJ} that $\gmn{a} > 0$
on $\overline{\Delta_{z,a}(r)}\setminus \bigcup_{w\in p^{-1}(z)}f_a\bigl(\Delta_{w,a}(r)\bigr)$
and so it will be pluriharmonic on a neighborhood of 
$\overline{\Delta_{z,a}(r)}\setminus \bigcup_{w\in p^{-1}(z)}f_a\bigl(\Delta_{w,a}(r)\bigr)$.
Thus $\nabla (\gmn{a}\rest{\Delta_{z,a}})$ will have only
finitely many zeros in $\overline{\Delta_{z,a}(r)}\setminus \bigcup_{w\in p^{-1}(z)}f_a\bigl(\Delta_{w,a}(r)\bigr)$.
By Lemma~\ref{nontangencyCondition} we know none of these zeros lie on the boundary.

The result then follows as
long as $\epsilon$ is sufficiently small as a consequence of
Lemma~\ref{indexOneAroundLoop} and Lemma~\ref{disksLandInDisks}.
\end{proof}

We recall \cite{bedfordsmillie7} Proposition 2.7, noting that the 
hypothesis is satisfied for all $a$ under consideration since $f_a$
is hyperbolic when the crossed mapping construction of {\henII} applies, 
and there is a continuous surjection from $\Nax{J}(p)$
to $J_a$, and hence $J_a$ is connected. Since $J_a$ is connected and $\vert a\vert<1$ 
then by Theorem 0.2 of \cite{bedfordsmillie6} it follows
that $f_a$ is unstably connected. 

\begin{proposition}
\name{prop-plusLaminationExtends}
If $f_a$ is hyperbolic and unstably connected, then
the union of $\Fpl{a}$ and the stable lamination
of $f_a\rest{{J_a}}$ form a lamination of the space $\upl{a}\cup\jpl{a}$.
\end{proposition}

\begin{observation}
For the maps we are studying, the union of $\Fmn{a}$
and the unstable lamination of $f_a\rest{{J_a}}$ 
do not form a lamination of the space $\umn{a}\cup \jmn{a}$.
This is because critical points on the local stable manifolds 
are tangencies between the stable foliation and $\Fmn{a}$. 
Taking forward images of
these tangencies gives accumulations of such tangencies near $J_a$.
But $\jmn{a}$ is transverse to $\jpl{a}$ everywhere since the map is hyperbolic, so
the unstable foliation and $\Fmn{a}$ can't be part of the same foliation. 
\end{observation}

We let $K^\circ(p)$ denote the interior of the filled Julia set of $p(x)$.

\begin{proposition}
\name{prop-captureAllAbutmentPoints}
For all sufficiently small nonzero $a$,
given any $z\in J(p)$ then
the only points of $\overline{\critLoc_a}$ 
which lie in $\bigDiskWithHoles{z}{a}{w}$
are the points $\tau_{0a}(z,c)\in \tau_{0a}(\partial H_{c0})=\partial H_{ca}$
where $c$ is a critical point of $p(z)$.

Additionally, when $a$ and $b$ are sufficiently small
the biholomorphism $\tau_{ab}\colon H_{ca}\to H_{cb}$ defined in Section
\ref{subsection-trappingAndMappingComponents} extends naturally to a homeomorphism
between $\overline{H_{ca}}$ and $\overline{H_{cb}}$. Since $\overline{H_{c0}}$ 
can be naturally identified with $\mb{C}\setminus K^\circ(p)$ then
the same is true for $\overline{H_{ca}}$.
\end{proposition}

\begin{proof}
We choose $\epsilon$ small enough that:
\begin{enumerate}
\item $\epsilon \leq \epsilon_T$, so Lemma~\ref{tubeToolbox} holds for $\vert a\vert <\epsilon$,
\item $J_a\subset V'$ when $\vert a\vert < \epsilon$, which we can do by Lemma~\ref{jMovesCont},
\item $\epsilon < r\delHO/M$ for the value $M$ in Lemma~\ref{tubesLandNearCurve},
\item Lemma~\ref{disksLandInDisks}, Corollary~\ref{cor-noPointsInJ} and Corollary~\ref{cor-preciseIndex}
all hold for $\vert a\vert < \epsilon$.
\end{enumerate}
We will show the result holds for $\vert a\vert < \epsilon$.

\smallhead{To each point $z\in H_{c0}$ we define a holomorphic motion $h_z$ such that $h_z(a)$ stays
in the critical component and the same leaf of $\Fpl{a}$.}

Given a critical point $c$ of $p(x)$, we
let $\tilde{H}_c=\{(x,y,a)\vert (x,y)\in H_{ca}, a\in \disk_\epsilon\}$.
Now given an arbitrary point $z\in \mb{C}\setminus K(p)$ 
we define the map $h_{cz}\colon \disk_\epsilon\to\tilde{H}_c\subset \mb{C}^2\times \disk_\epsilon$ by
$h_{cz}(a)\equiv\bigl(\tau_{0a}(z,c),a\bigr)\in \tilde{H}_c$. 
It follows from Proposition~\ref{holomorphicVariationOnCritical} that $h_{cz}$ is holomorphic.
One easily confirms that $\ppl{a}\bigl(h_{cz}(a)\bigr)
=\ppl{0}(z,c)$ which is independent of $a$. It follows that
\begin{equation}
\name{gplConstant}
\gpl{}\bigl(h_{cz}(a)\bigr)\equiv \gpl{0}(z,c) \text{\ for\ }a\in \disk_\epsilon.
\end{equation}

We also note that by Corollary~\ref{corPplBound} there is some
radius $R'$ such that for each critical point $c$ of $p(x)$, if 
$\vert x\vert \geq R'$, $(x,y)\in \overline{\tubeFilled_c}$ 
and $\vert a\vert <\epsilon$ then
$\gpl{a}(x,y) \geq 1$. Hence if $(x,y)\in H_{ca}$ 
and $\gpl{a}(x,y) < 1$ then 
$(x,y)\in W_c\equiv\Bigl(\overline{\tubeFilled_c} 
\cap (\disk_{R'}\times\mb{C})\Bigr)\times \disk_\epsilon$.
The set $W_c$ is clearly a bounded set in $\mb{C}^3$ since the set $\disk_{R'}\times\mb{C}$
has bounded $x$ coordinates and the set $\overline{\tubeFilled_c}$
has bounded $y$ coordinates.

We let $Y=\{z\in \mb{C}\setminus K(p)\ \vert\ \gpl{0}(z,c) < 1\}$. 
If $z\in Y$ then $h_{cz}(a)$ lands 
in the set $W_c$ since $\gpl{}\bigl(h_{cz}(a)\bigr)=\gpl{0}(z,c)<1$ for all $a\in \disk_\epsilon$.
Since $W_c$ is bounded it follows that 
$\{h_{cz}\vert z\in Y\}$
is a normal family of maps from $\disk_\epsilon$ 
into $\mb{C}^3$. The condition that $\gpl{0}(z,c) <1$
is the same as $\log\vert b_p(z)\vert < 1$ so $J(p)\times\{c\}\times\{0\}\subset \overline{Y}$.

To complete the proof of Proposition~\ref{prop-captureAllAbutmentPoints}
we need two lemmas.

\smallhead{We begin to show that abutments stay in our disks.}

\begin{lemma}
\name{sublemma-insideBigDisk}
Assume $z_k$ is a sequence of points of $\mb{C}\setminus K(p)$ converging to 
a point $z_\infty\in J(p)$.
Then, for each critical point $c$, the limit $g$ 
of any convergent subsequence of $h_{cz_k}$
satisfies $g(a)\in f_a^{-1}\bigl(\Delta_{z_\infty, a}(r)\bigr)$ 
for all $a\in \disk_{\epsilon}$.
\end{lemma}
\begin{proof}[Proof of Lemma~\ref{sublemma-insideBigDisk}]
We will show that $f_a\bigl(g(a)\bigr)\in \Delta_{z_\infty,a}(r)$ for all 
$a\in \disk_\epsilon$.
By Lemma~\ref{tubesLandNearCurve}, since $\vert a\vert< \epsilon < r\delHO/M$, then 
$\vert v\bigl(f_a(\tubeFilled_c)\bigr)\vert < r\delHO$.
Recall that $V'_r=\{(x,y)\in V'\big\arrowvert\ \vert v\vert< r\delHO\}= 
U'\times \disk_{r\delHO}$.
Now $\tilde{\Delta}_{z_\infty,a}(r)\equiv \{(x,y,a)\vert (x,y)\in\Delta_{z_\infty,a}(r),a\in\disk_\epsilon\}$
is defined in $V'_r\times\disk_\epsilon$ as the graph of 
$g_{z_\infty+}\colon \disk_{r\delHO}\times\disk_\epsilon\to U'$.
Thus $v-g_{z_\infty+}(u,a)$ is a holomorphic defining function
for $\tilde{\Delta}_{z_\infty,a}(r)\subset V'_r\times \disk_\epsilon$.
Since $\gpl{a}\bigl(h_{cz_k}(a)\bigr) >0$ for each $k$, then
$\bigl(v-g_{z_\infty}(u,a)\bigr)\circ \tilde{f}\circ h_{cz_k}$ is nonvanishing 
on $\disk_\epsilon$ for each $k$ whenever it is defined, i.e. whenever
$h_{cz_k}\in V'_r\times \disk_\epsilon$.
If $g$ is the limit of any convergent subsequence, then 
$\bigl(v-g_{z_\infty+}(u,a)\bigr)\circ f_0\circ g(0)=0$. 
Since
$g$ maps $\disk_\epsilon$ into
$\{(x,y)\vert\ \vert v(x,y)\vert < r\delHO\}$ it can be shown
that the set on which
$\bigl(v-g_{z_\infty+}(u,a)\bigr)\circ\tilde{f} \circ g$ vanishes
is both open and closed in $\disk_\epsilon$, so it is all
of $\disk_\epsilon$.
Thus $g(a)\in f_a^{-1}\bigl(\Delta_{z_\infty,a}(r)\bigr)$ for all $a\in\disk_\epsilon$.
\end{proof}

\begin{lemma}
\name{sublemma-outsideSmallDisks}
Assume $z_k$ is a sequence of points of $\mb{C}\setminus K(p)$ converging to 
a point $z_\infty\in J(p)$.
Then the limit $g$ of any convergent subsequence of $h_{cz_k}$ is disjoint from $\Delta_{w, a}(r)$
for each $w\in p^{-1}(z_\infty)$ 
for all $a\in \disk_{\epsilon}$ and for each critical point $c$.
\end{lemma}
\begin{proof}[Proof of Lemma~\ref{sublemma-outsideSmallDisks}.]
Each $g(a)\in\tubeFilled_c$ lies a positive
distance from $\vert v\vert < r\delHO$ by Definition~\ref{def-r}.
Since $\Delta_{w,a}(r)\subset V'_r$ this completes the proof.
\end{proof}

Now consider an arbitrary critical point $c_0$ of $p(z)$.
Consider also an arbitrary $a\in \disk_\epsilon$ and a sequence of points 
$\{w_i\}\in H_{c_0a}$ which converge to a point 
$w_\infty\in \overline{H_{c_0a}}\cap (J_{a+}\cup J_{a-})$ .
Then let $(z_i,c_0)=\tau_{a0}(w_i)\in H_{c_00}$ and consider the sequence of maps
$h_{cz_k}\colon \disk_\epsilon\to\mb{C}^3$ for each critical point $c$.

This is a normal family.
Choose some subsequence $h_{cz_{k_i}}$ so that $h_{cz_{k_i}}$
converges for each critical point $c$. For each critical point $c$ let $g_c$ be the limit
of this subsequence. 
By 
Lemma~\ref{sublemma-insideBigDisk},
$g_c(a)\in f_a^{-1}\bigl(\Delta_{z_\infty,a}(r)\bigr)$. 
By Lemma~\ref{tubeToolbox}, $g_c(a)\not\in\jmn{a}$, 
so $\gmn{a}$ is smooth at $g_c(a)$.
Since $g_c(a)$
is a limit of points of $H_{ca_0}$ then 
by Proposition~\ref{prop-plusLaminationExtends},  $\Fmn{a}$
and $\jpl{a}$ are tangent at $g_c(a)$ and so $g_c(a)$
is a point where $\nabla (\gmn{a}\rest{f_{a}^{-1}\bigl(\Delta_{z_\infty,a}(r)\bigr)})$
has nonzero index.

Now if $a\not= 0$ then since there are $d-1$ critical points of $p(x)$ 
then the set $\{g_c(a)\vert c\text{\ a critical point}\}$ is a set of
$d-1$ points of nonzero index in 
$f_{a}^{-1}\bigl(\Delta_{z_\infty,a}(r)\bigr)\setminus\bigcup_{w\in p^{-1}(z_\infty)}\Delta_{w,a}$.

\smallhead{Here we have an abutment for each $H_{ca}$ in 
$f_a^{-1}\bigl(\Delta_{z_\infty,a}(r)\bigr)$ but \it{not} in 
$\Delta_{w,a}(r), w\in p^{-1}(z)$. We conclude each has index $-1$.}

Since the index of the gradient of the restriction of $\gmn{a}$ is a vector
field in
$\Delta_{z_\infty,a}(r)\setminus \bigcup_{w\in p^{-1}(z_\infty)} f_{a}\bigl(\Delta_{w,a}(r)\bigr)$
by Corollary~\ref{cor-noPointsInJ}, and the index around the boundary
of this set 
is $1-d$ by Corollary~\ref{cor-preciseIndex}, 
then the same is clearly true for
$f_{a}^{-1}\bigl(\Delta_{z_\infty,a}(r)\bigr)\setminus \bigcup_{w\in p^{-1}(z_\infty)} \Delta_{w,a}(r)$. 
Thus each of the $d-1$ points $g_c(a)$ has index $-1$
and $g_c(a)$ is the unique point of nonzero index in the intersection of
$f_{a}\bigl(\Delta_{z_\infty,a}(r)\bigr)\setminus \bigcup_{w\in  p^{-1}(z_\infty)} \Delta_{w,a}(r)$
and the tube $\tubeFilled_c$. That the same holds for $a=0$ is easy to verify directly.

It follows that given any critical point $c$ of $p(z)$ then any 
convergent subsequence of $h_{cz_k}$
must converge to $g_c$. It follows that $h_{cz_k}\to g_c$. We denote $g_c$ by 
$h_{cz_\infty}\colon \disk_\epsilon\to\mb{C}^3$.
We have thus shown that, given a critical point $c$, 
if $z_k\in\mb{C}\setminus K(p)$ converges
to $z_\infty\in J(p)$ then 
$h_{cz_k}(a)\equiv\bigl(\tau_{0a}(z,c),a\bigr)\in \tilde{H}_c$
converges to a holomorphic function $g_c$ such that 
$g_c(a)$ is the unique point of nonzero index
in the intersection of $\tubeFilled_c$ and
$\bigDiskWithHoles{z_\infty}{a}{w}$.


We now construct the extension $\tau_{ab}\colon\overline{H_{ca}}\to\overline{H_{cb}}$
by defining $\tau_{ab}(w_\infty)=h_{cz_\infty}(b)$ whenever
$w_\infty\in\overline{H_{ca}}$ and $z_\infty$ is the limit 
of $\{\tau_{a0}(w_k)\}$ where $w_k\to w_\infty$.
This is well defined because if $w_k\to w_\infty$ 
and $w_k'\to w_\infty$ then $w_1,w'_1,w_2,w'_2,\dotsc$
converges to $w_\infty$ and by the above, the sequence of maps
$h_{c\tau_{a0}(w_1)},h_{c\tau_{a0}(w_1')},\dotsc$ converges to a single map $h_{cz_\infty}$
for $\vert a\vert < \epsilon$.
This is also continuous since if $w_k\in \overline{H_{ca}}$ and $w_k$ converges
to $w_\infty\in\overline{H_{ca}}$ but $\tau_{ab}(w_k)\not\to \tau_{ab}(w_\infty)$ then
there exists $\epsilon_1 > 0$ such that there are arbitrarily large 
values of $k$ with $\vert \tau_{ab}(w_k)-\tau_{ab}(w_\infty)\vert > \epsilon_1$.
Then replace each point $w_k$ with a point $w_k'\in H_{ca}$ such that 
$\vert w'_k - w_k\vert < {1}/{2^k}$ and 
$\vert \tau_{ab}(w'_k)-\tau_{ab}(w_k)\vert <\epsilon / 2$ (which we can
do by the definition of $\tau_{ab}(w_k)$ if $w_k\in \overline{H_{ac}}\setminus H_{ac}$
and we just take $w'_k=w_k$ otherwise). Then $w'_k$ is a sequence in $H_{ac}$
and $\tau_{ab}(w'_k)$ can not converge to $\tau_{ab}(w_\infty)$ because
there are arbitrarily large $k$ for which $\vert \tau_{ab}(w'_k)-\tau_{ab}(w_\infty)\vert
> \vert \tau_{ab}(w_k)-\tau_{ab}(w_\infty)\vert -
\vert \tau_{ab}(w'_k)-\tau_{ab}(w_k)\vert > \epsilon/2$ 
but $w'_k\to w_\infty$ since $\vert w'_k-w_k\vert < 1/2^k$.
But this is a contradiction since 
$h_{c\tau_{a0}(w'_k)}\to h_{c\tau_{a0}(w_\infty)}$ by our previous
work and $h_{c\tau_{a0}(w'_k)}(a)=w_{k'}$ and 
$h_{c\tau_{a0}(w'_k)}(b)=\tau_{ab}(w'_k)$ by definition.
Therefore $\tau_{ab}\colon\overline{H_{ca}}\to\overline{H_{cb}}$ is continuous.
Since $\tau_{ba}$ is clearly the inverse of $\tau_{ab}$ then $\tau_{ab}$ is a homeomorphism.
This completes the proof of Proposition~\ref{prop-plusLaminationExtends}.
\end{proof}


\begin{theorem}
\name{thm-classificationOfCriticalComponents}
For all sufficiently small $a$ 
every component of the critical locus is an iterate
of one of the components $H_c$.
\end{theorem}
\begin{proof}
The components of $H_{ca}$ are the only components of the critical locus
if $a=0$, so assume $a\not=0$.
If $W$ is a component of $\critLoc_a$ then
by Lemma~\ref{abutmentForComponents} we know that
$\partial W$ contains at least one point $w$ in either $\jpl{a}$ 
or $\jmn{a}$. If $w$ lies in $\jpl{a}\setminus \jmn{a}$ then $w$ lies in the stable
manifold of some point of $J_a$ so there is some $n'$
such that $f_a\cp{n'}(w)\in \Delta_{z',a}(r)$ for some
$z'\in J(p)$. Take $n$ to be the smallest such $n'$ (where $n'$ is allowed to be negative)
and take $z$ to be the corresponding point of $J(p)$. We know a smallest such $n$ exists
since $w\in \jpl{a}\setminus \jmn{a}$ so $\gmn{a}(w) - \inlinefrac{1}{d-1}\log\vert a\vert > 0$
so $\gmn{a}\bigl( f_a^{-k}(w)\bigr)-\inlinefrac{1}{d-1}\log\vert a\vert=
d^n\cdot \bigl( \gmn{a}(w)-\inlinefrac{1}{d-1}\log\vert a\vert\bigr)\to \infty$ as $k\to \infty$
but $\gmn{a}$ is bounded on compact sets (since
$V_0$ can certainly be assumed to be large enough to contain $V'$). 

It follows from our choice of $n$ that $f_a\cp{n}(w)\in 
\Delta_{z,a}(r)\setminus \bigcup_{y\in p^{-1}(z)} f_a\bigl(\Delta_{y,a}(r)\bigr)$. 
But then from Proposition~\ref{prop-captureAllAbutmentPoints}
and Lemma~\ref{tubeToolbox}
we conclude that $f_a\cp{n}(W)$ is an iterate of some $H_c$.

On the other hand if $w\in \jmn{a}$ then choose a sequence
of points $w_i\in W$ such that $w_i\to w$. Then
$\gpl{a}(w_i)\to \gpl{a}(w)=\mf{g} \geq 0$, and
$\gmn{a}(w_i)-\inlinefrac{1}{d-1}\log\vert a\vert\to \gmn{a}(w)-\inlinefrac{1}{d-1}\log\vert a\vert=0$. 
Then for every
$i\geq 1$ choose $n_i$ such that 
$1 < \gmn{a}\bigl(f_a^{-n_i}(w_i)\bigr)-\inlinefrac{1}{d-1}\log\vert a \vert\leq d$.
By the recursion relation for $\gmn{a}$ it follows that $n_i\to \infty$
as $i\to \infty$.
Then $\gpl{a}\bigl(f_a^{-n_i}(w_i)\bigr) \to 0$ as $ i\to \infty$ so
the sequence $f_a^{-n_i}(w_i)$ is a bounded sequence of points
in iterates of $W\subset\critLoc_a$ converging to $\jpl{a}$. 
By Proposition~\ref{prop-captureAllAbutmentPoints}
and Lemma~\ref{tubeToolbox}
the members of any convergent
subsequence of $f_a^{-n_i}(w_i)$ must lie in an iterate of $H_c$ for all
large $i$, so $W$ must be an iterate of $H_c$. This completes the proof.
\end{proof}

\newpage 
\part{Rigidity}

\sectionDone{Holonomy in the Critical Locus.}

\name{sectionHolonomyInTheCriticalLocus}


We will now consider a single map $f=f_a$, $a\not=0$ to which 
Theorem~\ref{thm-classificationOfCriticalComponents} applies
and for which $\vert a\vert < \epsilon_T$ (so Lemma~\ref{tubeToolbox}
and Theorem~\ref{thm-extensionToH} apply).
Since $a$ is fixed we let $\psi_+\equiv \ppl{a}$ \name{defOfPsi} 
and $\psi_-\equiv \eta\cdot\pmn{a}\circ f^{-1}$ 
where $\eta^{d-1}=\shortInlineFrac{1}{a}$ and $\eta$ is fixed. This gives the simpler relations
$\psi_-\circ f^{-1}=\psi_-^d$ and $\psi_+\circ f=\psi_+^d$. The function $\psi_-$
is well defined and holomorphic on $f(V_-)$. We will 
show that $\psi_-$ is a biholomorphism from a neighborhood of infinity in 
$H_c$ to a neighborhood of infinity in $\mb{C}$.
Because $a$ is fixed we will omit it from the notation throughout the rest of
this section.

\begin{lemma}
\name{psiMnBiholAtInfinity}
There is a neighborhood $H_c^\circ$ of infinity in $H_c$ which lies 
in $f(V_-)$ and such that 
$\psi_-$ is a biholomorphism from $H_c^\circ$ to $\mb{C}\setminus \overline{\disk}_{\pmnHole}$
for some $\pmnHole>1$.
\end{lemma}
\name{defOfHCirc}
\begin{proof}
First we note that 
\[f(V_-)=\{(x,y)\vert \ \vert p(y)-x\vert > \vert a y\vert \ \text{and}\ 
    \vert p(y)-x\vert > \vert a \vert \alpha\}.\]
Since the $y$ coordinate
of points $(x,y)\in H_c$ remain bounded as $\vert x\vert \to \infty$,
it follows that $(x,y)\in H_c$ implies
$(x,y)\in f(V_-)$ whenever $\vert x\vert$ is sufficiently large. 
This implies the first assertion.

To show the second result we consider $H_c$ to be lying
in $\mb{C}^2\subset \mb{P}^1\times\mb{P}^1$. We then
know that $H_c$ can be completed to become a disk by adding the point
$(\infty,c)$ and that its tangent space is given at $(\infty,c)$ by
$p''(c)\ed y + C \ed u=0$ where $C$ is some constant
and $u=\inlinefrac{1}{x}$. 
From Corollary~\ref{corPmnBound} we obtain
\[B^{-1} < \Bigl\arrowvert \inlinefrac{\pmnPlain(x,y)}{y}\Bigr\arrowvert < B\]
for $(x,y)\in V_-$.
Now if $(x,y)\in f(V_-)$ then 
$f^{-1}(x,y)=\Bigl(y,\inlinefrac{p(y)-x}{a}\Bigr)\in V_-$ and so we obtain
$ B^{-1}\Bigl\arrowvert \inlinefracp{p(y)-x}{a}\Bigr\arrowvert 
< \arrowvert \pmnPlain\bigl(f^{-1}(x,y)\bigr)\arrowvert 
< B\Bigl\arrowvert \inlinefracp{p(y)-x}{a}\Bigr\arrowvert $.
Since $\pmnPlain\circ f^{-1}=\psi_-/\eta$ this becomes
\[B^{-1}\Bigl\arrowvert \eta\inlinefracp{p(y)-x}{a}\Bigr\arrowvert 
< \arrowvert \psi_-(x,y)\arrowvert 
< B\Bigl\arrowvert \eta\inlinefracp{p(y)-x}{a}\Bigr\arrowvert\] 
for $(x,y)\in f(V_-)$.
From Corollary~\ref{corPplBound} we have
$B^{-1}\vert x\vert < \vert \psi_+(x,y)\vert < B \vert x\vert$ for $(x,y)\in V_+$.
Dividing one equality by the other gives
\[B^{-2} \Bigl\arrowvert \eta \inlinefracp{p(y)-x}{ax}\Bigr\arrowvert < 
\Bigl\arrowvert \inlinefrac{\psi_-(x,y)}{\psi_+(x,y)}\Bigr\arrowvert <
 B^{2} \Bigl\arrowvert \eta \inlinefracp{p(y)-x}{ax}\Bigr\arrowvert.\]
Since $\inlinefracp{p(y)-x}{x}$ is bounded on $H_c$ as $\vert x\vert \to\infty$,
by the Riemann extension theorem 
$\inlinefrac{\psi_-(x,y)}{\psi_+(x,y)}$ is holomorphic and nonzero 
on a neighborhood of $\infty\in H_c$. Hence
by shrinking $H_c^\circ$ if necessary we conclude
that $\psi_-(x,y)$ is a biholomorphism from the neighborhood $H_c^\circ$ 
of $\infty$ onto $\mb{C}\setminus\overline{\disk_\pmnHole}$ for some large $\pmnHole$.
\end{proof}

The following is shown in Proposition 6.2 of {\henII}.

\begin{lemma}
\name{leafStructureInVplAndVmn}
There exists $\mf{R} > 1$ such that for any $z\in \mb{C}$
with  $\vert z\vert > \mf{R}$ 
the fiber of $\psi_+$ in $V_+$ over $z$ 
and the fiber of $\psi_-$ in $f(V_-)$ over $z$ are each analytic disks.
\end{lemma}

\head{We define $\mathscr{V}_+$ and $\mathscr{V}_-$.}

\begin{definition}
Noting that $\vert \psi_+\vert$ and $\vert\psi_-\vert$ are well defined on
all of $\upl{}$ and $\umn{}$ respectively,
we let 
\[\mathscr{V}_+\equiv \{z\in V_+\vert \ \vert\psi_+(z)\vert > \mf{R}\}\]
and we let 
\[\mathscr{V}_-\equiv \{z\in V_-\vert \ \vert\psi_-(z)\vert > \mf{R}\}.\]
\end{definition}
\name{defOfScriptV}

We note that $f(\mathscr{V}_+)\subset \mathscr{V}_+$ and 
$f^{-1}(\mathscr{V}_-)\subset \mathscr{V}_-$
as follows from the recursion relations for $\psi_+$ and $\psi_-$.

\head{We define $\ms{S}$.}

The following is a consequence of the definitions:
\begin{lemma}
\name{sameLeafFirstCritereon}
Two points $z_1$ and $z_2$ are on the same leaf of $\Fpl{}$ iff
there exists $n\geq 0$ such that $f\cp{n}(z_1), f\cp{n}(z_2)\in \ms{V}_+$
and $\psi_+\bigl(f\cp{n}(z_1)\bigr)=\psi_+\bigl(f\cp{n}(z_2)\bigr)$.
Similarly, two points $z_1$ and $z_2$ are on the same leaf of $\Fmn{}$ iff
there exists $n\geq 0$ such that $f^{-n}(z_1), f^{-n}(z_2)\in \ms{V}_-$
and $\psi_-\bigl(f^{-n}(z_1)\bigr)=\psi_-\bigl(f^{-n}(z_2)\bigr)$.
\end{lemma}

Let $\ms{S}\equiv \{\omega\in\mb{C}\vert \omega^{d^n}=1\ \text{for some}\ n\geq 0\}\cong \mb{Q}_d/\mb{Z}.$
\name{defOfDiadicS}
Given two points $z_1,z_2\in\upl{}$,
it is clear that the   property
$\inlinefrac{\psi_+(z_1)}{\psi_+(z_2)}\in\ms{S}$
is independent of the branches of $\psi_+$ used.
Similarly for the property
$\inlinefrac{\psi_-(z_1)}{\psi_-(z_2)}\in\ms{S}$,
$z_1,z_2\in \umn{}$.

\begin{lemma}
\name{sameLeafSecondCritereon}
Two points $z_1,z_2\in \upl{}$ are on the same leaf of $\Fpl{}$ iff
$\inlinefrac{\psi_+(z_1)}{\psi_+(z_2)}\in\ms{S}$.
Similarly, two points $z_1,z_2\in \umn{}$ are on the same leaf of $\Fmn{}$ iff
$\inlinefrac{\psi_-(z_1)}{\psi_-(z_2)}\in\ms{S}$.
\end{lemma}
\begin{proof}
If $z_1$ and $z_2$ lie in $\ms{V}_+$ the first assertion follows from Lemmas
\ref{leafStructureInVplAndVmn} and \ref{sameLeafFirstCritereon} and
the recursion relationship for $\psi_+$.
Otherwise, choosing $k$ such that
$f\cp{k}(z_1),f\cp{k}(z_2)\in \ms{V}_+$ gives the first assertion.
The second is analogous.
\end{proof}

We will want to consider the holonomy maps of $H_c$ determined by
the foliations $\Fpl{}$ and $\Fmn{}$. By Theorem~\ref{thm-extensionToH}, $H_c$ can be identified
with $\mb{C}\setminus\overline{\disk}$ using $\psi_+$.

Assume that $z_1$ and $z_2$ are points of $H_c$ for some critical
point $c$ of $p(z)$ and that $z_1$ and $z_2$ 
lie on the same leaf $\lfpl{}(z_1)$ of $\Fpl{}$. 
Start with $z=z_1$ and then vary $z\in H_c$.
By Lemma~\ref{tubeToolbox},
the leaves of $\Fpl{}$ all intersect $H_c$ transversely
hence the intersection of $\lfpl{}(z)$ with $H_c$ which is 
near $z_2$ will vary holomorphically with $z$.
By this means we get a holomorphic map $h$
from a neighborhood $N_1\subset H_c$ of $z_1$ to 
a neighborhood $N_2\subset H_c$
of $z_2$. Since the inverse map is given by starting with $z=z_2$, varying 
$z\in H_c$ and following the intersection of
$H_c$ and $\lfpl{}(z)$ near $z_1$ then the holomorphic map
from $N_1$ to $N_2$ is a biholomorphism for suitably
chosen $N_1$ and $N_2$. Let $\mf{m}_+\colon N_1\to N_2$ \name{defOfMathFracM}
denote this biholomorphism.

\begin{lemma}
In the coordinates on $H_c$ given by $\psi_+\colon H_c\to\mb{C}\setminus\overline{\disk}$
the map $\mf{m}_+$ is given by $\mf{m}_+(z)=\omega z$ for some $\omega\in\ms{S}$.
If we continue $\mf{m}_+$ by holonomy then $\mf{m}_+$ extends
to the global automorphism of $H_c$ given by $\mf{m}_+(z)=\omega z$.
The same result holds on $H_c^\circ$ with coordinates
defined by $\psi_-$, and using any two points $z_1,z_2\in H_c^\circ$
on the same leaf of $\Fmn{}$, along with $\Fmn{}$, to construct
a holonomy map $\mf{m}_-$.
\end{lemma}
\begin{proof}
The function $\psi_+\colon H_c\to\mb{C}$ is well defined 
and so $\inlinefrac{\psi_+\bigl(\mf{m}_+(z)\bigr)}{\psi_+(z)}$ is a continuous
fuction of $z$ which takes values in $\ms{S}$. Hence it is constant. 
Thus the map $\mf{m}_+(z)$, which was only defined in a neighborhood
of a point on $H_c$, takes the form 
$\mf{m}_+(z)=\omega z$ in the coordinates defined
by $\psi_+$, and thus such a holonomy map gives a 
global automorphism of $H_c$.

The result for $H_c^\circ$ is proven the same way.
\end{proof}

One can attempt to picture this holonomy map in terms of monodromy. Assume $d=2$.
Because the Jacobian of $f$ is very small, 
the set $f(H_c)$ looks approximately like the curve $C(p)$. 
Consequently, if $z\in\mb{C}$ is sufficiently large then there will be exactly 
two points in $f(H_c)$ which map to $z$ under $\psi_+$.
The monodromy map $\mf{m}\colon z\to -z$, 
carried out on $f(H_c)$ instead of on $H_c$, interchanges such pairs. 
One visualizes a monodromy of a given order $2^n$ by looking at $f\cp{n}(H_c)\cap V_+$
and considering the fibers of $\pplPlain$ in $V_+$ as we move in a large loop around
$K_+$ along $H_c$.

  \newpage 
  
\section{Fiber Preserving Conjugacies.}
\name{section-conjugacies}

\subsectionDone{Statement of the Theorem}

Buzzard and Verma \cite{buzzard-verma} proved a stability result by using the $\lambda$-lemma
along the leaves of the foliation $\Fpl{}$. 
We are interested in deforming the underlying manifold to obtain a new holomorphic self map. 
Unfortunately, simply deforming the leaves of $\Fpl{}$ quasiconformally does not yield a well defined
complex manifold structure as it destroys the complex structure transverse to the foliation.
The first natural approach is to deform using both $\Fpl{}$ and $\Fmn{}$.
However, we will show that typically H\'enon maps can
not be deformed in this way.

In what follows we consider H\'enon maps $f(x,y)=\bigl(p(x)-ay,x\bigr)$ satisfying the following:

\begin{condition}
\name{hypothesis}
{\hspace{2in}}\par
\begin{itemize}
\item $p$ is hyperbolic with connected Julia set and simple critical points.
\item The Jacobian $a$ is so small that the hypothesis of 
Theorem~\ref{thm-extensionToH}, 
Theorem~\ref{thm-classificationOfCriticalComponents}, and
Lemma~\ref{psiMnBiholAtInfinity} all hold.
\end{itemize}
\end{condition}

\begin{convention}
We will be dealing with just two H\'enon maps, $f$ and $g$, 
in this section instead of a whole family $f_a$. 
Hence we will omit the subscript $a$, but will
use a subscript of $f$ or $g$ whenever necessary, e.g. $\Fpl{f}$ and $\Fpl{g}$
instead of $\Fpl{a}$, or $\pmnHole_f$ and $\pmnHole_g$ instead of $\pmnHole$.
\end{convention}

Assume that we are given two different H\'enon maps $f$ and $g$ arising from two such 
polynomials $p_f$ and $p_g$ and that $f$ and $g$ satisfy Condition~\ref{hypothesis}. 
We will show that there are severe obstructions to the existence
of a conjugacy between $f$ and $g$ on
$\upl{}\cup \umn{}$ which maps leaves of $\Fpl{f}$ and $\Fmn{f}$ to
the leaves of $\Fpl{g}$ and $\Fmn{g}$ respectively.

The folling result can be found as
Lemma 2.1 of \cite{buzzard-nondensityOfStability}.
We include a proof.

\begin{lemma}
\name{conjugacyPreservesTangency}
A homeomorphism from $\upl{f}\cup \umn{f}$ to $\upl{g}\cup \umn{g}$ which
maps the leaves of $\Fpl{f}$ and $\Fmn{f}$ to the leaves of $\Fpl{g}$
and $\Fmn{g}$ respectively necessarily maps
$\critLoc_f$ to $\critLoc_g$.
\end{lemma}
\begin{proof}
This is because two leaves $\leaf_1$ and $\leaf_2$ which are transverse in $\mb{C}^2$ 
intersect in a different manner topologically than two leaves which
are not.  
To see this, choose convenient $(x,y)$ coordinates so that
the point of intersection is the origin, $\leaf_1$
coincides with the $x$ axis, and $\leaf_2$ is transverse 
to the $y$ axis. Then choose a biholomorphic parameterization
$t\mapsto \bigl(g_1(t),g_2(t)\bigr)$ of a neighborhood of the 
intersection in $\leaf_2$ such that $t=0$ maps to the point of 
intersection.

Since the second leaf is transverse to the $y$ axis then $g_1'(0)\not=0$.
Therefore $g_1$ is a local biholomorphism. It follows that we can
write $g_2=\theta g_1^d$ 
for some nonvanishing holomorphic function $\theta$ defined in some
neighborhood of zero, where $d=1$ iff $\leaf_1$ and $\leaf_2$
intersect transversely. Choosing $\zeta$ a 
holomorphic function such that $\zeta^{d-1}=\theta$,
then $\zeta g_2=(\zeta g_1)^d$ so $\leaf_2$ is parameterized by
$\bigl(\zeta g_1,(\zeta g_2)^d\bigr)$. 
One can choose local coordinates 
so that that $\leaf_2$ is parameterized by $(t,t^d)$ and
$\leaf_1$ is the $x$ axis. Then if
$U$ is any sufficiently small open neighborhood of the origin we see that the
inclusion $\leaf_2\setminus\{0\}\hookrightarrow U\setminus \leaf_1$
can induce a surjective map of fundamental groups iff $d=1$.
\end{proof}

\begin{remark}
\name{rem-reductionToSimplerCaseByIterating}
We will assume that $h\colon \upl{f}\cup \umn{f}\to \upl{g}\cup \umn{g}$
is a homeomorphism which
maps leaves of $\Fpl{f}$ to leaves of $\Fpl{g}$ and maps leaves of
$\Fmn{f}$ to leaves of $\Fmn{g}$. 
Then by Lemma~\ref{conjugacyPreservesTangency}
$h$ maps the critical locus $\critLoc_f$ of $f$ to the critical locus
$\critLoc_g$ of $g$. We assume that there are critical points
$c_f$ and $c_g$ of $f$ and $g$ such that if we let $H_f$ and $H_g$
denote the components of  $\critLoc_f$ and $\critLoc_g$ asymptotic
to $y=c_f$ and $y=c_g$ respectively as $\vert x\vert\to\infty$
then $h$ maps $H_f$ to $H_g$. There is no loss of generality
in assuming that $h$ maps $H_f$ to $H_g$ since, 
by Theorem~\ref{thm-classificationOfCriticalComponents},
we can always choose integers $k$ and $\ell$
such that this assumption holds if we change coordinates by iterates of $f$ and $g$
so that in the new coordinates $h$ becomes the map
$g\cp{k}\circ h\circ f\cp{\ell}$. 
\end{remark}

There is one degenerate situation we wish to rule out. Since 
we are considering conjugacies 
$h\colon\upl{f}\cup \umn{f}\to \upl{g}\cup \umn{g}$ without
requiring that $h$ extends to $\mb{C}^2$ then
it is possible that $h\colon H_f\to H_g$ ``inverts'' $H_f$, meaning 
that it maps neighborhoods of the puncture in $H_f\cong \disk^*$
to open sets adjacent to the other boundary component of $\disk^*$.

\begin{condition}
\name{noInversion}
The map $h$ maps small neighborhoods of the 
puncture of $H_f$ at infinity to neighborhoods
of the puncture of $H_g$, where $H_f$ and $H_g$
are specified horizontal components of $\critLoc_f$
and $\critLoc_g$ respectively.
\end{condition}

\begin{observation}
Let $\conjMap\colon\mb{C}^2\to\mb{C}^2$ \name{defOfconjMap}
be the conjugation map $\conjMap(x,y)=(\conj{x},\conj{y})$,
and let $g=\conjMap\circ f\circ \conjMap$.
Then $\psi_{g+}=\conj{\psi_{f+}\circ \conjMap}$ and 
$\psi_{g-}=\conj{\psi_{f-}\circ \conjMap}$. Thus
as a conjugacy between $f$ and $g$, $\conjMap$
maps the leaves of $\Fpl{f}$ and $\Fmn{f}$ to the leaves of 
$\Fpl{g}$ and $\Fmn{g}$
respectively. 
\end{observation}

\begin{reduction}
\name{red-reductionToSimplerCaseByConjugation}
We can assume that $h\colon H_f\to H_g$ is orientation preserving
by replacing $g$ with $\conjMap\circ g\circ \conjMap$ and $h$ with $\conjMap\circ h$
if necessary.
\end{reduction}

\begin{lemma}
\name{hRotatesCircles}
In the coordinates on $H_f$ and $H_g$ given
by $\psi_{f+}$ and $\psi_{g+}$
one has $h(\omega z)=\omega h(z)$ for any $\omega\in S^1$.
The same property also holds for all $z\in H_f^\circ$ using the coordinates on
$H^\circ_f$ and $H^\circ_g$ given by 
$\psi_{f-}$ and $\psi_{g-}$, 
\end{lemma}
\begin{proof}
We will conduct the proof for $H_f$ and $H_g$
with coordinates 
$\psi_{f+}\colon H_f\to\mb{C}\setminus\overline{D}$
and $\psi_{g_+}\colon H_g\to\mb{C}\setminus\overline{D}$
given by Theorem~\ref{thm-extensionToH}.
The proof for the other part of the lemma is the same using
the coordinates given in Lemma~\ref{psiMnBiholAtInfinity}.
Given $\omega \in \ms{S}$ the function
$\inlinefrac{h(\omega z)}{h(z)}$ only takes values in $\ms{S}$ so
it is constant.
Thus, given $\omega\in \ms{S}$ there exists $\theta_\omega \in \ms{S}$ such
that $h(\omega z)=\theta_\omega h(z)$ for all $z\in H_f$. 
It follows that $h$ is $\ms{S}$-equivariant. Hence it maps
circles (centered at the origin) to circles.
Since $h$ is a homeomorphism then the order of the points 
$z,\omega z, \omega^2z,\dotsc, \omega^{d^n}z=z$ on the circle
must be preserved or reversed.
Thus 
$h(\omega z)=\omega h(z)$ or 
$h(\omega z)=\conj{\omega} h(z)$ for $\omega\in\ms{S}$ and hence 
for all $\omega\in S^1$. The latter possiblity can be 
eliminated since $h$ is orientation preserving.
\end{proof}

%

We will write $\chi_{f+}$, $\chi_{f-}$, $\chi_{g+}$ 
and $\chi_{g-}$ for the maps $1/\psi_{f+}$, $1/\psi_{f-}$, $1/\psi_{g+}$ and
$1/\psi_{g-}$. The point at infinity is a removable singularity
for each of these maps, each of which sends this point to the origin.
Hence $\chi_{f+}\colon H_f\to \disk$, $\chi_{g+}\colon H_g\to \disk$,
$\chi_{f-}\colon H_f^\circ \to \disk_{1/\pmnHole_f}$ and 
$\chi_{g-}\colon H_g^\circ \to \disk_{1/\pmnHole_g}$ are all biholomorphisms.

\begin{definition}
\name{def-hAndK} 
We define $\mf{h}\colon \disk\to \disk$ by 
$\mf{h}=\chi_{g+}\circ h\circ \chi_{f+}^{-1}$
and we similarly define $\mf{k}\colon \disk_{1/\pmnHole_f}\to \disk_{1/\pmnHole_g}$ by
$\mf{k}=\chi_{g-}\circ h\circ \chi_{f-}^{-1}(z)$. 
These maps satisfy $\mf{h}(\omega z)=\omega^{\pm 1}\mf{h}(z)$
and $\mf{k}(\omega z)=\omega^{\pm 1}\mf{k}(z)$, where the $\pm$ has the same sign
in both relationships.
\end{definition}
\name{defOfMathFrach}

\begin{definition}
\name{def-sigma}
Let $\sigma_g=\chi_{g-}\circ\chi_{g+}^{-1}\colon
\chi_{g+}(H_g^\circ)\to \disk_{1/\pmnHole_g}$ and let 
$\sigma_f=\chi_{f-}\circ \chi_{f+}^{-1}\colon
\chi_{f+}(H_f^\circ)\to \disk_{1/\pmnHole_f}$.
The maps $\sigma_g$ and $\sigma_f$ are 
the biholomorphic transition maps between
the coordinate systems in which $\mf{h}$ and $\mf{k}$ represent the map $h$.
\end{definition}
\name{defOfSigma} 

It is easy to confirm that
\begin{equation}
\name{conjugacyRelationship}
\sigma_g\circ\mf{h}=\mf{k}\circ \sigma_f. 
\end{equation}

Moreover, both
$\mf{k}\circ \sigma_f$ and $\sigma_g\circ \mf{h}$ are 
defined on $\Omega_f\equiv\chi_{f+}(H_f^\circ)$
and map this homeomorphically onto $\disk_{1/\pmnHole_g}$. 

We thus have the following commutative diagram:

\begin{equation}
\name{commutativeDiagram}
\xymatrix{
\D_{1/ \tau_f} \ar[r]^{\mf{k}}   
    &   {\disk_{1/\pmnHole_g}} \\  
 H_f^\circ \ar[u]_{\chi_{f-}} \ar[r]^h \ar[d]^{\chi_{f+}}  
    &    H_g^\circ \ar[u]^{\chi_{g-}} \ar[d]_{\chi_{g+}}  \\
 \Omega_f \ar[r]_{\mf{h}} \ar@(ul,dl)[uu]^{\sigma_f}  
    &    \Omega_g \ar@(ur,dr)[uu]_{\sigma_g}} 
\end{equation}

\begin{proposition}
\name{prop-conjugacyInvariant}
One of the following must hold:
\begin{enumerate}
\item \name{prop-conjugacyInvariant-SigmaOption} $\sigma_f(z)=\beta z$ and $\sigma_g(z)=\gamma z$ for constants $\beta,\gamma\in \mb{C}^*$
\par\hspace{-12pt} or
\item \name{prop-conjugacyInvariant-HAndKOption} there is a neighborhood of the origin about which
$\mf{h}(z)=\beta z$ and $\mf{k}(z)=\gamma z$ for constants $\beta,\gamma\in \mb{C}^*$.  
\end{enumerate}
\end{proposition}


In the next section we will derive Proposition~\ref{prop-conjugacyInvariant} from
general properties of holomorphic circle actions. 

\comm{
\begin{assumption}
\name{asForProof}
We assume, for the sake of conducting the proof, 
that conclusion~\ref{prop-conjugacyInvariant-SigmaOption}
of Proposition~\ref{prop-conjugacyInvariant} does not hold for the given maps.
It is easy to see that this is equivalent to the assumption
that neither $\sigma_f$ nor $\sigma_g$ is of the form $z\mapsto \beta z$
for a constant $\beta \in\mb{C}^*$. 
\end{assumption}

The proof of Proposition~\ref{prop-conjugacyInvariant} will take a couple of
sections.
%
Every use of the
word ``circle'' during the proof of Proposition~\ref{prop-conjugacyInvariant}
will refer to a circle centered about the origin.
We denote the circle of radius $r$ by $S_r$.
}

\subsection{Holomorphic circle actions}

Let us consider the circle $\R/\Z$ acting faithfully on a neighborhood $U$ of $0$ by biholomorphic maps
fixing $0$, i.e., we have a monomomorphism $t\mapsto \gamma^t$ from $\R/\Z$ to the group of biholomorphic
maps of $U$ such that $\gamma^t(0)=0$. We call it briefly a {\it holomorphic circle action}. 
The action $\gamma^t: z\mapsto e^{2\pi i t} z$ will be called {\it standard}.

\msk\nin
{\bf Remark 1 (straightening).} {\it Any holomorphic circle action is conformally conjugate to the standard one
(where  the conjugacy can be orientation reversing)}. Indeed, take some orbit $\Gamma$ of the action.
It is a topological circle that bounds a topological disk $V\ni 0 $. Uniformize $V$ by the round disk,
$h: V\ra \D$. Then the maps $\tl \gamma^t = h\circ \gamma^t \circ h^{-1}$ are holomorphic automorphisms of $\D$ fixing $0$,
so they are  rotations.
Thus, we obtain a monomomorphism $\R/\Z\ra \T$, where $\T$ is the group of rotations.
There are only two such monomomorpisms, $t\mapsto e^{\pm 2\pi i t}$, and they are conjugate by the
reflection $z\mapsto \bar z$.  
\msk

The orbits of any holomorphic circle action form an analytic foliation in the punctured neighborhood $U^*$ of $0$.

\msk\nin
{\bf Remark 2 (tangencies)}
  {\it Given a pair $(\gamma^t, \rho^s)$  of holomorphic circle actions, the associated foliations either coincide 
(and in this case the actions either coincide or conjugate by a conformal reflection%
\footnote{We will refer to such a pair as {\it trivial}.})
or have finitely many  tangencies on any given orbit.} 
Indeed, if the set of tangencies between the two foliations
is not finite on some orbit, then they have a common leaf (since they are analytic), 
and hence can be simultaneously straightened.

Given two circle actions $\gamma^t$ and $\tl \gamma^t$, 
a local homeomorphism $h$ near $0$ is called $(\gamma, \tl \gamma)$-{\it equivariant} 
if $h(\gamma^t z) = \tl\gamma^t h(z)$ or $h(\gamma^t z) = \tl\gamma^{-t} h(z)$.

\begin{proposition}
   Let $(\gamma^t, \rho^s)$ and $(\tl\gamma^t, \tl\rho^s)$  be two non-trivial pairs of holomorphic circle actions.
If a local homeomorphism $h$ is both $(\gamma, \tl \gamma)$- and $(\rho, \tl\rho)$-equivariant,
then $h$ is holomorphic.
\end{proposition}

\begin{proof}
  Without loss of generality, we can assume that $h$ conjugates $\gamma^t$ to $\tl \gamma^t$
(otherwise, compose $h$ with an appropriate conformal reflection). 

Let us take some orbit $\Gamma$ of the $\gamma$-action. By the above Remark 2,
we can pick a point $z\in \Gamma$  such that $\Gamma$ is transverse to the $\rho$-orbit at $z$,
and the orbits of $\tl\gamma$ and $\tl\rho^s$ are transverse at $\tl z = h(z)$.
Then the map $(t,s)\mapsto (\rho^s (z), \gamma^t(\rho^s z))$ gives a smooth local chart
near $\Gamma$. If $h$ conjugates $\rho^s$ to $\tl\rho^s$, 
let us consider the similar local  chart near $\tl\Gamma = h(\Gamma)$. 
Otherwise, let us consider the chart $(t,s)\mapsto (\tl\rho^{-s} (z), \tl\gamma^t(\rho^{-s} z))$.
In either case, the map $h$ becomes the identity in these coordinates. Hence $h$ is smooth near $\Gamma$,
and thus, it is smooth in a punctured neighborhood of $0$. 

Let us now show that the joint $(\gamma, \rho)$-action is transitive on the punctured neighborhood $U^*$ of $0$,
i.e., any two points in $U^*$ can be conneced by a concatenation of pieces of
the $\gamma$- and $\rho$-orbits. Indeed, the orbits of the joint action are open since the
domain of  any local chart
described above is contained in one orbit. Since $U^*$ is connected, it must be a single orbit. 

Let us now consider the conformal structure $\mu=h^*(\la)$ in $U$, where $\la$ is the standard conformal structure.
The structure $\mu$ is represented by a smooth family of infinitesimal ellipses in $U^*$.
Moreover, since  $\la$ is invariant under $\tl\gamma$ and $\tl\rho$ and $h$ is equivariant,
 $\mu$ is invaraint under the joint ($\gamma,\de)$-action. 

If the structure $\mu$ is standard then $h$ is holomorphic.
Otherwise, there is a non-circular  ellipse $\mu(z)$. 
Since $\mu$ is invariant under the transitive  $(\gamma, \rho)$-action, 
all the ellipses are non-circular on $U^*$.
Hence the big axes $l(z)\subset T_zU^*$ of the ellipses are well defined on $U^*$ and
form a $(\gamma, \rho)$-invariant line field over $U^*$. 

Let us take some $\rho$-orbit $\De_0$ and 
consider the outermost $\gamma$-orbit $\Gamma_0$ crossing $\De$.
Then $\Gamma_0$ and $\De_0$ have  tangency of even order at some point 
$z_0$.
Rotating the line field by an appropriate angle, we can make $l(z_0)$ tangent at $z_0$ to both
$\Gamma_0$ and $\De_0$. By invariance, $l(z)$ is then tangent to $\De_0$ at any point $z\in \De_0$. 

Let us now consider a nearby $\gamma$-orbit $\Gamma$ which is closer to the origin than $\Gamma_0$.
Then $\Gamma$ intersects $\De_0$ transversally at two points $z_+$ and $z_-$ near $z_0$.
Moreover, the angles $\alpha_{\pm}\in (-\pi/2, \pi/2)$ of these intersections 
have opposite signs.  
This contradicts to the invariance of the line field under the $\gamma$-action.

\end{proof}

\begin{corollary}
  Under the circumstances of the above Proposition,
if the $\gamma$- and $\tl\gamma$-actions are standard, then $h$ is linear, $z\mapsto \la z$.  
\end{corollary}
   
\begin{proof}
In this case, $h$ preserves the foliation by round circles centered at $0$.
But a biholomorphic map that fixes $0$ and maps a circle centered at $0$ to another such a circle
is linear. 
\end{proof}

\nin
{\it Proof of Proposition \ref{prop-conjugacyInvariant}}. 
  We know that the maps $\mf{h}$ and $\mf{k}$ are equivariant with respect to the standard circle action $\gamma^t$.
  Let $\rho^s=  \si_f^{-1}\circ \gamma^s\circ \si_f$ and  $\tl \rho^s=  \si_g^{-1}\circ \gamma^s \circ \si_g$.
  By (\ref{conjugacyRelationship}), $\mf{h}$ is  $(\rho, \tl\rho)$-equivariant.  
  If the maps $\si_f$ and $\si_g$ are not linear, then the pairs of actions, 
  $(\gamma, \rho)$ and $(\gamma, \tl\rho)$, are non-trivial.
  Then $\mf{h}$ is linear by the last Corollary.

  The same argument applies to $\mf{k}$. 
\QED

\subsection{Rigidity Results}

Here we translate the statement of Proposition~\ref{prop-conjugacyInvariant}
back into statements about the two given maps $f$ and $g$ which have a conjugacy $h$ between them which maps 
leaves of $\Fpl{f}$ to leaves of $\Fpl{g}$ and maps leaves of
$\Fmn{f}$ to leaves of $\Fmn{g}$. 

\begin{theorem}
\name{thm-generalRigidity}
Assume we are given two different H\'enon maps $f$ and $g$ satisfying
Condition~\ref{hypothesis}. Assume
$h\colon \upl{f}\cup\umn{f}\to \upl{g}\cup \umn{g}$ is a conjugacy between
$f$ and $g$ such that $h$ maps the leaves of $\Fpl{f}$ and $\Fmn{f}$
 to the leaves of $\Fpl{g}$ and $\Fmn{g}$ respectively.
Choose coordinates for the map $g\colon \mb{C}^2\to \mb{C}^2$
so that $h$ maps $H_f$ to $H_g$ for a pair of primary horizontal critical 
components $H_f$ and $H_g$ of $f$ and $g$.
Finally assume that $h\colon H_f\to H_g$ is orientation preserving
and satisfies Condition~\ref{noInversion}.
Then $h\colon \upl{f}\cap\umn{f}\to \upl{g}\cap\umn{g}$ is 
a biholomorphism. Also, 
$\psi_{g+}\circ h=\inlinefrac{1}{\beta}\psi_{f+}$ 
and $\psi_{g-}\circ h=\inlinefrac{1}{\gamma}\psi_{f-}$ 
on a neighborhood about infinity of $H_f$ where $\beta^{d-1}$ and
$\gamma^{d-1}$ must lie in $\ms{S}$.
\end{theorem}

\begin{corollary}
\name{cor-generalReversingRigidity}
Assume we are given two different H\'enon maps $f$ and $g$ satisfying
Condition~\ref{hypothesis}. Assume
$h\colon \upl{f}\cup\umn{f}\to \upl{g}\cup \umn{g}$ is a conjugacy between
$f$ and $g$ such that $h$ maps the leaves of $\Fpl{f}$ and $\Fmn{f}$
to the leaves of $\Fpl{g}$ and $\Fmn{g}$ respectively.
Choose coordinates for the map $g\colon \mb{C}^2\to \mb{C}^2$
so that $h$ maps $H_f$ to $H_g$ for a pair of primary horizontal critical 
components $H_f$ and $H_g$ of $f$ and $g$.
Finally assume that $h\colon H_f\to H_g$ is orientation reversing
and satisfies Condition~\ref{noInversion}.
Then $\conjMap\circ h\colon \upl{f}\cap\umn{f}\to \conjMap(\upl{g}\cap\umn{g})$ is 
a biholomorphism. Also, 
$\conj{\psi_{g+}\circ h}=\inlinefrac{1}{\beta}\psi_{f+}$ 
and $\conj{\psi_{g-}\circ h}=\inlinefrac{1}{\gamma}\psi_{f-}$ 
on a neighborhood about infinity of $H_f$ where $\beta^{d-1}$ and
$\gamma^{d-1}$ must lie in $\ms{S}$.
\end{corollary}
\begin{proof}
Replace $g$ with $\conjMap\circ g\circ \conjMap$
and $h$ with $\conjMap\circ h$. 
\end{proof}

\begin{proof}[Proof of Theorem.]
Applying
Lemma~\ref{conjugacyPreservesTangency} we see the critical
locus of $f$ maps by $h$ to the critical locus of $g$.
As in Remark~\ref{rem-reductionToSimplerCaseByIterating}
we can change coordinates using iterates of $f$ and $g$
so there are primary horizontal components $H_f$ and $H_g$ such that 
$h$ maps $H_f$ to $H_g$. Then Proposition~\ref{prop-conjugacyInvariant} gives the following cases:

\paragraph{Case 1. $\sigma_f(z)=\beta z$ and $\sigma_g(z)=\gamma z$.}

If we write these two equations out using Definition~\ref{def-sigma}
we obtain
$\psi_{f-}=\inlinefrac{1}{\beta}\psi_{f+}$ and 
$\psi_{g-}=\inlinefrac{1}{\gamma}\psi_{g+}$.
Now while $\psi_{f+}\colon \upl{f}\to\mb{C}^*$
is not well defined, the map to the quotient
group $\psi_{f+}\colon \upl{f}\to\mb{C}^*/\ms{S}$
is well defined. Similarly for 
$\psi_{f-}\colon \umn{f}\to\mb{C}^*/\ms{S}$.

Now there exists some $M>0$ such that if 
$z\in H_{f}$ and $\gpl{f}(z)>M$ then 
$\psi_{f-}(z)=\inlinefrac{1}{\beta}\psi_{f+}(z)$. 

Now if $z\in \upl{f}$ and $\gpl{f}(z) > M$ then
there is some leaf $\ell\in\Fpl{f}$ containing
both $z$ and some point $w\in H_{f}$.
Thus $\gpl{w} > M$ so
$\psi_{f-}(z)\sim \psi_{f-}(w)=
\inlinefrac{1}{\beta}\psi_{f+}(w)\sim\inlinefrac{1}{\beta}\psi_{f+}(z)$

But then if $z\in\upl{f}$ and $\gpl{z}>M$ then 
\begin{equation}
\inlinefrac{1}{\beta}\psi_{f+}(z)\sim\psi_{f-}(z)\sim 
\psi_{f-}^{d}\circ f(z)\sim\inlinefrac{1}{\beta^d}\psi_{f+}^d\circ f(z)= 
\inlinefrac{1}{\beta^d}\psi_{f+}^{d^2}(z).
\end{equation}
Hence $\beta^{d-1}\sim \psi_{f+}^{d^2-1}(z)$. However since
$\psi_{f+}^{d^2-1}(z)$ has a locally continuous branch about $z$
one concludes that this branch would have to be constant, and
hence $\psi_{f+}(z)$ must be constant. This is a contradiction.
Thus no Henon map $f$ exists for which
$\psi_{f-}=\inlinefrac{1}{\beta}\psi_{f+}$ 
for a neighborhood of infinity in $H_{f}$.

\paragraph{Case 2. $\mf{h}(z)=\beta z$ and $\mf{k}(z)=\gamma z$.}

If we write these two equations out using Definition~\ref{def-hAndK}
then $\psi_{g+}\circ h=\inlinefrac{1}{\beta}\psi_{f+}$ 
and $\psi_{g-}\circ h=\inlinefrac{1}{\gamma}\psi_{f-}$.
These hold in a neighborhood of infinity in $H_{f}$.
Now, as in Case~1, there exists some $M>0$ such that if $z\in\upl{f}$ 
and $\gpl{f}(z) > M$ then 
$\psi_{g+}\circ h(z)\sim \inlinefrac{1}{\beta}\psi_{f+}(z)$ in $\mb{C}^*/\ms{S}$.
Then 
\begin{multline}
\inlinefrac{1}{\beta}\psi_{f+}^d(z)=\inlinefrac{1}{\beta}\psi_{f+}\circ f(z)
\sim\psi_{g+}\circ h\circ f(z)= \\
\psi_{g+}\circ g\circ h(z)\sim\psi_{g+}^d\circ h(z)
\sim \inlinefrac{1}{\beta^d}\psi_{f+}^d(z)
\end{multline}
from which it follows that $\beta^{d-1}\in \ms{S}$. 

We will show that $h\colon \upl{f}\cap\umn{f}\to \upl{g}\cap \umn{g}$
is a biholomorphism. To do this is suffices to show that 
$h\colon\upl{f}\cap\umn{f}\to \upl{g}\cap \umn{g}$
is holomorphic. To accomplish this is will suffice to 
show that if $z\in \upl{f}\cap\umn{f}$ then $h$
is holomorphic on $\lfpl{f}(z)$ and on $\lfmn{f}(z)$.
From this it will immediately follow from Osgood's
theorem that $h$ is holomorphic on 
$(\upl{f}\cap \umn{f})\setminus\critLoc$.
Since $h\colon \upl{f}\cap \umn{f}\to\upl{g}\cap \umn{g}$
is continuous then by the Riemann extension theorem
it will further follow that 
$h\colon\upl{f}\cap\umn{f}\to \upl{g}\cap \umn{g}$
is holomorphic. In order to prove this we will need the following.

\begin{lemma}
Given $z_0\in (\upl{f}\cap\umn{f})\setminus\critLoc_f$,
assume that $\lfpl{f}(z_0)$ meets $H_f$ at a point $w_0$.
Then there is a neighborhood $U$ of $z_0$ in $\lfmn{f}(z_0)$
such that there is a holomorphic holonomy map $\zeta_{f,z_0,w_0}\colon U\to H_f$
which maps $z\in U$ to the intersection of $\lfpl{f}(z)$ and $H_f$
near $w_0$. Since $\Fpl{f}$ is transverse to $\lfmn{f}(z_0)$
and $H_f$ at $z_0$ and $w_0$ respectively then 
$\zeta_{f,z_0,w_0}$ can be assumed to be a biholomorphism
from $U$ onto its image.
\end{lemma}
\begin{proof}
This is geometrically self evident.
\end{proof}


\begin{figure}
\name{holonomyPicture}
\centerline{\includegraphics{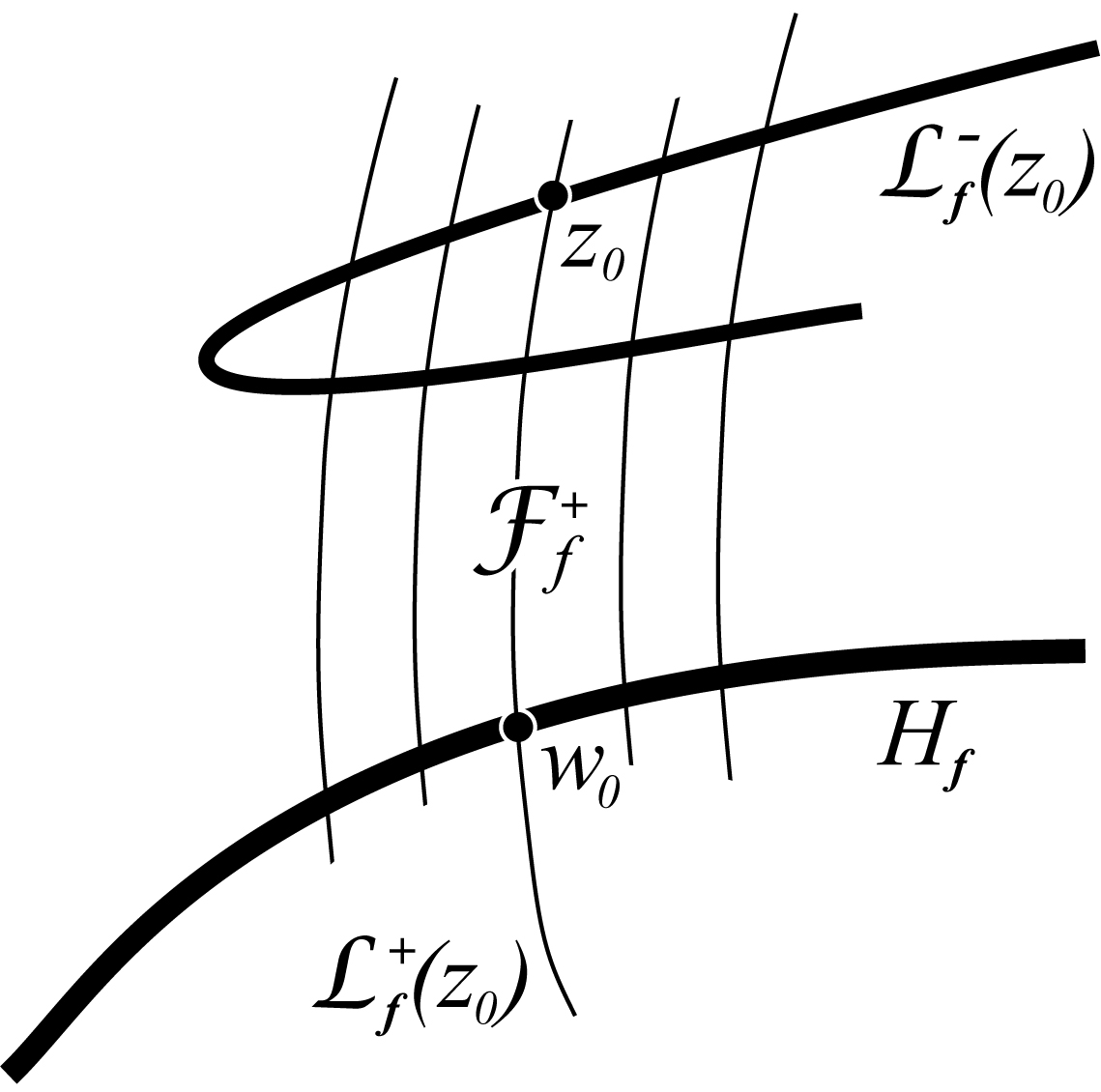}}
\caption{Holonomy using $\Fpl{f}$.}
\end{figure}

Now assume $z_0\in(\upl{f}\cap\umn{f})\setminus\critLoc_f$
and $\gpl{f}(z_0) > M$ and $\gpl{g}(h(z_0)) > M$. Choose 
$w_0\in \lfpl{f}(z_0)\cap H_f$.  Then if $U$ is 
a sufficiently small open neighborhood of $z_0$ in $\lfmn{f}(z_0)$ 
then one can see that
\begin{equation}
\name{eq-hMadeHol}
h\rest{U}=\zeta^{-1}_{g,h(z_0),h(w_0)}\circ h\circ \zeta_{f,z_0,w_0}
\end{equation}
because $h$ maps $H_f$ to $H_g$, maps $\lfmn{f}(z_0)$
to $\lfmn{g}(h(z_0))$ and maps leaves of $\Fpl{f}$
to leaves of $\Fpl{g}$. But because
$h$ in the right hand side of equation \eqref{eq-hMadeHol} is
applied on $H_f$ near infinity as long as $U$ is sufficiently small,
then it is holomorphic. Thus equation \eqref{eq-hMadeHol}
represents $h\rest{U}$ as a composition of three holomorphic functions.
Hence $h\rest{U}$ is holomorphic.

Now in general, if $z_0$ is any point in 
$(\upl{f}\cap\umn{f})\setminus\critLoc{f}$ then 
choosing $n$ sufficiently large that $\gpl{f}(f\cp{n}(z_0))>M$
and $\gpl{g}(h(f\cp{n}(z_0)))>M$ and writing
$h=g^{-n}\circ h\circ f\cp{n}$ one concludes
that $h$ is holomorphic on $\lfmn{f}(z_0)$ near $z_0$.

The proof that $h$ is holomorphic on leaves of $\Fmn{f}$
on $\upl{f}\cap \umn{f}$ is identical.
This completes the proof.
\end{proof}

Eliminating $h$ from the conclusion 
of Theorem~\ref{thm-generalRigidity} we obtain
%
%
%
\begin{equation}
\name{eq-specialConj}
\gamma\psi_{g-}\circ\psi_{g+}^{-1}(z)=\psi_{f-}\circ\psi_{f+}^{-1}(\beta z)
\end{equation}
for some nonzero constants $\beta$ and $\gamma$.

We now assume that $f$ has degree two, so
$f$ and $g$ are H\'enon maps of the form
$f(x,y)=(x^2+c_1-a_1y,x)$ and $g(x,y)=(x^2+c_2-a_2y,x)$.
Rewriting equation \eqref{eq-specialConj} in terms of 
$\chi_{f+}, \chi_{f-}, \chi_{g+}$ and $\chi_{g-}$ we obtain
\begin{equation}
\name{eq-specialConjEq}
\chi_{g-}\circ \chi_{g+}^{-1}(\beta z)=
\gamma\cdot \chi_{f-}\circ\chi_{f+}^{-1}(z).
\end{equation}
Our first goal will be to find all possible (nondegenerate) 
choices of $f(x,y)=(x^2+c_1-a_1y,x)$, $g(x,y)=(x^2+c_2-a_2y,x)$
and $\beta$ and $\gamma$ nonzero constants so that
equation \eqref{eq-specialConjEq} holds.

The first three nonzero terms of the Taylor series expansion of 
\begin{equation}
\name{eq-diffEq}
\chi_{g-}\circ \chi_{g+}^{-1}(\beta z)-
\gamma\cdot \chi_{f-}\circ\chi_{f+}^{-1}(z)
\end{equation}
are
\begin{multline}
(\gamma a_1^2-\beta a_2^2)z + (\gamma a_1^2c_1-\beta^2a_2^2 c_2)z^2 +\\
(\gamma a_1^2c_1^2-\beta^3a_2^2c_2^2+\frac{1}{2}\gamma a_1^2c_1 -
\frac{1}{2}\gamma a_1^4 c_1 + \frac{1}{2}\beta^3a_2^4c_2-\frac{1}{2}\beta^3a_2^2c_2)z^3
\end{multline}

Since we can assume that each of $a_1,a_2,\beta,\gamma$ are nonzero,
it easily follows from the first two terms that
\begin{equation}
\name{eq-partialSol}
\gamma=\dfrac{a_2^2}{a_1^2}\beta,\quad
c_1=c_2\beta.
\end{equation}

\begin{lemma}
If $\chi_{g-}\circ \chi_{g+}^{-1}(z)=\chi_{f-}\circ\chi_{f+}^{-1}(z)$
for nonsingular
$$
f(x,y)=\bigl(x^2+c_1-a_1y,x\bigr)
           \quad {\mathrm{and}}\quad
g(x,y)=\bigl(x^2+c_2-a_2y,x\bigr)
$$
then $f=g$.
\end{lemma}
\begin{proof}
From \eqref{eq-partialSol} it follows that
$a_1=\pm a_2$ and $c_1=c_2$. To eliminate the remaining case
assume $a_1=-a_2$. Now \eqref{eq-diffEq}
must vanish, but the coefficients of 
$z^4$, $z^5$ and $z^7$ in \eqref{eq-diffEq}
generate the ideal $(a_1^3)$, so $a_1$ would have to be zero.
\end{proof}

From the third term  of \eqref{eq-diffEq}
one obtains that either $a_2=1$, $a_2=-1$, $c_1=0$,
or $\beta=\dfrac{a_1^2-1}{a_2^2-1}$.

Each of the cases $a_2=1$, $a_2=-1$, and 
$\beta=\dfrac{a_1^2-1}{a_2^2-1}$
can be reduced to the case $c_1=0$ by taking more terms of the Taylor
series of \eqref{eq-diffEq} and the calculating 
Groebner basis of the resulting coefficients. In the case $c_1=0$
it can be shown that $c_2=0$ and that $a_1=a_2$ (so $f$ and $g$ 
are the same H\'enon map), and that $\beta=\gamma$, both of 
which must be the same primative root of unity.

The following table gives the necessary details to 
verify these calulations using \textregistered{Maple}.
Here $n$ is the degree to which we need to calculate the Taylor series for each calculation. 

\begin{center}
\begin{tabular}{ | c | c | l | p{1.05in} | p{.9in} |}
\hline
{\bf{Case}} & {\bf{n}} & {\bf{Ordering}} & {\bf{Relevant Basis Elements}} & {\bf{Conclusion}} \\
\hline
$\beta=\dfrac{a_1^2-1}{a_2^2-1}$ & $7$ & $\operatorname{tdeg}(a_1,c_1,a_2)$ 
& $-a_2^2c_1(a_1^2-1)\cdot$\newline$(a_2^2-1)^5(a_1-a_2)$ & $c_1=0$ \\
\hline 
$a_2=1$ & $8$ & $\operatorname{plex}(\beta,c_1,a_1)$ 
& $\beta c_1(a_1-1)$,\newline$\beta c_1(\beta-1)$ & $c_1=0$  \\
\hline
$a_2=-1$ & $8$ & $\operatorname{plex}(\beta,c_1,a_1)$ 
& $\beta c_1(a_1+1)$,\newline $\beta c_1(\beta-1)$ & $c_1=0$ \\ 
\hline
$c_1=0$ & $13$ & $\operatorname{tdeg}(\beta,a_2,a_1)$ 
& $\beta a_1 a_2^2(a_1-a_2)$,\newline$\beta a_1 a_2^2 (s^3-1)$ 
& $\beta^2+\beta+1=0$\newline $\gamma=\beta$ \newline$a_1=a_2$ \newline$c_1=c_2=0$ \\
\hline
\end{tabular}
\end{center}

We conclude that:
 
\begin{lemma}
\name{problemSimplified}
If equation~\eqref{eq-specialConj} holds with $\beta$ and $\gamma$
nonzero and $f(x,y)=(x^2+c_1-a_1y,x)$ and $g(x,y)=(x^2+c_2-a_2y,x)$ 
nondegenerate then 
the maps $f$ and $g$ must be the same Henon map and either 
(1) $\beta=\gamma=1$ (the trivial solution), or else (2) $a_1=a_2$, $c_1=c_2=0$,
$\beta=\gamma$ and $\beta$ is a primitive cubic root of unity.
\end{lemma}

\begin{theorem}
\name{thm-quadraticRigidity}
Assume $h$ is a conjugacy between quadratic H\'enon maps
$f(x,y)=(x^2+c_1-a_1y,x)$ and $g(x,y)=(x^2+c_2-a_2y,x)$ 
satisfying Condition~\ref{hypothesis}.
Assume further that $h$ maps the leaves of $\Fpl{f}$ and $\Fmn{f}$
to the leaves of $\Fpl{g}$ and $\Fmn{g}$ respectively.
Choose coordinates for the map $g\colon \mb{C}^2\to \mb{C}^2$
so that $h$ maps $H_f$ to $H_g$ for the primary horizontal critical 
components $H_f$ and $H_g$ of $f$ and $g$.
Finally assume that $h\colon H_f\to H_g$ is orientation preserving
and satisfies Condition~\ref{noInversion}.
Then $f=g$ and $h\cp{2}(x,y)$ is the identity map on $\upl{f}\cap\umn{f}$.
\end{theorem}
\begin{proof}
This is easy now, as from Theorem~\ref{thm-generalRigidity}
$\beta\in \ms{S}$. By Lemma~\ref{problemSimplified} the maps
$f$ and $g$ must be the same Henon map, and additionally
$\beta=\gamma=1$ since primitive cubic roots of unity do not
belong to $\ms{S}$. 

To see that conclusion about $h\cp{2}$ one notes that
since $\beta=\gamma=1$ then $h\rest{H_f}$ is the identity map.
If we define $\theta(x,y)=(\psi_{f+}(x,y),\psi_{g+}(x,y))$
then about any sufficiently large point $z_0\in H_f$ and 
the Jacobian of $\theta$ is a defining function for
$H_f$ about $z_0$. Since $H_f$ has multiplicity one then
$\theta$ is locally a two to one map about $H_f$.
Now $\theta(h(z))=\theta(z)$ since $\beta=\gamma=1$.
One concludes that about such a point $z_0$ (which
is fixed by $h$) $h$ must either exchange points in the
fibers of $\theta$ or must leave them fixed. Either way,
$h\cp{2}$ must fix all points in a neighborhood of $z_0$.
Since $h\cp{2}$ is the identity map on an open set, it must
be the identity map everywhere.
\end{proof}

\begin{corollary}
\name{cor-quadraticReversingRigidity}
Assume $h$ is a conjugacy between quadratic H\'enon maps
$f(x,y)=(x^2+c_1-a_1y,x)$ and $g(x,y)=(x^2+c_2-a_2y,x)$ 
satisfying Condition~\ref{hypothesis}.
Assume further that $h$ maps the leaves of $\Fpl{f}$ and $\Fmn{f}$
to the leaves of $\Fpl{g}$ and $\Fmn{g}$ respectively.
Choose coordinates for the map $g\colon \mb{C}^2\to \mb{C}^2$
so that $h$ maps $H_f$ to $H_g$ for the primary horizontal critical 
components $H_f$ and $H_g$ of $f$ and $g$.
Finally assume that $h\colon H_f\to H_g$ is orientation reversing
and satisfies Condition~\ref{noInversion}.
Then $f=\mf{c}\circ g\circ \mf{c}$ and 
$\mf{c}\circ h\circ \mf{c}\circ h$ is the identity map on $\upl{f}\cap\umn{f}$.
\end{corollary}

\newpage

\appendix
\section*{List of Notations}
\addcontentsline{toc}{part}{List of Notations}

We provide a list of notations here as a reference.

{\small

\newcommand{\widthValueForTable}{2.75in}

\newcommand{\pack}[1]{\parbox{\widthValueForTable}{\vspace{2pt} #1 \vspace{1pt}}}

\begin{tabular}{ c | c | p{\widthValueForTable} }
\bf{Notatation} & \bf{Section} & \bf{Meaning} \\
\hline
$f_a$
 & \ref{defOff}
 & \pack{The H\'enon map under consideration.}\\
\hline
$a$
 & \ref{defOfa}
 & \pack{The Jacobian of $f_a$.} \\
\hline
$p(x)$
 & \ref{defOfp}
 & \pack{A monic polynomial used to define $f_a$.} \\
\hline
$p(x,y)$
 & \ref{defOfHomogeneousP}
 & \pack{The homogeneous version of $p(x)$.} \\
\hline
$d$
 & \ref{defOfd}
 & \pack{The degree of $p(x)$.} \\
\hline
$q(x)$
 & \ref{defOfq}
 & \pack{The polynomial $p(x)$ with its leading term removed.} \\
\hline
$\degq$
 & \ref{defOfDegq}
 & \pack{The degree of $q(x)$.} \\
\hline
$R$
 & \ref{defOfR}
 & \pack{ A large radius 
  used to define 
  $V_+$ and $V_-$.} \\
\hline
 $V_+, V_-$
 & \ref{defOfV}
 & \pack{Regions 
  which describe the large scale behaviour of H\'enon maps.} \\
\hline
$\kpl{a},\kmn{a}$
 &
 & \pack{The set of points whose orbit under forward (respectively backward) iteration under $f_a$ remain bounded.} \\
\hline
$\jpl{a},\jmn{a}$
 &
 & \pack{The boundaries of $\kpl{a}$ and $\kmn{a}$ respectively.} \\
\hline
$J_a$
 &
 & \pack{The intersection of $\jpl{a}$ and $\jmn{a}$.} \\
\hline
$\upl{a},\umn{a}$ 
 & \ref{defOfU} 
 & \pack{The set of points whose orbit under forward (respectively backward) iteration under $f_a$ remain bounded.} \\
\hline
$\bigUPl,\bigUMn$ 
 & \ref{defOfU} 
 & \pack{These are subsets of $\mb{C}^2\times \disk_R$ whose restriction to $\mb{C}^2\times \{a\}$
  is $\upl{a}$ and $\umn{a}$ respectively.} \\
\hline
$C(p)$
 &
 & \pack{The curve $x=p(y)$ in $\mb{C}^2$. This equal to $\jmn{0}$.} \\
\hline
$\disk_r$ 
 &
 & \pack{The disk of radius $r$ in $\mb{C}$.} \\
\hline
$\hat{V}_-, \hat{V}_+$
 & \ref{defOfHatVMn},\ref{defOfHatVPl}
 & \pack{The regions $V_+$ and $V_-$ extended to $\mb{P}^1\times \mb{P}^1$.} \\
\hline
$\ppl{a},\pmn{a}$
 & \ref{defOfPplAndPmn}
 & \pack{Holomorphic functions defined on $V_+$ and $V_-$ respectively
  measuring that rate of escape to infinity under forward and backward iteration respectively.} \\
\hline
$\gpl{a},\gmn{a}$
 & \ref{defOfG}
 & \pack{ Plurisubharmonic functions on $\mb{C}^2$ measuring the rate of
  escape to infinity under forward and backward iteration respectively.} \\
\hline
$\smn{k},\spl{k}$
 & \ref{defOfSmn},\ref{defOfSpl}
 & \pack{Auxilliary functions used to construct and study $\ppl{a}$ and $\pmn{a}$.} \\
\hline
$x_n,y_n$
 & \ref{defOfXAndY}
 & \pack{These are given by $(x_n,y_n)\equiv f_a\cp{n}(x,y)$, $n\in \mb{Z}$.} \\
\hline
$ \critLoc_a$
 & \ref{defOfCritLoc}
 & \pack{This is the critical locus of $f_a$. It tangency locus of the foliations $\Fpl{a}$ and $\Fmn{a}$.} \\
\hline
$\mf{V}_-$
 & \ref{defOfMfVMn}
 & \pack{This set is the forward image of $V_-\times \disk_R$ under the map 
  $(x,y,a)\mapsto \bigl(f_a(x,y),a\bigr)$.} \\
\hline
\end{tabular}

\newpage
\begin{tabular}{ c | c | p{\widthValueForTable} }
\bf{Notatation} & \bf{Section} & \bf{Meaning} \\
\hline
$U$ 
 & \ref{defOfHobObU}
 & \pack{A neighborhood of $J(p)$ used to construct telescopes.} \\
\hline
$U'$
 & \ref{defOfHobObUPrime}
 & \pack{$U'\equiv p^{-1}(U)$.} \\
\hline
$\delHO$
 & \ref{defOfDelHO}
 & \pack{A small value used to define a neighborhood of $C(p)$.} \\
\hline
$v(x,y)$
 & \ref{defOfv}
 & \pack{$v(x,y)\equiv p(y)-x$} \\
\hline
$V',V'_r$
 & \ref{defOfVPrime},\ref{defOfVPrimeR}
 &  \pack{$V'=\pr_1^{-1}(U)\cap v^{-1}(\disk_{\delHO})$ and \\ $V'_r=\{(x,y)\in V' \big\arrowvert \ \vert v(x,y)\vert < r\delHO\}$} \\
\hline
$\rHubOb$
 & \ref{defOfHubObr}
 & \pack{The radius of telescoping neighborhoods in the Kobayashi metric on $U$} \\
\hline
$\beta$
 & \ref{defOfBeta}
 & \pack{The disk of Euclidean radius $\beta$ about any point $u\in U'$ is mapped biholomorphically by $p$.}\\
\hline
$u(x,y)$
 & \ref{defOfu}
 & \pack{A function on $V'$. The functions $u,v$ give useful coordinates on $V'$.} \\
\hline
$A$
 & \ref{defOfA}
 & \pack{A value choose small enough that the constructions in {\henII} hold when $0<\vert a\vert < A$.}\\
\hline
$U_z$
 & \ref{defOfUz}
 & \pack{The ball of radius $\rHubOb$ about $z\in J(p)$, measured in the Kobashi metric on $U$.} \\
\hline
$\Nax{J}(p)$
 & \ref{defOfNaxJ}
 & \pack{The natural extension of $p\colon J(p)\to J(p)$.} \\
\hline
$\pi_a$
 & \ref{defOfPiA}
 & \pack{A homeomorphism from the natural extension $\Nax{J}(p)$ to $J_a$ for $a$ sufficiently small.} \\
\hline
$g_{z,\pm},\check{g}_{z,\pm}$
 & \ref{defOfg},\ref{defOfgCheck}
 & \pack{Holomorphic functions, the graph of which gives specific local stable/unstable manifolds of the point in question.} \\ 
\hline
$\Delta_{z,a},\Delta_{z,a}(r)$
 & \ref{defOfLocalStableMan}
 & \pack{$\Delta_{z,a}$ is the local stable manifold in $B_z$ for $f_a$. $\Delta_{z,a}(r)$ is a smaller disk
within $\Delta_{z,a}$ (not necessarily about a point of $J_a$).} \\
\hline
$H_{ca}$
 & \ref{defOfH}
 & \pack{The component of the critical locus of $f_a$ asymptotic to $y=c$.} \\
\hline
$\psi_+,\psi_-$
 & \ref{defOfPsi}
 & \pack{Same as $\ppl{a}$, $\pmn{a}$ but with $\pmn{a}$ rescaled to give the simplest iteration formula.} \\
\hline
$\mathscr{V}_+,\mathscr{V}_-$
 & \ref{defOfScriptV}
 & \pack{Subsets of $V_+$ and $V_-$ for which the fibers of $\psi_+$ and $\psi_-$ are disks.} \\
\hline
$\tilde{f}$
 & \ref{defOfFTilde}
 & \pack{ $\tilde{f}(x,y,a)\equiv \bigl(f_a(x,y),a\bigr)$.} \\
\hline
$\ms{S}$
 & \ref{defOfDiadicS}
 & \pack{The set of all roots of unity $\omega$ satisfying $\omega^{d^k}=1$ for some $k\in\mb{N}$.} \\
\hline
\end{tabular}

\newpage
\begin{tabular}{ c | c | p{\widthValueForTable} }
\bf{Notatation} & \bf{Section} & \bf{Meaning} \\
\hline
$\mf{m}+$, $\mf{m}_-$
 & \ref{defOfMathFracM}
 & \pack{The monodromy map of the critical locus determined by either $\Fpl{}$ or $\Fmn{}$, along 
       with a pair of points on the critical locus on the same leaf of $\Fpl{}$ or $\Fmn{}$.} \\
\hline
$H_c^\circ$, $\pmnHole$
 & \ref{defOfHCirc}
 & \pack{$H_c^\circ$ is a neighborhood of infinity in $H_c$ which is mapped biholomorphically
   to $\mb{C}\setminus\disk_{\pmnHole}$.} \\
\hline
$\mf{c}$
 & \ref{defOfconjMap}
 & \pack{The map $\mf{c}(x,y)=(\conj{x},\conj{y})$.} \\
\hline
$\mf{h}$,$\mf{k}$ 
 & \ref{defOfMathFrach}
 & \pack{There are the restriction of a conjugacy $h$ between two Henon maps to 
a component of the critical locus. $\mf{h}$ and $\mf{k}$ are the same conjugacy
written in different coordinate systems.} \\
\hline
$\sigma$
 & \ref{defOfSigma}
 & \pack{Change of complex coordinate functions. In each case the origin is a fixed point.} \\
\hline
$S_r$
 & 
 & \pack{The circle of radius $r$ about the origin in $\mb{C}$.} \\
\hline

\end{tabular}
}


\bibliography{refer}
\bibliographystyle{alpha}

\end{document}